\documentclass{amsart}
\usepackage{a4}
\usepackage{graphics}
\usepackage{color}
\usepackage{amssymb}
\usepackage[all]{xy}
\usepackage{amsmath}
\usepackage{stmaryrd}
\usepackage{longtable}
\usepackage{booktabs}

\usepackage{tikz}

\usetikzlibrary{patterns,snakes}

\CompileMatrices

\newtheorem{theorem}{Theorem}[section]
\newtheorem{lemma}[theorem]{Lemma}
\newtheorem{proposition}[theorem]{Proposition}
\newtheorem{corollary}[theorem]{Corollary}

\newtheorem{case}{Case~(I)}

\newtheorem*{case2}{Case~(II)}
\theoremstyle{definition}
\newtheorem{notation}[theorem]{Notation}
\newtheorem{definition}[theorem]{Definition}

\newtheorem{construction}[theorem]{Construction}

\newtheorem{innervartype}{No.}
\newenvironment{vartype}[1]
  {\renewcommand\theinnervartype{~#1}\innervartype}
  {\endinnervartype}

\newtheorem{remark}[theorem]{Remark}
\theoremstyle{remark}

\numberwithin{equation}{section}
\def\quot{/\!\!/}

\def\wh#1{\widehat{#1}}

\def\ol#1{\overline{#1}}
\def\bangle#1{\langle #1 \rangle}
\def\sei{\mathrel{\mathop:}=}

\def\CC{{\mathbb C}}
\def\KK{{\mathbb K}}

\def\ZZ{{\mathbb Z}}

\def\QQ{{\mathbb Q}}
\def\PP{{\mathbb P}}

\def\Eff{{\rm Eff}}
\def\Mov{{\rm Mov}}

\def\Ample{{\rm Ample}}
\def\SAmple{{\rm SAmple}}
\def\rlv{{\rm rlv}}

\def\trop{{\rm trop}}
\def\ff{{\mathfrak{F}}}

\def\Cl{\operatorname{Cl}}

\def\Spec{{\rm Spec}}

\def\cone{{\rm cone}}
\def\lin{{\rm lin}}

\def\rk{{\rm rk}\,}

\def\cox{{\mathcal R}}
\def\tm{{\tau^-}}
\def\tp{{\tau^+}}
\def\tx{{\tau_X}}

\subjclass[2010]{14J45, 14L30}

\begin{document}
\title[Smooth projective $T$-varieties of complexity~1 with $\rho(X) = 2$]%
{Smooth projective varieties with a torus action of complexity~1 and Picard number~2}
\author[A.~Fahrner]{Anne Fahrner} 
\address{Mathematisches Institut, Universit\"at T\"ubingen,
Auf der Morgenstelle 10, 72076 T\"ubingen, Germany}
\email{fahrner@math.uni-tuebingen.de}
\author[J.~Hausen]{J\"urgen Hausen}
\address{Mathematisches Institut, Universit\"at T\"ubingen,
Auf der Morgenstelle 10, 72076 T\"ubingen, Germany}
\email{hausen@math.uni-tuebingen.de}
\author[M.~Nicolussi]{Michele Nicolussi} 
\address{Mathematisches Institut, Universit\"at T\"ubingen,
Auf der Morgenstelle 10, 72076 T\"ubingen, Germany}
\email{michele.nicolussi@uni-tuebingen.de}

\begin{abstract}
We give an explicit description of all
smooth varieties with a torus action of 
complexity one having Picard number at most two.
As a consequence, we classify in every dimension 
the smooth (almost) Fano varieties with a torus action of 
complexity one having Picard number two.
It turns out that all the Fano examples are obtained 
via an iterated generalized cone construction
from a series of smooth varieties of dimension at
most seven.
\end{abstract}

\maketitle


\section{Introduction}

A basic intention of this article is to 
contribute to the classification of 
smooth (almost) Fano varieties with torus action.
Most studied in this context
are the toric Fano varieties; based on 
their description in terms of lattice 
polytopes, there are meanwhile 
classification results up to dimension 
nine~\cite{Ba:1982,WaWa,Ba:1999,KrNi,Obro,Paff9}.
We go one step beyond the toric case 
and focus on rational varieties with a 
torus action of complexity one, i.e., 
the general torus orbit is 
of dimension one less than the variety;
see~\cite{Su} for results on smooth 
Fano threefolds with an action of a 
two-dimensional torus.

Instead of bounding the dimension, we 
look here at varieties of small Picard 
number. 
Recall that for toric varieties,
the projectives spaces are the only smooth 
examples of Picard number one,  
and we have Kleinschmidt's 
description~\cite{Kl} of all smooth toric 
varieties of Picard number two, which 
in particular allows to figure out the 
(almost) Fano ones in this setting.
We follow that line and study first 
arbitrary smooth projective rational 
varieties with a torus action of 
complexity one. 
The case of Picard number one is 
basically settled by a result of 
Liendo and S\"u\ss~\cite[Thm.~6.5]{LiSu}:
the only non-toric examples are
the smooth projective quadrics in dimensions 
three and four.
Picard number two means to provide an analogue
of Kleinschmidt's description for complexity 
one.

Our approach goes via the Cox ring and we 
use the methods developed 
in~\cite{HaSu:2010,HaHe:2013,ArDeHaLa};
the ground field $\KK$ is algebraically 
closed and of characteristic zero.
Recall that the Cox ring is graded by the 
divisor class group and, together with 
the choice of an ample class, it fixes 
our variety up to isomorphism;
we refer to~\cite{ArDeHaLa} for the 
basic background.
Here comes the first result.

\begin{theorem}
\label{thm:main1}
Every smooth rational projective non-toric 
variety of Picard number two that admits a 
torus action of complexity one 
is isomorphic to precisely one of the following 
varieties $X$, specified by their Cox ring 
$\cox(X)$ and an ample class $u \in \Cl(X)$,
where we always have $\Cl(X) = \ZZ^2$ and the grading 
is fixed by the matrix $[w_1, \ldots ,w_r]$
of generator degrees $\deg(T_i), \deg(S_j) \in \Cl(X)$.

\medskip

{\centering
{\small
\setlength{\tabcolsep}{4pt}
\begin{longtable}{ccccc}
No.
&
\small{$\mathcal{R}(X)$}
&
\small{$[w_1,\ldots, w_r]$}
&
\small{$u$}
&
\small{$\dim(X)$}
\\
\toprule
1
&
$
\frac
{\KK[T_1, \ldots , T_7]}
{\langle T_{1}T_{2}T_{3}^2+T_{4}T_{5}+T_6T_7 \rangle}
$
&
\tiny{
\setlength{\arraycolsep}{2pt}
$
\begin{array}{c}
\left[\!\!
\begin{array}{ccccccc}
0 & 0 & 1 & 1 & 1 & 1 & 1 
\\ 
1 & 1 & 0 & a & 2-a & b & 2-b
\end{array}
\!\!\right]
\\[1em]
1 \le a \le b 
\end{array}
$
}
&
\tiny{
\setlength{\arraycolsep}{2pt}
$
\left[\!\!
\begin{array}{c}
1
\\
1+b
\end{array}
\!\!\right]
$
}
&
\small{$4$}
\\
\midrule
2
&
$
\frac
{\KK[T_1, \ldots , T_7]}
{\langle T_{1}T_{2}T_{3}+T_{4}T_{5}+T_6T_7 \rangle}
$
&
\tiny{
\setlength{\arraycolsep}{2pt}
$
\left[\!\!
\begin{array}{ccccccc}
0 & 0 & 1 & 1 & 0 & 1 & 0 
\\ 
1 & 1 & 0 & 1 & 1 & 1 & 1
\end{array}
\!\!\right]
$
}
&
\tiny{
\setlength{\arraycolsep}{2pt}
$
\left[\!\!
\begin{array}{c}
1
\\
2
\end{array}
\!\!\right]
$
}
&
\small{$4$} 
\\
\midrule
3
&
$
\frac{\KK[T_1, \ldots , T_6]}
{\langle T_{1}T_{2}T_{3}^2+T_{4}T_{5}+T_{6}^2 \rangle}
$
&
\tiny{
\setlength{\arraycolsep}{2pt}
$
\begin{array}{c}
\left[\!\!
\begin{array}{cccccc}
0 & 0 & 1 & 1 & 1 & 1 
\\ 
1 & 1 & 0 & 2-a & a & 1
\end{array}
\!\!\right]
\\[1em]
a \ge 1
\end{array}
$
}
&
\tiny{
\setlength{\arraycolsep}{2pt}
$
\left[\!\!
\begin{array}{c}
1
\\
1+a
\end{array}
\!\!\right]
$
}
&
\small{$3$} 
\\
\midrule
4
&
$
\begin{array}{c}
\frac
{\KK[T_1, \ldots , T_6, S_1,\ldots,S_m]}
{\langle T_{1}T_{2}^{l_{2}}+T_{3}T_{4}^{l_{4}}+T_5T_{6}^{l_{6}} \rangle}
\\
\scriptstyle m \ge 0
\end{array}
$
&
\tiny{
\setlength{\arraycolsep}{2pt}
$
\begin{array}{c}
\left[\!\!
\begin{array}{cccccc|ccc}
0 & 1 & a & 1 & b & 1 & c_1 & \ldots & c_m 
\\ 
1 & 0 & 1 & 0 & 1 & 0 & 1 & \ldots & 1
\end{array}
\!\!\right]
\\[1em]
0 \le a \le b, \ c_1 \le \ldots \le c_m, 
\\
l_{2}=a+l_{4}=b+l_{6}
\end{array}
$
}
&
\tiny{
\setlength{\arraycolsep}{2pt}
$
\begin{array}{c}
\left[\!\!
\begin{array}{c}
d+1
\\
1
\end{array}
\!\!\right]
\\[1em]
d \sei \max(b,c_m)
\end{array}
$
}
&
\small{$m+3$}
\\
\midrule
5
&
$
\begin{array}{c}
\frac
{\KK[T_1, \ldots , T_6, S_1, \ldots, S_m]}
{\langle T_{1}T_{2}+T_{3}^2T_{4}+T_5^2T_{6} \rangle}
\\
\scriptstyle m \ge 0
\end{array}
$
&
\tiny{
\setlength{\arraycolsep}{2pt}
$
\begin{array}{c}
\left[\!\!
\begin{array}{cccccc|ccc}
0 &  2a+1 & a & 1 & a & 1 & 1 & \ldots & 1 
\\ 
1 & 1 & 1 & 0 & 1 & 0 & 0 & \ldots & 0
\end{array}
\!\!\right]
\\[1em]
a \ge 0
\end{array}
$
}
&
\tiny{
\setlength{\arraycolsep}{2pt}
$
\left[\!\!
\begin{array}{c}
2a+2
\\
1
\end{array}
\!\!\right]
$
}
&
\small{$m+3$} 
\\
\midrule 
6
&
$
\begin{array}{c}
\frac
{\KK[T_1, \ldots , T_6, S_1,\ldots,S_m]} 
{\langle T_{1}T_{2}+T_{3}T_{4}+T_5^2T_{6} \rangle}
\\
\scriptstyle m \ge 1
\end{array}
$
&
\tiny{
\setlength{\arraycolsep}{2pt}
$
\begin{array}{c}
\left[\!\!
\begin{array}{cccccc|ccc}
0 & 2c+1 & a & b & c & 1 & 1 & \ldots & 1 
\\ 
1 & 1 & 1 & 1 & 1 & 0 & 0 & \ldots & 0
\end{array}
\!\!\right]
\\[1em]
a, b, c \ge 0, \quad a<b, \\
a+b=2c+1
\end{array}
$
}
&
\tiny{
\setlength{\arraycolsep}{2pt}
$
\left[\!\!
\begin{array}{c}
2c+2
\\
1
\end{array}
\!\!\right]
$
}
&
\small{$m+3$} 
\\
\midrule 
7
&
$
\begin{array}{c}
\frac
{\KK[T_1, \ldots , T_6, S_1,\ldots,S_m]}
{\langle T_{1}T_{2}+T_{3}T_{4}+T_5T_{6} \rangle}
\\
\scriptstyle m \ge 1
\end{array}
$
&
\tiny{
\setlength{\arraycolsep}{2pt}
$
\left[\!\!
\begin{array}{cccccc|ccc}
0 & 0 & 0 & 0 & -1 & 1 & 1 & \ldots & 1 
\\ 
1 & 1 & 1 & 1 & 1 & 1 & 0 & \ldots & 0
\end{array}
\!\!\right]
$
}
&
\tiny{
\setlength{\arraycolsep}{2pt}
$
\left[\!\! 
\begin{array}{c}
1
\\
2
\end{array}
\!\!\right]
$
}
&
\small{$m+3$} 
\\
\midrule
8
&
$
\begin{array}{c}
\frac
{\KK[T_1, \ldots , T_6, S_1,\ldots,S_m]}
{\langle T_{1}T_{2}+T_{3}T_{4}+T_5T_{6} \rangle}
\\
\scriptstyle m \ge 2
\end{array}
$
&
\tiny{
\setlength{\arraycolsep}{2pt}
$
\begin{array}{c}
\left[\!\! 
\begin{array}{cccccc|cccc}
0 & 0 & 0 & 0 & 0 & 0 & 1 & 1 & \ldots & 1 
\\
1 & 1 & 1 & 1 & 1 & 1 & 0 & a_2 & \ldots & a_m
\end{array}
\!\!\right]
\\[1em]
0 \le a_2 \le \ldots \le a_m, ~a_m >0
\end{array}
$
}
&
\tiny{
\setlength{\arraycolsep}{2pt}
$
\left[\!\! 
\begin{array}{c}
1
\\
a_m+1
\end{array}
\!\!\right]
$
}
&
\small{$m+3$}
\\
\midrule
9
&
$
\begin{array}{c}
\frac
{\KK[T_1, \ldots , T_6, S_1, \ldots, S_m]}
{\langle T_{1}T_{2}+T_{3}T_{4}+T_5T_{6} \rangle}
\\
\scriptstyle m \geq 2
\end{array}
$
&
\tiny{
\setlength{\arraycolsep}{2pt}
$
\begin{array}{c}
\left[\!\! 
\begin{array}{cccc|ccc}
0 & a_2 & \ldots & a_6 & 1 & \ldots & 1 
\\ 
1 & 1 & \ldots & 1 & 0 & \ldots & 0
\end{array}
\!\!\right]
\\[1em]
0 \le a_3 \le a_5 \le a_6 \le a_4 \le a_2,\\
 a_2=a_3+a_4=a_5+a_6
\end{array}
$
}
&
\tiny{
\setlength{\arraycolsep}{2pt}
$
\left[\!\! 
\begin{array}{c}
a_2+1
\\
1
\end{array}
\!\!\right]
$
}
&
\small{$m+3$} 
\\
\midrule 
10
&
$
\begin{array}{c}
\frac
{\KK[T_1, \ldots , T_5, S_1,\ldots,S_m]}
{\langle T_{1}T_{2}+T_{3}T_{4}+T_{5}^2 \rangle}
\\
\scriptstyle m \ge 1
\end{array}
$
&
\tiny{
\setlength{\arraycolsep}{2pt}
$
\left[\!\! 
\begin{array}{ccccc|ccc}
1 & 1 & 1 & 1 & 1 & 0 & \ldots & 0 
\\ 
-1 & 1 & 0 & 0 & 0 & 1 & \ldots & 1
\end{array}
\!\!\right]
$
}
&
\tiny{
\setlength{\arraycolsep}{2pt}
$
\left[\!\! 
\begin{array}{c}
2
\\
1
\end{array}
\!\!\right]
$
}
&
\small{$m+2$} 
\\
\midrule
11
&
$
\begin{array}{c}
\frac
{\KK[T_1, \ldots , T_5, S_1,\ldots,S_m]}
{\langle T_{1}T_{2}+T_{3}T_{4}+T_{5}^2 \rangle}
\\
\scriptstyle m \geq 2
\end{array}
$
&
\tiny{
\setlength{\arraycolsep}{2pt}
$
\begin{array}{c}
\left[\!\! 
\begin{array}{ccccc|cccc}
1 & 1 & 1 & 1 & 1 & 0 & a_2 & \ldots & a_m 
\\ 
0 & 0 & 0 & 0 & 0 & 1 & 1 & \ldots & 1
\end{array}
\!\!\right]
\\[1em]
0\le a_2 \le \ldots \le a_m, ~ a_m>0
\end{array}
$
}
&
\tiny{
\setlength{\arraycolsep}{2pt}
$
\left[\!\! 
\begin{array}{c}
a_m + 1
\\
1
\end{array}
\!\!\right]
$
}
&
\small{$m+2$}
\\
\midrule
12
& 
$
\begin{array}{c}
\frac{\KK[T_1, \ldots , T_5, S_1,\ldots,S_m]}
{\langle T_{1}T_{2}+T_{3}T_{4}+T_{5}^2 \rangle}
\\
\scriptstyle m \geq 2
\end{array}
$
&
\tiny{
\setlength{\arraycolsep}{2pt}
$
\begin{array}{c}
\left[\!\!
\begin{array}{ccccc|cccc}
1 & 1 & 1 & 1 & 1 & 0 & 0 & \ldots & 0 
\\ 
0 & 2c & a & b & c & 1 & 1 & \ldots & 1
\end{array}
\!\!\right]
\\[1em]
0 \le a \le c \le b, \ a+b=2c
\end{array}
$
}
&
\tiny{
\setlength{\arraycolsep}{2pt}
$
\left[\!\! 
\begin{array}{c}
1
\\
2c+1
\end{array}
\!\!\right]
$
}
&
\small{$m+2$} 
\\
\midrule
13
&
$
\begin{array}{c}
\frac
{\KK[T_1, \ldots , T_8]}
{
\left\langle 
\begin{array}{l}
\scriptstyle T_{1}T_{2}+T_{3}T_{4}+T_5T_{6},
\\[-3pt]
\scriptstyle \lambda T_{3}T_{4}+T_{5}T_{6}+T_{7}T_{8} 
\end{array}
\right\rangle
}
\\
\scriptstyle \lambda \in \KK^*  \setminus \{1\}
\end{array}
$
&
\tiny{
\setlength{\arraycolsep}{2pt}
$
\left[\!\! 
\begin{array}{cccccccc}
1 & 0 & 1 & 0 & 1 & 0 & 1 & 0 
\\ 
0 & 1 & 0 & 1 & 0 & 1 & 0 & 1
\end{array}
\!\!\right]
$
}
&
\tiny{
\setlength{\arraycolsep}{2pt}
$
\left[\!\! 
\begin{array}{c}
1
\\
1
\end{array}
\!\!\right]
$
}
&
\small{$4$}
\\
\bottomrule
\end{longtable}
}
}
\noindent
Moreover, each of the listed data defines 
a smooth rational non-toric 
projective variety of Picard number two 
coming with a torus action of complexity one.
\end{theorem}

Note that by our approach we obtain the 
Cox ring of the respective varieties for free
which in turn allows an explicit treatment of 
geometric questions by means of Cox ring 
based techniques.
In particular, the canonical divisor of the 
varieties listed in Theorem~\ref{thm:main1} 
admits a simple description in terms of the 
defining data.
This enables us to determine for every dimension 
the (finitely many) non-toric smooth rational 
Fano varieties of Picard number two that admit 
a torus action of complexity one;
we refer to Section~\ref{sec:geomfanos} for 
a geometric description of the listed 
varieties.

\begin{theorem}
\label{thm:main2}
Every smooth rational non-toric Fano variety 
of Picard number two that admits a torus action 
of complexity one is isomorphic to precisely 
one of the following varieties $X$, specified 
by their Cox ring $\cox(X)$, where the grading 
by $\Cl(X) = \ZZ^2$ is given by the 
matrix $[w_1, \ldots ,w_r]$ of generator 
degrees $\deg(T_i), \deg(S_j) \in \Cl(X)$
and we list the (ample) anticanonical class 
$-\mathcal{K}_X$.

\medskip

{\centering
{\small
\setlength{\tabcolsep}{4pt}
\begin{longtable}{ccccc}
No.
&
\small{$\mathcal{R}(X)$}
&
\small{$[w_1,\ldots, w_r]$}
&
\small{$-\mathcal{K}_X$}
&
\small{$\dim(X)$}
\\
\toprule
1
&
$
\frac
{\KK[T_1, \ldots , T_7]}
{\langle T_{1}T_{2}T_{3}^2+T_{4}T_{5}+T_6T_7 \rangle}
$
&
\tiny{
\setlength{\arraycolsep}{2pt}
$
\left[\!\!
\begin{array}{ccccccc}
0 & 0 & 1 & 1 & 1 & 1 & 1 
\\ 
1 & 1 & 0 & 1 & 1 &1 & 1
\end{array}
\!\!\right]
$
}
&
\tiny{
\setlength{\arraycolsep}{2pt}
$
\left[\!\!
\begin{array}{c}
3
\\
4
\end{array}
\!\!\right]
$
}
&
\small{$4$}
\\
\midrule
2
&
$
\frac
{\KK[T_1, \ldots , T_7]}
{\langle T_{1}T_{2}T_{3}+T_{4}T_{5}+T_6T_7 \rangle}
$
&
\tiny{
\setlength{\arraycolsep}{2pt}
$
\left[\!\!
\begin{array}{ccccccc}
0 & 0 & 1 & 1 & 0 & 1 & 0 
\\ 
1 & 1 & 0 & 1 & 1 & 1 & 1
\end{array}
\!\!\right]
$
}
&
\tiny{
\setlength{\arraycolsep}{2pt}
$
\left[\!\!
\begin{array}{c}
2
\\
4
\end{array}
\!\!\right]
$
}
&
\small{$4$} 
\\
\midrule
3
&
$
\frac{\KK[T_1, \ldots , T_6]}
{\langle T_{1}T_{2}T_{3}^2+T_{4}T_{5}+T_{6}^2 \rangle}
$
&
\tiny{
\setlength{\arraycolsep}{2pt}
$
\left[\!\!
\begin{array}{cccccc}
0 & 0 & 1 & 1 & 1 & 1 
\\ 
1 & 1 & 0 & 1 & 1 & 1
\end{array}
\!\!\right]
$
}
&
\tiny{
\setlength{\arraycolsep}{2pt}
$
\left[\!\!
\begin{array}{c}
2
\\
3
\end{array}
\!\!\right]
$
}
&
\small{$3$} 
\\
\midrule
4.A
&
$
\begin{array}{c}
\frac
{\KK[T_1, \ldots , T_6, S_1,\ldots,S_m]}
{\langle T_{1}T_{2}+T_{3}T_{4}+T_5T_{6} \rangle}
\\
\scriptstyle m \ge 0
\end{array}
$
&
\tiny{
\setlength{\arraycolsep}{2pt}
$
\begin{array}{c}
\left[\!\!
\begin{array}{cccccc|cccc}
0 & 1 & 0 & 1 & 0 & 1 & c & 0 & \ldots & 0 
\\ 
1 & 0 & 1 & 0 & 1 & 0 & 1 & 1 & \ldots & 1
\end{array}
\!\!\right]
\\[1em]
c \in \{-1, 0 \}, \\ c:=0 \text{ if } m=0
\end{array}
$
}
&
\tiny{
\setlength{\arraycolsep}{2pt}
$
\left[\!\!
\begin{array}{c}
2+c 
\\
2+m
\end{array}
\!\!\right]
$
}
&
\small{$m+3$}
\\
\midrule
4.B
&
$
\begin{array}{c}
\frac
{\KK[T_1, \ldots , T_6, S_1,\ldots,S_m]}
{\langle T_{1}T_{2}^2+T_{3}T_{4}+T_5T_{6} \rangle}
\\
\scriptstyle m \ge 0
\end{array}
$
&
\tiny{
\setlength{\arraycolsep}{2pt}
$
\left[\!\!
\begin{array}{cccccc|ccc}
0 & 1 & 1 & 1 & 1 & 1 & 1 & \ldots & 1
\\ 
1 & 0 & 1 & 0 & 1 & 0 & 1 & \ldots & 1
\end{array}
\!\!\right]
$
}
&
\tiny{
\setlength{\arraycolsep}{2pt}
$
\left[\!\!
\begin{array}{c}
3+m 
\\
2+m
\end{array}
\!\!\right]
$
}
&
\small{$m+3$}
\\
\midrule
4.C
&
$
\begin{array}{c}
\frac
{\KK[T_1, \ldots , T_6, S_1,\ldots,S_m]}
{\langle T_{1}T_{2}^2+T_{3}T_{4}^2+T_5T_{6}^2 \rangle}
\\
\scriptstyle m \ge 0
\end{array}
$
&
\tiny{
\setlength{\arraycolsep}{2pt}
$
\left[\!\!
\begin{array}{cccccc|ccc}
0 & 1 & 0 & 1 & 0 & 1 & 0 & \ldots & 0 
\\ 
1 & 0 & 1 & 0 & 1 & 0 & 1 & \ldots & 1
\end{array}
\!\!\right]
$
}
&
\tiny{
\setlength{\arraycolsep}{2pt}
$
\left[\!\!
\begin{array}{c}
1
\\
2+m
\end{array}
\!\!\right]
$
}
&
\small{$m+3$}
%
%
%
%
\\
\midrule
5
&
$
\begin{array}{c}
\frac
{\KK[T_1, \ldots , T_6, S_1, \ldots, S_m]}
{\langle T_{1}T_{2}+T_{3}^2T_{4}+T_5^2T_{6} \rangle}
\\
\scriptstyle m \ge 1
\end{array}
$
&
\tiny{
\setlength{\arraycolsep}{2pt}
$
\begin{array}{c}
\left[\!\!
\begin{array}{cccccc|ccc}
0 & 2a+1 & a & 1 & a & 1 & 1 & \ldots & 1 
\\ 
1 & 1 & 1 & 0 & 1 & 0 & 0 & \ldots & 0
\end{array}
\!\!\right]
\\[1em]
0 \le 2a < m
\end{array}
$
}
&
\tiny{
\setlength{\arraycolsep}{2pt}
$
\left[\!\!
\begin{array}{c}
2a+m+2
\\
2
\end{array}
\!\!\right]
$
}
&
\small{$m+3$} 
\\
\midrule
6
& 
$
\begin{array}{c}
\frac
{\KK[T_1, \ldots , T_6, S_1,\ldots,S_m]} 
{\langle T_{1}T_{2}+T_{3}T_{4}+T_5^2T_{6} \rangle}
\\
\scriptstyle m \ge 1
\end{array}
$
&
\tiny{
\setlength{\arraycolsep}{2pt}
$
\begin{array}{c}
\left[\!\!
\begin{array}{cccccc|ccc}
0 & 2c+1 & a & b & c & 1 & 1 & \ldots & 1 
\\ 
1 & 1 & 1 & 1 & 1 & 0 & 0 & \ldots & 0
\end{array}
\!\!\right]
\\[1em]
a, b, c \ge 0, \quad a<b,\\
a+b=2c+1,\\
m > 3c+1
\end{array}
$
}
&
\tiny{
\setlength{\arraycolsep}{2pt}
$
\left[\!\!
\begin{array}{c}
3c+2+m
\\
3
\end{array}
\!\!\right]
$
}
&
\small{$m+3$} 
\\
\midrule 
7
&
$
\begin{array}{c}
\frac
{\KK[T_1, \ldots , T_6, S_1,\ldots,S_m]}
{\langle T_{1}T_{2}+T_{3}T_{4}+T_5T_{6} \rangle}
\\
\scriptstyle 1 \le m \le 3
\end{array}
$
&
\tiny{
\setlength{\arraycolsep}{2pt}
$
\left[\!\!
\begin{array}{cccccc|ccc}
0 & 0 & 0 & 0 & -1 & 1 & 1 & \ldots & 1 
\\ 
1 & 1 & 1 & 1 & 1 & 1 & 0 & \ldots & 0
\end{array}
\!\!\right]
$
}
&
\tiny{
\setlength{\arraycolsep}{2pt}
$
\left[\!\! 
\begin{array}{c}
m
\\
4
\end{array}
\!\!\right]
$
}
&
\small{$m+3$} 
\\
\midrule
8
&
$
\begin{array}{c}
\frac
{\KK[T_1, \ldots , T_6, S_1,\ldots,S_m]}
{\langle T_{1}T_{2}+T_{3}T_{4}+T_5T_{6} \rangle}
\\
\scriptstyle m \ge 2
\end{array}
$
&
\tiny{
\setlength{\arraycolsep}{2pt}
$
\begin{array}{c}
\left[\!\! 
\begin{array}{cccccc|cccc}
0 & 0 & 0 & 0 & 0 & 0 & 1 & 1 & \ldots & 1 
\\
1 & 1 & 1 & 1 & 1 & 1 & 0 & a_2 & \ldots & a_m
\end{array}
\!\!\right]
\\[1em]
0 \le a_2 \le \ldots \le a_m, 
\\
a_m\in\{1,2,3\},
\\
4+\sum_{k=2}^m a_k > ma_m
\end{array}
$
}
&
\tiny{
\setlength{\arraycolsep}{2pt}
$
\left[\!\! 
\begin{array}{c}
m
\\
4+\sum_{k=2}^m a_k
\end{array}
\!\!\right]
$
}
&
\small{$m+3$} 
\\
\midrule
9
&
$
\begin{array}{c}
\frac
{\KK[T_1, \ldots , T_6, S_1, \ldots, S_m]}
{\langle T_{1}T_{2}+T_{3}T_{4}+T_5T_{6} \rangle}
\\
\scriptstyle m \geq 2
\end{array}
$
&
\tiny{
\setlength{\arraycolsep}{2pt}
$
\begin{array}{c}
\left[\!\! 
\begin{array}{cccc|ccc}
0 & a_2 & \ldots & a_6 & 1 & \ldots & 1 
\\ 
1 & 1 & \ldots & 1 & 0 & \ldots & 0
\end{array}
\!\!\right]
\\[1em]
0 \le a_3 \le a_5 \le a_6 \le a_4 \le a_2,\\
 a_2=a_3+a_4=a_5+a_6, \\ 
2a_2 < m
\end{array}
$
}
&
\tiny{
\setlength{\arraycolsep}{2pt}
$
\left[\!\! 
\begin{array}{c}
2a_2+m
\\
4
\end{array}
\!\!\right]
$
}
&
\small{$m+3$} 
\\
\midrule
10
& 
$
\begin{array}{c}
\frac
{\KK[T_1, \ldots , T_5, S_1,\ldots,S_m]}
{\langle T_{1}T_{2}+T_{3}T_{4}+T_{5}^2 \rangle}
\\
\scriptstyle 1 \le m \le 2
\end{array}
$
&
\tiny{
\setlength{\arraycolsep}{2pt}
$
\left[\!\! 
\begin{array}{ccccc|ccc}
1 & 1 & 1 & 1 & 1 & 0 & \ldots & 0 
\\ 
-1 & 1 & 0 & 0 & 0 & 1 & \ldots & 1
\end{array}
\!\!\right]
$
}
&
\tiny{
\setlength{\arraycolsep}{2pt}
$
\left[\!\! 
\begin{array}{c}
3
\\
m
\end{array}
\!\!\right]
$
}
&
\small{$m+2$} 
\\
\midrule
11
&
$
\begin{array}{c}
\frac
{\KK[T_1, \ldots , T_5, S_1,\ldots,S_m]}
{\langle T_{1}T_{2}+T_{3}T_{4}+T_{5}^2 \rangle}
\\
\scriptstyle m \geq 2
\end{array}
$
&
\tiny{
\setlength{\arraycolsep}{2pt}
$
\begin{array}{c}
\left[\!\! 
\begin{array}{ccccc|cccc}
1 & 1 & 1 & 1 & 1 & 0 & a_2 & \ldots & a_m 
\\ 
0 & 0 & 0 & 0 & 0 & 1 & 1 & \ldots & 1
\end{array}
\!\!\right]
\\[1em]
0\le a_2 \le \ldots \le a_m,
\\
a_m\in\{1,2\},
\\
~3+\sum_{k=2}^m a_k > ma_m
\end{array}
$
}
&
\tiny{
\setlength{\arraycolsep}{2pt}
$
\left[\!\! 
\begin{array}{c}
3+ \sum_{k=2}^m a_k
\\
m
\end{array}
\!\!\right]
$
}
&
\small{$m+2$}
\\
\midrule 
12
&
$
\begin{array}{c}
\frac{\KK[T_1, \ldots , T_5, S_1,\ldots,S_m]}
{\langle T_{1}T_{2}+T_{3}T_{4}+T_{5}^2 \rangle}
\\
\scriptstyle m \geq 2
\end{array}
$
&
\tiny{
\setlength{\arraycolsep}{2pt}
$
\begin{array}{c}
\left[\!\!
\begin{array}{ccccc|cccc}
1 & 1 & 1 & 1 & 1 & 0 & 0 & \ldots & 0 
\\ 
0 & 2c & a & b & c & 1 & 1 & \ldots & 1
\end{array}
\!\!\right]
\\[1em]
0 \le a \le c \le b,  \ a+b=2c, \\
3c<m
\end{array}
$
}
&
\tiny{
\setlength{\arraycolsep}{2pt}
$
\left[\!\! 
\begin{array}{c}
3
\\
3c+m
\end{array}
\!\!\right]
$
}
&
\small{$m+2$} 
\\
\midrule
13
&
$
\begin{array}{c}
\frac
{\KK[T_1, \ldots , T_8]}
{
\left\langle 
\begin{array}{l}
\scriptstyle T_{1}T_{2}+T_{3}T_{4}+T_5T_{6},
\\[-3pt]
\scriptstyle \lambda T_{3}T_{4}+T_{5}T_{6}+T_{7}T_{8} 
\end{array}
\right\rangle
}
\\
\scriptstyle \lambda \in \KK^*  \setminus \{1\}
\end{array}
$
&
\tiny{
\setlength{\arraycolsep}{2pt}
$
\left[\!\! 
\begin{array}{cccccccc}
1 & 0 & 1 & 0 & 1 & 0 & 1 & 0 
\\ 
0 & 1 & 0 & 1 & 0 & 1 & 0 & 1
\end{array}
\!\!\right]
$
}
&
\tiny{
\setlength{\arraycolsep}{2pt}
$
\left[\!\! 
\begin{array}{c}
2
\\
2
\end{array}
\!\!\right]
$
}
&
\small{$4$}
\\
\bottomrule
\end{longtable}
}
}
\noindent
Moreover, each of the listed data defines 
a smooth rational non-toric 
Fano variety of Picard number two 
coming with a torus action of complexity one.
\end{theorem}

For $\KK = \CC$, the assumption 
of rationality can be omitted in
Theorem~\ref{thm:main2}
due to~\cite[Sec.~2.1]{ProkhorovFano}
and~\cite[Rem.~4.4.1.5]{ArDeHaLa}.
A closer look to the varieties of 
Theorem~\ref{thm:main2} 
reveals that they all are obtained
from a series of lower dimensional 
varieties via iterating the following 
procedure: we take a certain $\PP_1$-bundle 
over the given variety, 
apply a natural series of flips and then 
contract a prime divisor.
In terms of Cox rings, this generalized 
cone construction simply means duplicating 
a free weight, i.e., given a variable
not showing up in the defining relations,
one adds a further one of the same degree,
see Section~\ref{section:finite}.
Proposition~\ref{prop:noisodivs} and 
Theorem~\ref{thm:duplicate}
then yield the following.

\goodbreak

\begin{corollary}
\label{cor:duplicate}
Every smooth rational non-toric Fano variety 
with a torus action of complexity one and Picard 
number two arises via iterated 
duplication of a free weight from a smooth rational 
projective (not necessarily Fano)
variety with a torus action of complexity 
one, Picard number two and dimension at most seven.
\end{corollary}

Note that we cannot expect such a statement 
in general: 
Remark~\ref{rem:toric-not} shows that
the smooth toric Fano varieties of Picard 
number two do not allow a bound~$d$ such that
they all arise via iterated duplication of 
free weights from smooth varieties of dimension 
at most $d$.

Similar to the Fano varieties, we can figure out 
the almost Fano varieties from Theorem~\ref{thm:main1}, 
i.e., those with a big and nef anticanonical divisor.
In general, i.e., without the assumption of a torus 
action, the classification of smooth almost Fano 
varieties of Picard number two is widely open;
for the threefold case, we refer to the work of 
Jahnke, Peternell and Radloff~\cite{JaPeRa1,JaPeRa2}. 
In the setting of a torus action of complexity 
one, the following result together with 
Theorem~\ref{thm:main2} settles the problem in any 
dimension; by a \emph{truly almost Fano variety} we 
mean an almost Fano variety which is not Fano.

\begin{theorem}
\label{thm:main3}
Every smooth rational projective non-toric 
truly almost Fano variety of Picard number 
two that admits a 
torus action of complexity one 
is isomorphic to precisely one of the following 
varieties $X$, specified by their Cox ring 
$\cox(X)$ and an ample class $u \in \Cl(X)$,
where we always have $\Cl(X) = \ZZ^2$ and the grading 
is fixed by the matrix $[w_1, \ldots ,w_r]$
of generator degrees $\deg(T_i), \deg(S_j) \in \Cl(X)$.

\medskip

{\centering
{\small
\setlength{\tabcolsep}{4pt}
\begin{longtable}{ccccc}
No.
&
\small{$\mathcal{R}(X)$}
&
\small{$[w_1,\ldots, w_r]$}
&
\small{$u$}
&
\small{$\dim(X)$}
\\
\toprule
4.A
&
$
\begin{array}{c}
\frac
{\KK[T_1, \ldots , T_6, S_1,\ldots,S_m]}
{\langle T_{1}T_{2}+T_{3}T_{4}+T_5T_{6} \rangle}
\\
\scriptstyle m \ge 1
\end{array}
$
&
\tiny{
\setlength{\arraycolsep}{2pt}
$
\begin{array}{c}
\left[\!\!
\begin{array}{cccccc|ccc}
0 & 1 & 0 & 1 & 0 & 1 & c_1 & \ldots & c_m
\\ 
1 & 0 & 1 & 0 & 1 & 0 & 1 &  \ldots & 1
\end{array}
\!\!\right]
\\[1em]
c_1 \le \ldots \le c_m \\
d \sei \max(0,c_m) \\
(2+m)d = 2 + c_1 +\dots +c_m
\end{array}
$
}
&
\tiny{
\setlength{\arraycolsep}{2pt}
$
\left[\!\!
\begin{array}{c}
1
\\
1+d
\end{array}
\!\!\right]
$
}
&
\small{$m+3$}
\\
\midrule
4.B
&
$
\begin{array}{c}
\frac
{\KK[T_1, \ldots , T_6, S_1,\ldots,S_m]}
{\langle T_{1}T_{2}^{2}+T_{3}T_{4}+T_5T_{6} \rangle}
\\
\scriptstyle m \ge 1
\end{array}
$
&
\tiny{
\setlength{\arraycolsep}{2pt}
$
\left[\!\!
\begin{array}{cccccc|cccc}
0 & 1 & 1 & 1 & 1 & 1 & 0 & 1 & \ldots & 1
\\ 
1 & 0 & 1 & 0 & 1 & 0 & 1 & 1 & \ldots & 1
\end{array}
\!\!\right]
$
}
&
\tiny{
\setlength{\arraycolsep}{2pt}
$
\left[\!\!
\begin{array}{c}
1
\\
2
\end{array}
\!\!\right]
$
}
&
\small{$m+3$}
\\
\midrule
4.C
&
$
\begin{array}{c}
\frac
{\KK[T_1, \ldots , T_6, S_1,\ldots,S_m]}
{\langle T_{1}T_{2}^{2}+T_{3}T_{4}^{2}+T_5T_{6}^{2} \rangle}
\\
\scriptstyle m \ge 1
\end{array}
$
&
\tiny{
\setlength{\arraycolsep}{2pt}
$
\left[\!\!
\begin{array}{cccccc|cccc}
0 & 1 & 0 & 1 & 0 & 1 & -1 & 0 & \ldots & 0
\\ 
1 & 0 & 1 & 0 & 1 & 0 & 1 & 1 & \ldots & 1
\end{array}
\!\!\right]
$
}
&
\tiny{
\setlength{\arraycolsep}{2pt}
$
\left[\!\!
\begin{array}{c}
1
\\
1
\end{array}
\!\!\right]
$
}
&
\small{$m+3$}
\\
\midrule
4.D
&
$
\begin{array}{c}
\frac
{\KK[T_1, \ldots , T_6, S_1,\ldots,S_m]}
{\langle T_{1}T_{2}^{2}+T_{3}T_{4}^{2}+T_5T_{6} \rangle}
\\
\scriptstyle m \ge 0
\end{array}
$
&
\tiny{
\setlength{\arraycolsep}{2pt}
$
\left[\!\!
\begin{array}{cccccc|ccc}
0 & 1 & 0 & 1 & 1 & 1 & 1  & \ldots & 1
\\ 
1 & 0 & 1 & 0 & 1 & 0 & 1  & \ldots & 1
\end{array}
\!\!\right]
$
}
&
\tiny{
\setlength{\arraycolsep}{2pt}
$
\left[\!\!
\begin{array}{c}
1
\\
2
\end{array}
\!\!\right]
$
}
&
\small{$m+3$}
\\
\midrule
4.E
&
$
\begin{array}{c}
\frac
{\KK[T_1, \ldots , T_6, S_1,\ldots,S_m]}
{\langle T_{1}T_{2}^{3}+T_{3}T_{4}+T_5T_{6} \rangle}
\\
\scriptstyle m \ge 0
\end{array}
$
&
\tiny{
\setlength{\arraycolsep}{2pt}
$
\left[\!\!
\begin{array}{cccccc|ccc}
0 & 1 & 2 & 1 & 2 & 1 & 2  & \ldots & 2
\\ 
1 & 0 & 1 & 0 & 1 & 0 & 1  & \ldots & 1
\end{array}
\!\!\right]
$
}
&
\tiny{
\setlength{\arraycolsep}{2pt}
$
\left[\!\!
\begin{array}{c}
1
\\
3
\end{array}
\!\!\right]
$
}
&
\small{$m+3$}
\\
\midrule
4.F
&
$
\begin{array}{c}
\frac
{\KK[T_1, \ldots , T_6, S_1,\ldots,S_m]}
{\langle T_{1}T_{2}^{3}+T_{3}T_{4}^{2}+T_5T_{6}^{2} \rangle}
\\
\scriptstyle m \ge 0
\end{array}
$
&
\tiny{
\setlength{\arraycolsep}{2pt}
$
\left[\!\!
\begin{array}{cccccc|ccc}
0 & 1 & 1 & 1 & 1 & 1 & 1  & \ldots & 1
\\ 
1 & 0 & 1 & 0 & 1 & 0 & 1  & \ldots & 1
\end{array}
\!\!\right]
$
}
&
\tiny{
\setlength{\arraycolsep}{2pt}
$
\left[\!\!
\begin{array}{c}
1
\\
2
\end{array}
\!\!\right]
$
}
&
\small{$m+3$}
\\
\midrule
5
&
$
\begin{array}{c}
\frac
{\KK[T_1, \ldots , T_6, S_1, \ldots, S_m]}
{\langle T_{1}T_{2}+T_{3}^2T_{4}+T_5^2T_{6} \rangle}
\\
\scriptstyle m \ge 0
\end{array}
$
&
\tiny{
\setlength{\arraycolsep}{2pt}
$
\begin{array}{c}
\left[\!\!
\begin{array}{cccccc|ccc}
0 &  2a+1 & a & 1 & a & 1 & 1 & \ldots & 1 
\\ 
1 & 1 & 1 & 0 & 1 & 0 & 0 & \ldots & 0
\end{array}
\!\!\right]
\\[1em]
m=2a
\end{array}
$
}
&
\tiny{
\setlength{\arraycolsep}{2pt}
$
\left[\!\!
\begin{array}{c}
m+2
\\
1
\end{array}
\!\!\right]
$
}
&
\small{$m+3$} 
\\
\midrule 
6
&
$
\begin{array}{c}
\frac
{\KK[T_1, \ldots , T_6, S_1,\ldots,S_m]} 
{\langle T_{1}T_{2}+T_{3}T_{4}+T_5^2T_{6} \rangle}
\\
\scriptstyle m \ge 1
\end{array}
$
&
\tiny{
\setlength{\arraycolsep}{2pt}
$
\begin{array}{c}
\left[\!\!
\begin{array}{cccccc|ccc}
0 & 2c+1 & a & b & c & 1 & 1 & \ldots & 1 
\\ 
1 & 1 & 1 & 1 & 1 & 0 & 0 & \ldots & 0
\end{array}
\!\!\right]
\\[1em]
a, b, c \ge 0, \quad a<b, \\
a+b=2c+1, \\
m=3c+1
\end{array}
$
}
&
\tiny{
\setlength{\arraycolsep}{2pt}
$
\left[\!\!
\begin{array}{c}
2c+2
\\
1
\end{array}
\!\!\right]
$
}
&
\small{$m+3$} 
\\
\midrule 
7
&
$
\begin{array}{c}
\frac
{\KK[T_1, \ldots , T_6, S_1,\ldots,S_m]}
{\langle T_{1}T_{2}+T_{3}T_{4}+T_5T_{6} \rangle}
\\
\scriptstyle m = 4
\end{array}
$
&
\tiny{
\setlength{\arraycolsep}{2pt}
$
\left[\!\!
\begin{array}{cccccc|cccc}
0 & 0 & 0 & 0 & -1 & 1 & 1 & 1 & 1 & 1 
\\ 
1 & 1 & 1 & 1 & 1 & 1 & 0 & 0 & 0 & 0
\end{array}
\!\!\right]
$
}
&
\tiny{
\setlength{\arraycolsep}{2pt}
$
\left[\!\! 
\begin{array}{c}
1
\\
2
\end{array}
\!\!\right]
$
}
&
\small{$7$} 
\\
\midrule
8
&
$
\begin{array}{c}
\frac
{\KK[T_1, \ldots , T_6, S_1,\ldots,S_m]}
{\langle T_{1}T_{2}+T_{3}T_{4}+T_5T_{6} \rangle}
\\
\scriptstyle m \ge 2
\end{array}
$
&
\tiny{
\setlength{\arraycolsep}{2pt}
$
\begin{array}{c}
\left[\!\! 
\begin{array}{cccccc|cccc}
0 & 0 & 0 & 0 & 0 & 0 & 1 & 1 & \ldots & 1 
\\
1 & 1 & 1 & 1 & 1 & 1 & 0 & a_2 & \ldots & a_m
\end{array}
\!\!\right]
\\[1em]
0 \le a_2 \le \ldots \le a_m, ~a_m >0, \\
4 + a_2 + \ldots + a_m = ma_m
\end{array}
$
}
&
\tiny{
\setlength{\arraycolsep}{2pt}
$
\left[\!\! 
\begin{array}{c}
1
\\
a_m+1
\end{array}
\!\!\right]
$
}
&
\small{$m+3$}
\\
\midrule
9
&
$
\begin{array}{c}
\frac
{\KK[T_1, \ldots , T_6, S_1, \ldots, S_m]}
{\langle T_{1}T_{2}+T_{3}T_{4}+T_5T_{6} \rangle}
\\
\scriptstyle m \geq 2
\end{array}
$
&
\tiny{
\setlength{\arraycolsep}{2pt}
$
\begin{array}{c}
\left[\!\! 
\begin{array}{cccc|ccc}
0 & a_2 & \ldots & a_6 & 1 & \ldots & 1 
\\ 
1 & 1 & \ldots & 1 & 0 & \ldots & 0
\end{array}
\!\!\right]
\\[1em]
0 \le a_3 \le a_5 \le a_6 \le a_4 \le a_2,\\
 a_2=a_3+a_4=a_5+a_6, \\
m = 2a_2
\end{array}
$
}
&
\tiny{
\setlength{\arraycolsep}{2pt}
$
\left[\!\! 
\begin{array}{c}
a_2+1
\\
1
\end{array}
\!\!\right]
$
}
&
\small{$m+3$} 
\\
\midrule 
10
&
$
\begin{array}{c}
\frac
{\KK[T_1, \ldots , T_5, S_1,\ldots,S_m]}
{\langle T_{1}T_{2}+T_{3}T_{4}+T_{5}^2 \rangle}
\\
\scriptstyle m = 3
\end{array}
$
&
\tiny{
\setlength{\arraycolsep}{2pt}
$
\left[\!\! 
\begin{array}{ccccc|ccc}
1 & 1 & 1 & 1 & 1 & 0 & 0 & 0 
\\ 
-1 & 1 & 0 & 0 & 0 & 1 & 1 & 1
\end{array}
\!\!\right]
$
}
&
\tiny{
\setlength{\arraycolsep}{2pt}
$
\left[\!\! 
\begin{array}{c}
2
\\
1
\end{array}
\!\!\right]
$
}
&
\small{$5$} 
\\
\midrule
11
&
$
\begin{array}{c}
\frac
{\KK[T_1, \ldots , T_5, S_1,\ldots,S_m]}
{\langle T_{1}T_{2}+T_{3}T_{4}+T_{5}^2 \rangle}
\\
\scriptstyle m \geq 2
\end{array}
$
&
\tiny{
\setlength{\arraycolsep}{2pt}
$
\begin{array}{c}
\left[\!\! 
\begin{array}{ccccc|cccc}
1 & 1 & 1 & 1 & 1 & 0 & a_2 & \ldots & a_m 
\\ 
0 & 0 & 0 & 0 & 0 & 1 & 1 & \ldots & 1
\end{array}
\!\!\right]
\\[1em]
0\le a_2 \le \ldots \le a_m, ~ a_m>0, \\
3 + a_2 + \ldots + a_m = ma_m
\end{array}
$
}
&
\tiny{
\setlength{\arraycolsep}{2pt}
$
\left[\!\! 
\begin{array}{c}
1
\\
a_m+1
\end{array}
\!\!\right]
$
}
&
\small{$m+2$}
\\
\midrule
12
& 
$
\begin{array}{c}
\frac{\KK[T_1, \ldots , T_5, S_1,\ldots,S_m]}
{\langle T_{1}T_{2}+T_{3}T_{4}+T_{5}^2 \rangle}
\\
\scriptstyle m \geq 3
\end{array}
$
&
\tiny{
\setlength{\arraycolsep}{2pt}
$
\begin{array}{c}
\left[\!\!
\begin{array}{ccccc|cccc}
1 & 1 & 1 & 1 & 1 & 0 & 0 & \ldots & 0 
\\ 
0 & 2c & a & b & c & 1 & 1 & \ldots & 1
\end{array}
\!\!\right]
\\[1em]
0 \le a \le c \le b, \ a+b=2c, \\
m=3c
\end{array}
$
}
&
\tiny{
\setlength{\arraycolsep}{2pt}
$
\left[\!\! 
\begin{array}{c}
1
\\
2c+1
\end{array}
\!\!\right]
$
}
&
\small{$m+2$} 
\\
\bottomrule
\end{longtable}
}
}
\noindent
Moreover, each of the listed data defines 
a smooth rational non-toric 
truly almost Fano variety of Picard number two 
coming with a torus action of complexity one.
\end{theorem}

The article is organized as follows.
In Section~\ref{sec:cpl1}, we briefly
present the necessary background on
rational varieties $X$ with a torus action 
of complexity one.
In Section~\ref{sec:firstStruct}, we 
derive first constraints on the 
defining data for smooth $X$ of 
Picard number two.
Section~\ref{sec:classif} is devoted
to proving the main results.
In Section~\ref{section:finite}, we
introduce and discuss duplication of free 
weights and show how to obtain the 
Fano varieties of Theorem~\ref{thm:main2}
via this procedure from lower dimensional 
varieties.
Finally, in Section~\ref{sec:geomfanos}, we 
describe the Fano varieties of Theorem~\ref{thm:main2}
in more geometric terms. 

\goodbreak

We would like to thank Ivo Radloff for his 
interest in the subject and for helpful 
discussions.

\tableofcontents


\section{Varieties with torus action of complexity one}
\label{sec:cpl1}

We recall from~\cite{HaSu:2010,HaHe:2013, ArDeHaLa} 
the Cox ring based approach to normal (projective) 
rational varieties~$X$ with a torus action of 
complexity one and thereby fix the notation 
used throughout the article. 
The first step is to describe the possible 
Cox rings~$\mathcal{R}(X)$; they are encoded 
by a pair $(A,P)$ of matrices of the following
shape.

\begin{notation}
\label{constr:defdata}
Fix $r \in \ZZ_{\ge 1}$, a sequence 
$n_0, \ldots, n_r \in \ZZ_{\ge 1}$, set 
$n := n_0 + \ldots + n_r$, and fix  
integers $m \in \ZZ_{\ge 0}$ and $0 < s < n+m-r$.
A pair $(A,P)$ of \emph{defining matrices} 
consists of 
\begin{itemize}
\item 
a matrix $A := [a_0, \ldots, a_r]$ 
with pairwise linearly independent 
column vectors $a_0, \ldots, a_r \in \KK^2$,
\item 
an integral block matrix $P$ of size 
$(r + s) \times (n + m)$, the columns 
of which are pairwise different primitive
vectors generating $\QQ^{r+s}$ as a cone:
\begin{eqnarray*}
P
& = & 
\left[ 
\begin{array}{cc}
L & 0 
\\
d & d'  
\end{array}
\right],
\end{eqnarray*}
where $d$ is an $(s \times n)$-matrix, $d'$ an $(s \times m)$-matrix 
and $L$ an $(r \times n)$-matrix built from tuples 
$l_i := (l_{i1}, \ldots, l_{in_i}) \in \ZZ_{\ge 1}^{n_i}$ 
as follows
\begin{eqnarray*}
L
& = & 
\left[
\begin{array}{cccc}
-l_0 & l_1 &   \ldots & 0 
\\
\vdots & \vdots   & \ddots & \vdots
\\
-l_0 & 0 &\ldots  & l_{r} 
\end{array}
\right].
\end{eqnarray*}
\end{itemize}
Denote by $v_{ij}$, where 
$0 \le i \le r$ and $1 \le j \le n_i$,
the first $n$ columns of $P$
and by~$v_k$, where $1 \le k \le m$,
the last $m$ ones.
Moreover, $e_{ij},e_k \in \ZZ^{n+m}$ 
are the canonical basis vectors indexed 
accordingly, i.e., 
$P$ sends $e_{ij}$ to $v_{ij}$ and $e_k$ to $v_k$. 
\end{notation}

\begin{construction}
\label{constr:RAPdown}
Fix $(A,P)$ as in~\ref{constr:defdata}.
Consider the polynomial ring 
$\KK[T_{ij},S_k]$ in the variables 
$T_{ij}$, where 
$0 \le i \le r$, $1 \le j \le n_i$,
and $S_k$, where $1 \le k \le m$.
For every $0 \le i \le r$, define a monomial
$$
T_i^{l_i} 
\  := \ 
T_{i1}^{l_{i1}} \cdots T_{in_i}^{l_{in_i}}
\ \in \ 
\KK[T_{ij},S_k].
$$
Denote by $\mathfrak{I}$ the set of 
all triples $I = (i_1,i_2,i_3)$ with 
$0 \le i_1 < i_2 < i_3 \le r$ 
and define for any $I \in \mathfrak{I}$ 
a trinomial 
$$
g_I
\ := \
g_{i_1,i_2,i_3}
\ := \
\det
\left[
\begin{array}{ccc}
T_{i_1}^{l_{i_1}} & T_{i_2}^{l_{i_2}} & T_{i_3}^{l_{i_3}}
\\
a_{i_1} & a_{i_2} & a_{i_3}
\end{array}
\right].
$$
Let $P^*$ denote the transpose of $P$,
consider the factor group 
$K := \ZZ^{n+m}/\text{im}(P^*)$
and the projection $Q \colon \ZZ^{n+m} \to K$.
We define a $K$-grading on 
$\KK[T_{ij},S_k]$ by setting
$$ 
\deg(T_{ij}) 
 \ := \ 
w_{ij}
\ := \ 
Q(e_{ij}),
\qquad
\deg(S_{k}) 
 \ := \ 
w_k 
\ := \ 
Q(e_{k}).
$$
Then the trinomials $g_I$ just introduced 
are $K$-homogeneous, all of the same degree.
In particular, we obtain a $K$-graded 
factor ring  
\begin{eqnarray*}
R(A,P)
& := &
\KK[T_{ij},S_k; \; 0 \le i \le r, \, 1 \le j \le n_i, 1 \le k \le m] 
\ / \
\bangle{g_I; \; I \in \mathfrak{I}}.
\end{eqnarray*}
\end{construction}

The rings $R(A,P)$ are precisely those which occur
as Cox rings of normal rational projective 
(or, more generally, complete $A_2$-) varieties with 
a torus action of complexity one; 
see~\cite[Theorem~1.5]{HaHe:2013}.
We recall basic properties.

\begin{remark}
\label{rem:ci}
The $K$-graded ring $R(A,P)$ of Construction~\ref{constr:RAPdown}
is a complete intersection: with $g_i := g_{i,i+1,i+2}$ 
we have 
$$ 
\bangle{g_I; \; I \in \mathfrak{I}}
\ = \ 
\bangle{g_0,\ldots,g_{r-2}},
\qquad\quad
\dim(R(A,P)) \ = \ n+m-(r-1).
$$
\end{remark}

\begin{remark}
\label{remark:admissibleops}
The following operations
on the columns and rows of the defining 
matrix $P$ do not change the isomorphy type
of the graded ring $R(A,P)$; 
we call them \emph{admissible operations}:
\begin{enumerate}
\item
swap two columns inside a block
$v_{ij_1}, \ldots, v_{ij_{n_i}}$,
\item
swap two whole column blocks
$v_{ij_1}, \ldots, v_{ij_{n_i}}$
and $v_{i'j_1}, \ldots, v_{i'j_{n_{i'}}}$,
\item
add multiples of the upper $r$ rows
to one of the last $s$ rows,
\item
any elementary row operation among the last $s$
rows,
\item
swap two columns inside the $d'$ block.
\end{enumerate}
The operations of type~(iii) and~(iv) do 
not even change $R(A,P)$, whereas  
types~(i), (ii), (v) correspond to 
certain renumberings of the variables
of $R(A,P)$ keeping the (graded) isomorphy 
type.
\end{remark}

\begin{remark}
If we have $n_i=1$ and $l_{i1} = 1$ in a defining 
matrix~$P$, then we may eliminate the variable 
$T_{i1}$ in $R(A,P)$ by modifying $P$ appropriately.
This can be repeated until $P$ is 
\emph{irredundant} in the sense that 
$l_{i1} + \ldots + l_{in_i} \ge 2$ holds 
for all $i = 0,\ldots, r$. 
\end{remark}

We come to the construction of all normal 
projective varieties sharing a given $R(A,P)$ 
as their Cox ring.
By $K_\QQ := K \otimes_{\ZZ} \QQ$ we denote
the rational vector space associated to an 
abelian group $K$.
We shortly write $w$ for $w \otimes 1 \in K_\QQ$ 
and, similarly, we keep the symbols when passing 
from homomorphisms $K \to K'$ 
to the associated linear maps $K_\QQ \to K'_\QQ$.  
Moreover, when we speak of a cone
$\tau \subseteq K_\QQ$, then we mean a convex, 
polyhedral cone in $K_\QQ$.
The relative interior of $\tau$ is denoted by 
$\tau^\circ$.

\begin{definition}
The \emph{moving cone} in $K_\QQ$ of the $K$-graded 
ring $R(A,P)$ from Construction~\ref{constr:RAPdown}
is the 
$$ 
\Mov(A,P) 
\ := \ 
\bigcap_{i,j} \cone(Q(e_{uv},e_{t}; \; (u,v) \ne (i,j)))  
\ \cap \ 
\bigcap_{k} \cone(Q(e_{uv},e_{t}; \; t \ne k)).
$$
\end{definition}

\begin{construction}
\label{constr:RAPu}
Take $R(A,P)$ as in Construction~\ref{constr:RAPdown}
and fix $u \in \Mov(A,P)^\circ$.
The $K$-grading on $\KK[T_{ij},S_k]$
defines an action of the quasitorus 
$H := \Spec \; \KK[K]$ on $\ol{Z} := \KK^{n+m}$
leaving $\ol{X} := V(g_I; \; I \in \mathfrak{I}) \subseteq \ol{Z}$ 
invariant. Consider 
$$ 
\wh{Z}
\ := \ 
\{z \in \ol{Z}; \; 
f(z) \ne 0 \text{ for some } 
f \in \KK[T_{ij},S_k]_{ \nu u}, \, 
\nu \in \ZZ_{>0} \}
\ \subseteq \ 
\ol{Z},
$$
the set of $H$-semistable points
with respect to the weight $u$.
Then $\wh{X} := \ol{X} \cap \wh{Z}$ is 
an open $H$-invariant set in $\ol{X}$ 
and we have a commutative diagram
$$ 
\xymatrix{
{\wh{X}}
\ar[r]
\ar[d]_{\quot H}^{\pi}
&
{\wh{Z}}
\ar[d]^{\quot H}
\\
X(A,P,u)
\ar[r]
&
Z}
$$
where $X=X(A,P,u)$ is a variety with torus action of complexity one,
$Z := \wh{Z} \quot H$ is a toric variety,
the downward maps are characteristic spaces and
the lower horizontal arrow is a closed embedding. 
We have
$$ 
\dim(X) = s+1,
\qquad
\Cl(X) \ \cong \ K,
\qquad
\mathcal{R}(X) \ \cong \ R(A,P).
$$
Moreover, for an irredundant defining matrix $P$, the 
variety $X = X(A,P)$ 
is non-toric if and only if $r \ge 2$ holds.
\end{construction}

See~\cite{HaHe:2013,ArDeHaLa} 
for the proof that this construction
yields indeed all normal rational projective varieties with 
a torus action of complexity one.
We will make intensive use of the machinery 
developed in~\cite{BeHa:2007,Ha:2008,ArDeHaLa}. 
Let us briefly summarize the necessary 
notions and statements in a series of remarks 
adapted to our needs.

\begin{remark}
\label{const:rlvu}
Fix defining matrices $(A,P)$ 
and let $\gamma \subseteq \QQ^{n+m}$ 
be the positive orthant,
spanned by the canonical basis vectors 
$e_{ij},e_k \in \ZZ^{n+m}$.
Every face $\gamma_0 \preceq \gamma$ defines a toric orbit 
in $\ol{Z} = \KK^{n+m}$:
$$ 
\ol{Z}(\gamma_0)
\ := \ 
\{
z \in \ol{Z}; \; 
z_{ij} \neq 0 \Leftrightarrow e_{ij} \in \gamma_0  
\text{ and } 
z_{k} \neq 0 \Leftrightarrow e_{k} \in\gamma_0 
\}
\ \subseteq \
\ol{Z}
$$
We say that $\gamma_0 \preceq \gamma$ is an 
\emph{$\mathfrak{F}$-face} (for $(A,P)$) 
if the associated toric
orbit meets the total coordinate space 
$\ol{X} = V(g_I; \; I \in \mathfrak{I}) \subseteq \ol{Z}$,
that means if we have
$$
\ol{X}(\gamma_0) 
\ := \
\ol{X} \cap \ol{Z}(\gamma_0)
\ \ne \
\emptyset.
$$  
In particular, $\ol{X}$ is the disjoint union 
of the locally closed pieces $\ol{X}(\gamma_0)$
associated to the $\mathfrak{F}$-faces.
\end{remark}

\begin{remark}
\label{rem:rlv2}
Fix $u \in \Mov(A,P)^\circ$.
Then, for the ambient toric variety $Z$ 
and $X=X(A,P,u)$ of  
Construction~\ref{constr:RAPu},
we have the collections of 
\emph{relevant faces}:
\begin{eqnarray*}
\rlv(Z) 
& := & 
\{ \gamma_0 \preceq \gamma; \; u \in Q(\gamma_0)^\circ\},
\\
\rlv(X) 
& := & 
\{ \gamma_0 \in \rlv(Z); \;
\gamma_0 \text{ is an } \mathfrak{F}\text{-face} 
\}.
\end{eqnarray*}
Let $\gamma_0^* := \gamma_0^\perp \cap \gamma\preceq \gamma$ 
denote the complementary face of $\gamma_0 \preceq \gamma$.
Then there is a bijection between $\rlv(Z)$ and the fan 
$\Sigma$ of the toric variety $Z$:
$$
\rlv(Z) \ \to \ \Sigma,
\qquad\qquad
\gamma_0 \ \mapsto \ P(\gamma_0^*).
$$  
The toric orbits of $Z$ correspond to the 
cones of the fan $\Sigma$ and thus to the 
cones of~$\rlv(Z)$. 
Concretely, the toric orbit of $Z$ associated 
with $\gamma_0\in \rlv(Z)$  is 
$$ 
Z(\gamma_0)
\ = \
\pi(\ol{Z}(\gamma_0)).
$$ 
The relevant faces $\rlv(X)$ of $X$  
define exactly the toric orbits of 
$Z$ that intersect $X \subseteq Z$
non-trivially
and thus give a locally closed 
decomposition
$$
X 
\ = \ 
\bigcup_{\gamma_0\in\rlv(X)} X(\gamma_0),
\qquad\qquad
X(\gamma_0) 
\ := \ 
X \cap Z(\gamma_0) 
\ = \ 
\pi((\ol{X}(\gamma_0)).
$$
The fan $\Sigma_X$ generated 
by the cones $\sigma=P(\gamma_0^*)$,
where $\gamma_0\in\rlv(X)$,
defines the minimal toric open subset 
$Z_X \subseteq Z$ containing $X$.
For the set of rays we have
$$ 
\Sigma_X^{(1)}
\ = \ 
\Sigma^{(1)}
\ = \ 
\{\varrho_{ij}, \varrho_k; \ 
0 \le i \le r, \ 1 \le j \le n_i, \ 1 \le k \le m\},
$$
where the $\varrho_{ij} := \cone(v_{ij})$ and
$\varrho_k := \cone(v_k)$ are the rays 
through the columns $v_{ij}$ and~$v_k$ of the 
defining matrix $P$.
\end{remark}

\begin{remark}
\label{rem:divcones}
Let $X=X(A,P,u)$ arise from   
Construction~\ref{constr:RAPu}.
Then the cones of effective, 
movable, semiample and ample
divisor classes are given as
$$ 
\Eff(X) 
\ = \ 
Q(\gamma),
\qquad
\Mov(X) 
\ = \ 
\Mov(A,P)
\ = \ 
\bigcap_{\gamma_0 \text{ facet of }\gamma} Q(\gamma_0),
$$
$$
\SAmple(X)
\ = \ 
\bigcap_{\gamma_0\in\rlv(X)} Q(\gamma_0),
\qquad
\Ample(X)
\ = \ 
\bigcap_{\gamma_0\in\rlv(X)} Q(\gamma_0)^\circ.
$$
In particular, the GIT-fan of the $H$-action 
on $\ol{X}$ induces the Mori chamber 
decomposition, i.e., it subdivides $\Mov(X)$ into 
the nef cones of the small birational relatives of~$X$. 
\end{remark}

\begin{remark}
\label{rem:Qfact}
Let $X=X(A,P,u)$ arise from   
Construction~\ref{constr:RAPu}.
Consider $\gamma_0\in\rlv(X)$ 
and $x\in X(\gamma_0)$.
Then the following statements hold:
\begin{enumerate}
\item
$x$ is $\QQ$-factorial if and only 
if $Q(\gamma_0)$ is full-dimensional,
\item
$x$ is factorial if and only if 
$Q$ maps $\lin(\gamma_0)\cap\ZZ^{n+m}$ onto $\Cl(X)$,
\item
$x$ is smooth if and only if $x$ is factorial 
and all $z \in \pi^{-1}(x)$ are
smooth in~$\ol{X}$.
\end{enumerate}
\end{remark}

\begin{remark}
\label{rem:fanoRAP}
Let $X=X(A,P,u)$ arise from   
Construction~\ref{constr:RAPu}.
The anticanonical class of $X$ does 
not depend on $u$ and is given by 
$$ 
-\mathcal{K}_X
\ = \ 
\kappa(A,P)
\ := \ 
\sum_{i,j} Q(e_{ij})
\ + \ 
\sum_k Q(e_k)
\ - \ 
(r-1) \sum_{j=0}^{n_0} l_{0j} Q(e_{0j})
\ \in \ 
K .
$$
In particular, a $K$-graded ring $R(A,P)$ 
is the Cox ring of a Fano variety 
if and only if $\kappa(A,P)$ belongs
to the relative interior of $\Mov(A,P)$. 
\end{remark}

\begin{remark}
\label{rem:trop}
Consider $X \subseteq Z$, where 
$X=X(A,P,u)$ and $Z$ are as in 
Construction~\ref{constr:RAPu}.
Then, with $\lambda := 0 \times \QQ^s \subseteq \QQ^{r+s}$, 
the canonical basis vectors $e_1,\ldots, e_r\in\ZZ^{r+s}$ and
$e_0 := -e_1- \ldots -e_r$,
the associated tropical variety is  
$$ 
\trop(X) 
\ = \ 
\lambda_0 \cup \ldots \cup \lambda_r
\ \subseteq \ 
\QQ^{r+s},
\qquad \text{where} \quad
\lambda_i \ := \  \lambda + \cone(e_i).
$$
Note that this defines the coarsest possible quasifan structure
on $\trop(X)$, and the lineality space of this quasifan is $\lambda$.
Moreover, a cone $\sigma \in \Sigma$
corresponds to $\gamma_0 \in \rlv(X)$ 
if and only if 
$\sigma^\circ \cap \trop(X) \ne \emptyset$ 
holds. 
\end{remark}

\begin{definition}
Consider $X \subseteq Z$, where 
$X=X(A,P,u)$ and $Z$ are as in 
Construction~\ref{constr:RAPu}.
A cone $\sigma \in \Sigma_X$ is called
\begin{enumerate}
\item
\emph{big}, if $\sigma \cap \lambda_i^\circ \ne \emptyset$ 
holds for each $i = 0, \ldots, r$.
\item
\emph{elementary big} if it is big,
has no rays inside $\lambda$ 
and precisely one inside $\lambda_i$ 
for each $i = 0, \ldots, r$.
\item
a \emph{leaf cone} if 
$\sigma \subseteq \lambda_i$ holds for some $i$. 
\end{enumerate}
We say that the variety $X$ is \emph{weakly tropical}, 
if the fan $\Sigma_X$ is supported on 
the tropical variety $\trop(X)$.
\end{definition}

\begin{remark}
\label{rem:weaklytrop}
Let $X=X(A,P,u)$ arise from   
Construction~\ref{constr:RAPu}.
\begin{enumerate}
\item
The fan $\Sigma_X$ is generated 
by big cones and leaf cones.
\item
Every big cone of $\Sigma_X$ is 
of the form $P(\gamma_0^*)$ with 
a $\gamma_0 \in \rlv(X)$.
\item
The tropical variety $\trop(X)$ 
is contained in the support of 
$\Sigma_X$.  
\item
$X$ is weakly tropical if and only 
$\Sigma_X$ consists of leaf 
cones.
\item
If $X$ is weakly tropical,
then 
$\lambda \subseteq \trop(X)$
is a union of cones 
of $\Sigma_X$.
\end{enumerate}
\end{remark}


\section{First structural constraints}
\label{sec:firstStruct}

We derive first constraints on the defining 
matrices of smooth rational varieties with 
a torus action of complexity one having 
Picard number two.
We work in the notation of Section~\ref{sec:cpl1}.
The aim is to show the following.

\begin{proposition}
\label{prop:smooth-rho2}
Let $X$ be a non-toric smooth rational 
projective variety with a torus action
of complexity one and Picard number 
$\rho(X) = 2$.
Then $X \cong X(A,P,u)$, 
where $P$ is irredundant and fits into one 
of the following cases:
\begin{enumerate}
\item[(I)]
We have $r=2$ and one of the following constellations:
\begin{enumerate}
\item
$m \ge 0$ and $n = 4+n_0$, where $n_0 \geq 3$,  $n_1 = n_2 = 2$.
\item
$m = 0$ and $n = 6$, where $n_0 = 3$,  $n_1 = 2$, $n_2 = 1$.
\item
$m = 0$ and $n = 5$, where $n_0 = 3$,  $n_1 = 1$, $n_2 = 1$.
\item
$m\ge 0$ and $n = 6$, where $n_0 = n_1 = n_2 = 2$.
\item
$m\ge 0$ and $n = 5$, where $n_0 = n_1 = 2$, $n_2 = 1$.
\item
$m\ge 1$ and $n = 4$, where $n_0 = 2$, $n_1 = n_2 = 1$.
\end{enumerate}
\item[(II)]
We have $r=3$ and one of the following constellations:
\begin{enumerate}
\item
$m = 0$ and $n = 8$, where $n_0 = n_1 = n_2 = n_3 = 2$.
\item
$m = 0$ and $n = 7$, where $n_0 = n_1 = n_2 = 2$, $n_3 = 1$.
\item
$m = 0$ and $n = 6$, where $n_0 = n_1 = 2$, $n_2 = n_3 = 1$.
\end{enumerate}
\end{enumerate}
\end{proposition}

The statement is an immediate consequence of 
Propositions~\ref{prop:4-1-m-pos} 
and~\ref{prop:4-1-m-0}; see end of this 
section.
Throughout the whole section, the defining 
matrix $P$ is irredundant.
In particular, $X(A, P, u)$ is non-toric
if and only if $r \ge 2$ holds, i.e., we have 
a relation in the Cox ring.
During our considerations, we will freely use 
the Remarks~\ref{const:rlvu} to~\ref{rem:weaklytrop}.

We first study the impact of $X = X(A,P,u)$
being locally factorial on the defining matrix~$P$,
where locally factorial means that the local rings 
of the points $x \in X$ are unique factorization 
domains.

\begin{lemma}
\label{lem:ebcni2}
Let $X=X(A,P,u)$ be non-toric and locally factorial. 
If $X$ is weakly tropical, then $n_i \geq 2$ 
holds for all $i=0, \ldots , r$.
\end{lemma}

\begin{proof}
Assume that $n_i=1$ holds for some $i$. 
Since $X$ is weakly tropical, 
there exists a cone $\sigma  \in \Sigma_X$ 
of dimension $s+1$ contained in the 
leaf $\lambda_i$. 
Because of $n_i=1$ we have 
$\sigma = \varrho_{i1}+\tau$ 
with a face $\tau \preceq \sigma$ 
such that $\tau \subseteq \lambda$. 
Now, $\sigma = P(\gamma_0^*)$ holds for 
some $\gamma_0 \subseteq \rlv(X)$.
Since the points of $X(\gamma_0)$ are 
factorial, $\sigma$ is a regular cone.
Thus, also $\tau \subseteq \lambda$ 
must be regular.
This implies $l_{i1}=1$, 
contradicting irredundancy of $P$.
\end{proof}

\begin{lemma}
\label{lem:weaklytroprho5}
Let $X=X(A,P,u)$ be non-toric and locally factorial. 
If $X$ is weakly tropical, then $\rho(X)\ge r+3$ holds.
\end{lemma}

\begin{proof}
Lemma \ref{lem:ebcni2} ensures $n_i \geq 2$ for 
all $i=1, \ldots , r$, hence $n \geq 2 \cdot (r+1)$.
The $s$-dimensional lineality space 
$\lambda = \{0\} \times  \QQ^s \subseteq \trop(X)$ 
is a union of cones of $\Sigma_X$.
Thus $P$ must have at least $s+1$ columns 
$v_k$ which means $m \geq s+1$.
Together this yields
$$
\rho(X) 
\ = \ 
n + m - (r-1) - (s+1)
\ \ge \   
r + 3.
$$
\end{proof}

\begin{lemma}
\label{lem:notwtrop}
Let $X=X(A, P, u)$ be non-toric and not weakly 
tropical. 
If $X$ is $\QQ$-factorial, then there is an 
elementary big cone in $\Sigma_X$.
\end{lemma}

\begin{proof}
Since $X$ is not weakly tropical, 
there exists a big cone $\sigma \in \Sigma_X$.
We have $\sigma = P(\gamma_0^*)$ with 
$\gamma_0 \in \rlv(X)$.
Since the points of $X(\gamma_0)$ are 
$\QQ$-factorial, 
the cone $\sigma$ is simplical.
For every $i = 0 \ldots, r$ choose a ray 
$\varrho_i \preceq \sigma$ 
with $\varrho_i \in \lambda_i$.
Then 
$\sigma_0 := \varrho_0 + \ldots + \varrho_r \preceq \sigma$
is as wanted.
\end{proof}

\begin{corollary}
\label{cor:ebc}
Let $X=X(A,P,u)$ be non-toric and locally factorial.
If $\rho(X) \leq 4$ holds, then there exists 
an elementary big cone $\sigma \in \Sigma_X$.
\end{corollary}

Next we investigate the effect of quasismoothness 
on the defining matrix $P$,
where we call $X = X(A,P,u)$ \emph{quasismooth} 
if $\wh{X}$ is smooth.
Thus, quasismoothness means that $X$ has at most 
quotient singularities by quasitori.
The smoothness of $\wh{X}$ will lead 
to conditions on $P$ via the Jacobian 
of the defining relations of $\ol{X}$.

\begin{remark}
Let $(A,P)$ be defining matrices.
Then the Jacobian $J_g$ of 
the defining relations $g_0,\ldots,g_{r-2}$
from Remark~\ref{rem:ci}  
is of the shape $J_g=(J,0)$ with 
a zero block of size $(r-1) \times m$ 
corresponding to the variables 
$S_1, \dots , S_m$ and a block
$$
J 
\ \sei \ 
\left[
\begin{array}{cccccccccc}
\delta_{10} &  \delta_{11} & \delta_{12} & 0 & 
\\
0 & \delta_{21} & \delta _{22} & \delta_{23} & 0
\\
 &  &  &  & & \vdots & 
\\
 &  &  &  & &  & \delta_{r-2,r-3} & \delta_{r-2,r-2} & \delta_{r-2,r-1}&0
\\ 
  &  &  &  & &  & 0& \delta_{r-1,r-2} & \delta_{r-1,r-1} & \delta_{r-1,r}
\\ 
\end{array}
\right],
$$
of size $(r-1) \times n$, 
where each vector $\delta_{a,i}$ is a nonzero multiple of 
the gradient of the monomial $T_i^{l_i}$:
$$
\delta_{a,i} \ = \ 
\alpha_{a,i} 
\left( 
l_{i1} \frac{T_i^{l_i}}{T_{i1}}, 
\ \ldots, \ 
l_{in_i} \frac{T_i^{l_i}}{T_{in_i}}
\right),
\qquad
\alpha_{a,i} \ \in \ \KK^*.
$$
For given $1 \le a,b \le r-1$, $0 \le i \le r$
and $z \in \ol{X}$, 
we have $\delta_{a,i}(z) = 0$ if and only if 
$\delta_{b,i}(z) = 0$.
Moreover, the Jacobian $J_g(z)$ of a point $z\in\ol{X}$ 
is of full rank if and only if 
$\delta_{a,i}(z) = 0$ holds for at most two 
different~$i = 0, \ldots, r$.
\end{remark}

\begin{lemma}
\label{lem:triple1xy}
Assume that $X=X(A,P,u)$ is non-toric 
and that there is an elementary big cone 
$\sigma=\varrho_{0j_0}+\ldots+\varrho_{rj_r}\in\Sigma_X$.
If $X$ is quasismooth, then $l_{ij_i} \ge 2$ holds 
for at most two $i=0,\dots,r$.
\end{lemma}

\begin{proof}
We have $\sigma = P(\gamma_0^*)$ with a relevant face 
$\gamma_0 \in \rlv(X)$.
Since $X$ is quasismooth, any $z \in \ol{X}(\gamma_0)$ 
is a smooth point of~$\ol{X}$. 
Thus, $J_g(z)$ is of full rank $r-1$.
Consequently, $\delta_{a,i}(z) = 0$ holds 
for at most two different~$i$.
This means $l_{ij_i} \ge 2$ for at most 
two different~$i$.
\end{proof}

\begin{corollary}
\label{cor:qg-bigtower}
Let $X = X(A,P,u)$ be non-toric and quasismooth.
If there is an elementary big cone in $\Sigma_X$,
then $n_i=1$ holds for at most two different 
$i=0, \dots ,r$.
\end{corollary}

\begin{lemma}
\label{lem:sampleFF}
Let $(A,P)$ be defining matrices.
Consider the rays $\gamma_{k} := \cone(e_k)$
and $\gamma_{ij} := \cone(e_{ij})$ 
of the orthant $\gamma \subseteq \QQ^{r+s}$ 
and the twodimensional faces
$$
\gamma_{k_1,k_2} \ := \gamma_{k_1} + \gamma_{k_2},
\quad
\gamma_{ij,k} := \gamma_{ij} + \gamma_{k},
\quad
\gamma_{i_1j_1,i_2j_2} := \gamma_{i_1j_1} + \gamma_{i_2j_2}.
$$
\begin{enumerate}
\item 
All $\gamma_k$, resp.~$\gamma_{k_1,k_2}$, 
are $\mathfrak{F}$-faces and 
each $\ol{X}(\gamma_{k})$, resp.~$\ol{X}(\gamma_{k_1,k_2})$,
consists of singular points of $\ol{X}$.
\item 
A given $\gamma_{ij}$, resp.~$\gamma_{ij,k}$,
is an $\mathfrak{F}$-face if and only if 
$n_i \ge 2$ holds. 
In that case, $\ol{X}(\gamma_{ij})$, resp.~$\ol{X}(\gamma_{ij,k})$,
consists of smooth points of $\ol{X}$ 
if and only if $r=2$, $n_i =2$ and $l_{i,3-j} = 1$ hold.
\item 
A given $\gamma_{ij_1,ij_2}$ with $j_1 \ne j_2$ 
is an $\mathfrak{F}$-face if and only if 
$n_i \ge 3$ holds.
In that case, $\ol{X}(\gamma_{ij_1,ij_2})$ consists of 
smooth points of $\ol{X}$ if and only if
$r=2$, $n_i=3$ and $l_{ij} = 1$ for the $j \ne j_1,j_2$ hold.
\item 
A given $\gamma_{i_1j_1,i_2j_2}$ with $i_1\neq i_2$ is 
an $\ff$-face if and only if we have 
$n_{i_1}, n_{i_2}\ge 2$ or $n_{i_1}=n_{i_2}=1$ and $r=2$.
In the former case, $\ol{X}(\gamma_{i_1j_1,i_2j_2})$ consists 
of smooth points of $\ol{X}$ if and only if 
one of the following holds:
\begin{itemize}
\item
$r=2$, $n_{i_t}=2$ and $l_{i_t,3-j_t} = 1$ for a $t\in\{1,2\}$,
\item
$r=3$, $n_{i_1}=n_{i_2}=2$, $l_{i_1,3-j_1}=l_{i_2,3-j_2} = 1$.
\end{itemize}
\end{enumerate}
\end{lemma}

\begin{proof}
The statements follow directly from the 
structure of the defining relations 
$g_0, \ldots, g_{r-2}$ of $R(A,P)$ 
and the shape of the Jacobian $J_g$.
\end{proof}

We now restrict to the case that the 
rational divisor class group 
$\Cl(X)_\QQ = K_\QQ$ of $X = X(A,P,u)$ 
is of dimension two. 
Set $\tx:=\Ample(X)$.
Then the effective cone $\Eff(X)$ is 
of dimension two and is uniquely 
decomposed into three convex sets
$$
\Eff(X) 
\ = \ 
\tp \cup \tx \cup \tm,
$$
such that $\tp, \tm$ do not intersect
the ample cone $\tx$ and 
$\tp \cap \tm$ consists of the origin.
Recall that $u \in \tx$ holds and
that, due to $\tx \subseteq \Mov(X)$,
each of $\tp$ and~$\tm$ 
contains at least two of the weights
$w_{ij},w_k$.
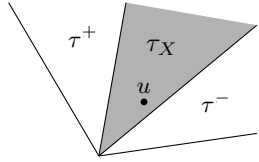
\begin{center}
    \begin{tikzpicture}[scale=0.6]
    \path[fill=gray!60!] (0,0)--(3.5,2.9)--(0.6,3.4)--(0,0);
    \path[fill, color=black] (1.4,2.4) circle (0.0ex)  node[]{$\tx$};
     \path[fill, color=black] (1,1.2) circle (0.5ex)  node[]{};
    \path[fill, color=black] (1,1.5) circle (0.0ex)  node[]{\small{$u$}};
    \draw (0,0)--(0.6,3.4);
    \draw (0,0) --(-2,3.4);
    \path[fill, color=black] (-0.35,2.65) circle (0.0ex)  node[]{\small{$\tau^+$}};
    \draw (0,0)  -- (3.5,2.9);
    \draw (0,0)  -- (3.5,0.5);
    \path[fill, color=black] (2.6,1.2) circle (0.0ex)  node[]{\small{$\tau^-$}};
    \path[fill, color=white] (4,1.9) circle (0.0ex);
  \end{tikzpicture}   
\end{center}

\begin{remark}
\label{rem:projFF}
Consider $X = X(A,P,u)$ such that 
$\Cl(X)_\QQ$ is of dimension two.
Then, for every $\mathfrak{F}$-face
$\{0\} \ne \gamma_0 \preceq \gamma$
precisely one of the following 
inclusions holds
$$
Q(\gamma_0) \ \subseteq \ \tp,
\qquad
\tx \ \subseteq \ Q(\gamma_0)^\circ,
\qquad
Q(\gamma_0) \ \subseteq \ \tm.
$$
The $\mathfrak{F}$-faces
$\gamma_0 \preceq \gamma$
satisfying the second inclusion 
are exactly those with $\gamma_0 \in \rlv(X)$,
i.e., the relevant ones.
\end{remark}

\begin{lemma}
\label{lem:tau}
Let $X = X(A, P, u)$ be non-toric with $\rk(\Cl(X))=2$. 
\begin{enumerate}
\item 
Suppose that $X$ is $\QQ$-factorial.
Then $w_{k} \notin \tx$ holds for all $1\le k \le m$
and for all $0\le i \le r$ with $n_i\ge2$ 
we have $w_{ij} \notin \tx$, where 
$1 \leq j \leq n_i$.
\item 
Suppose that $X$ is quasismooth, $m > 0$ 
holds and there is
$0 \le i_1 \le r$ with $n_{i_1} \ge 3$. 
Then the $w_{ij}, w_k$ with $n_i \ge 3$, 
$j=1, \ldots , n_i$ and $k=1, \ldots , m$
lie either all in $\tp$ or all in $\tm$.
\item 
Suppose that $X$ is quasismooth and there 
is $0 \le i_1 \le r$ with $n_{i_1} \ge 4$. 
Then the $w_{ij}$ with $n_i \geq 4$ and 
$j=1, \ldots , n_i$
lie either all in $\tp$ or all in $\tm$.
\item
Suppose that $X$ is quasismooth and there 
exist $0 \le i_1 < i_2 \le r$ with 
$n_{i_1}, n_{i_2} \ge 3$.
Then the $w_{ij}$ with $n_i \ge 3$, 
$j=1, \ldots , n_i$ lie either all 
in $\tp$ or all in $\tm$.
\item
Suppose that $X$ is quasismooth.
Then $w_1, \ldots, w_m$ lie
either all in $\tp$ or all in $\tm$.
\end{enumerate} 
\end{lemma}

\begin{proof}
We prove~(i).
By Lemma~\ref{lem:sampleFF}~(i) and (ii),
the rays $\gamma_{k}, \gamma_{ij} \preceq \gamma$
with $n_i \ge 2$  
are $\mathfrak{F}$-faces.
Since $X$ is $\QQ$-factorial, the 
ample cone $\tx \subseteq K_\QQ$
of $X$ is of dimension two
and thus
$\tx \subseteq Q(\gamma_{ij})^\circ$
or  
$\tx \subseteq Q(\gamma_{k})^\circ$
is not possible.
Remark~\ref{rem:projFF} yields 
the assertion.

We turn to~(ii). 
By Lemma~\ref{lem:sampleFF}~(i) and~(ii),
all $\gamma_{k}, \gamma_{ij}, \gamma_{ij,k} \preceq \gamma$
in question are $\mathfrak{F}$-faces and 
the corresponding pieces 
in $\ol{X}$ consist of singular points. 
Because $X$ is quasismooth, none of 
these $\mathfrak{F}$-faces is relevant.
Thus, Remark~\ref{rem:projFF} gives
$w_{i_11} \in \tp$ or $w_{i_11} \in \tm$;
say we have $w_{i_11} \in \tp$.
Then, applying again Remark~\ref{rem:projFF},
we obtain $w_k,w_{ij} \in \tp$ for 
$k = 1,\ldots, m$, 
all $i$ with $n_i \ge 3$ and $j=1,\ldots,n_i$.

Assertion~(iii) is proved analogously: 
treat first $\gamma_{i_11,i_12}$ with
Lemma~\ref{lem:sampleFF}~(iii),
then $\gamma_{i_11,ij}$ with
Lemma~\ref{lem:sampleFF}~(iii) and~(iv).
Similarly, we obtain~(iv) by treating  
first $\gamma_{i_11,i_21}$ and then all
$\gamma_{i_11,ij}$ and $\gamma_{i_21,ij}$
with Lemma~\ref{lem:sampleFF}~(iii) 
and~(iv).
Finally, we obtain~(v) 
using Lemma~\ref{lem:sampleFF}~(i).
\end{proof}

\begin{proposition}
\label{prop:4-1-m-pos}
Let $X = X(A,P,u)$ be non-toric, 
quasismooth and $\QQ$-factorial 
with $\rho(X)=2$.
Assume that there is an elementary 
big cone in $\Sigma_X$ and 
that we have $n_0 \ge  \ldots \ge n_r$.
If $m > 0$ holds, then there is 
a $\gamma_{ij,k} \in \rlv(X)$,
we have $r=2$ and the constellation 
of the $n_i$ is 
$(n_0,2,2)$,
$(2,2,1)$
or
$(2,1,1)$.
\end{proposition}

\begin{proof}
According to Lemma~\ref{lem:tau}~(v),
we may assume $w_1, \ldots, w_m \in \tp$.
We claim that there is a
$w_{i_1j_1} \in \tm$ with $n_{i_1} \ge 2$.
Otherwise, use Corollary~\ref{cor:qg-bigtower}
to see that there exist $w_{ij}$ with $n_i \ge 2$ 
and Lemma~\ref{lem:tau}~(i) to see that they 
all lie in $\tp$.
Since all monomials $T_{i}^{l_i}$ have 
the same degree in $K$, we obtain 
in addition $w_{i1} \in \tp$ for all $i$ 
with $n_i=1$. 
But then no weights $w_{ij}, w_k$ are 
left to lie in $\tm$, a contradiction.

Having verified the claim, we may take
a $w_{i_1j_1} \in \tm$ with $n_{i_1} \ge 2$.
Then $\gamma_{i_1j_1,1} \in \rlv(X)$ is as 
desired.
Moreover, Lemma~\ref{lem:sampleFF}~(ii)
yields $r=2$ and $n_{i_1}=2$.
If $n_0 \ge 3$ holds, then 
Lemma~\ref{lem:tau}~(ii) gives 
$w_{ij} \in \tp$ for all $i$ 
with $n_i \ge 3$.
Moreover, as all $T_{i}^{l_i}$ share 
the same $K$-degree, we have
$w_{i1} \in \tp$ for all $i$ 
with $n_i=1$. 
By the same reason, 
one of the $w_{i_11}$, $w_{i_12}$ 
must lie in $\tp$.
As $\tm$ contains at least two 
weights, there is a 
$w_{i_2j_2} \in \tm$  with $n_{i_2} = 2$
and $i_1 \ne i_2$.
Thus, the constellation of 
$n_0 \ge n_1 \ge n_2$ is as claimed.
\end{proof}

\begin{proposition}
\label{prop:4-1-m-0}
Let $X = X(A,P,u)$ be non-toric, 
quasismooth and $\QQ$-factorial 
with $\rho(X)=2$.
Assume that there is an elementary 
big cone in $\Sigma_X$ and 
that we have $n_0 \ge  \ldots \ge n_r$.
If $m = 0$ holds, then  there is 
a $\gamma_{i_1j_1,i_2j_2} \in \rlv(X)$,
we have $r \le 3$ and the constellation 
of the $n_i$ is one of the following
$$ 
\begin{array}{lcl}
r = 2 \colon & & (n_0,2,2), \ (3,2,1), \ (3,1,1), \
(2,2,2), \ (2,2,1), 
\\
r = 3 \colon & & (2,2,2,2), \ (2,2,2,1), \ (2,2,1,1).
\end{array}
$$
\end{proposition}

\begin{proof}
We first show $n_1 \le 2$.
Otherwise, we had $n_1 \ge 3$.
Then, according to Lemma~\ref{lem:tau}~(iv),
we may assume that all the $w_{ij}$ 
with $n_i \ge 3$ lie in $\tp$.
In particular, $w_{11}$, lies in 
$\tp$.
Because all monomials $T_{i}^{l_i}$ have 
the same degree in $K$,  
also  $w_{i1} \in \tp $ holds for all $i$ 
with $n_i=1$.
At least two weights $w_{i_1j_1}$ and $w_{i_2j_2}$
must belong to $\tm$.
For these, only $n_{i_1} = n_{i_2} = 2$ and 
$i_1 \ne i_2$ is possible.
Applying Lemma~\ref{lem:sampleFF}~(iv)
to $\gamma_{11,i_1j_1} \in \rlv(X)$ 
gives $r = 2$, contradicting $n_0\ge n_1 \ge 3$ and 
$n_{i_1} = n_{i_2} = 2$.

We treat the case $n_0 \ge 4$.
By Lemma~\ref{lem:tau}~(iii), we can
assume $w_{01}, \ldots, w_{0n_0} \in \tp$.
As before, we obtain $w_{i1} \in \tp$ 
for all $i$ with $n_i=1$
and we find two weights 
$w_{i_1j_1}, w_{i_2j_2} \in \tm$ 
with $n_{i_1} = n_{i_2} = 2$ and 
$i_1 \ne i_2$.
Then $\gamma_{01,i_1j_1} \in \rlv(X)$ is 
as wanted.
Lemma~\ref{lem:sampleFF}~(iv) gives
$r = 2$ and we end up with 
$(n_0,2,2)$.

Now let $n_0=3$.
Lemma~\ref{lem:tau}~(i) guarantees 
that no $w_{0j}$ lies in $\tx$.
If weights $w_{0j}$ occur in both cones
$\tp$ and $\tm$, say 
$w_{01} \in \tp$ and $w_{02} \in \tm$,
then $\gamma_{01,02}$ is as wanted.
Lemma~\ref{lem:sampleFF}~(iii) yields
$r=2$ and we obtain the constellations
$(n_0,2,2)$, $(3,2,1)$ and $(3,1,1)$.
So, assume that all weights $w_{0j}$ 
lie in one of $\tp$ and $\tm$, 
say in~$\tp$.
Then we proceed as in the case 
$n_0 \ge 4$ to obtain a
$\gamma_{01,i_1j_1} \in \rlv(X)$
and $r=2$ with the constellation
$(3,2,2)$.

Finally, let $n_0 \le 2$. 
Corollary~\ref{cor:qg-bigtower}
yields $n_0 = 2$.
According to Lemma~\ref{lem:tau}~(i) 
no $w_{ij}$ with $n_i=2$ lies in $\tx$.
So, we may assume $w_{01} \in \tp$.
Moreover, all $w_{ij}$ with $n_i=1$ 
lie together in one $\tp$, $\tx$ 
or in $\tm$.
Since each of $\tp$ and $\tm$
contains two weights, 
we obtain  $n_1=2$ and some 
$\gamma_{0j_1,1j_2}$ is as wanted.
Lemma~\ref{lem:sampleFF}~(iv) shows
$r \le 3$.   
\end{proof}

We retrieve a special case of~\cite[Cor. 4.18]{Deb}.

\begin{corollary}
\label{cor:clfree}
Let $X = X(A,P,u)$ be smooth with $\rho(X)=2$.
Then the divisor class group $\Cl(X)$ is 
torsion-free.
\end{corollary}

\begin{proof}
By Corollary~\ref{cor:ebc}, there is an elementary big 
cone in $\Sigma_X$.
Thus, Propositions~\ref{prop:4-1-m-pos} 
and~\ref{prop:4-1-m-0} deliver a twodimensional
$\gamma_0 \in \rlv(X)$.
The corresponding weights generate $K$ as a group.
This gives $\Cl(X) \cong K \cong \ZZ^2$.
\end{proof}

\begin{proof}[Proof of Proposition~\ref{prop:smooth-rho2}]
The variety $X$ is isomorphic to some $X(A,P,u)$, 
where after suitable admissible operations
we may assume $n_0 \ge \ldots \ge n_r$.
Thus, Propositions~\ref{prop:4-1-m-pos} 
and~\ref{prop:4-1-m-0} apply.
\end{proof}


\section{Proof of Theorems~\ref{thm:main1}, 
\ref{thm:main2} and~\ref{thm:main3}}
\label{sec:classif}

We prove Theorems~\ref{thm:main1},~\ref{thm:main2}
and~\ref{thm:main3} by going through
the cases established in 
Proposition~\ref{prop:smooth-rho2}.
The notation is the same as in 
Sections~\ref{sec:cpl1} and~\ref{sec:firstStruct}.
So, we deal with a smooth projective  
variety $X = X(A,P,u)$ of Picard number 
$\rho(X) = 2$ coming with an effective torus 
action of complexity one.

From Corollary~\ref{cor:clfree} we know that 
$\Cl(X) = K =  \ZZ^2$ holds.
With $w_{ij} = Q(e_{ij})$ and $w_k = Q(e_k)$,
the columns of the $2 \times (n+m)$ 
degree matrix $Q$ will be written as 
$$
w_{ij} \ = \ (w_{ij}^1,w_{ij}^2) \ \in \ \ZZ^2,
\qquad\qquad
w_{k} \ = \ (w_{k}^1,w_{k}^2) \ \in \ \ZZ^2.
$$
Recall that all relations $g_0, \ldots, g_{r-2}$ 
of $R(A,P)$ have the same degree in $K = \ZZ^2$; 
we set for short
$$
\mu 
\ = \
(\mu^1, \mu^2) 
\ := \ 
\deg(g_0) 
\  \in \
\ZZ^2.
$$
We will frequently work with the faces of the 
orthant $\gamma = \QQ_{\ge 0}^{n+m}$ 
introduced in Lemma~\ref{lem:sampleFF}:
$$ 
\gamma_{ij,k} \ = \ \cone(e_{ij},e_k) \ \preceq \ \gamma,
\qquad
\gamma_{i_1j_1,i_2j_2} \ = \ \cone(e_{i_1j_1},e_{i_2j_2}) \ \preceq \ \gamma.
$$

\begin{remark}
\label{rem:Q}
Consider a face $\gamma_0 \preceq \gamma$ 
of type $\gamma_{ij,k}$ or $\gamma_{i_1j_1,i_2j_2}$.
Write $e'$, $e''$ for the two generators of $\gamma_0$ 
and  $w' = Q(e')$, $w'' = Q(e'')$ for the corresponding 
columns of the degree matrix $Q$ such that 
$(w',w'')$ is positively oriented in $\ZZ^2$.
Then Remark~\ref{rem:Qfact} tells us
\begin{eqnarray*}
\gamma_0 \ \in \ \rlv(X)
& \Rightarrow &  
\det(w',w'') \ = \ 1.
\end{eqnarray*}
So, if $\gamma_0 \in \rlv(X)$, then we 
may multiply  $Q$ from 
the left with a unimodular \mbox{$2 \times 2$} 
matrix transforming $w'$ and $w''$ 
into $(1,0)$ and $(0,1)$.
This change of coordinates on $\Cl(X)$ 
does not affect the defining data $(A,P)$.
If $w' = (1,0)$ and $w'' = (0,1)$
hold and $e \in \gamma$ is a
canonical basis vector
with corresponding column $w = Q(e)$, 
then we have 
\begin{align*}
\cone(e',e) \in \rlv(X) 
\quad &\Rightarrow \quad
w  =  (w^1,1),
\\
\cone(e'',e) \in \rlv(X) 
\quad &\Rightarrow \quad
w  =  (1,w^2).
\end{align*}
\end{remark}

We are ready to go through the 
cases of Proposition~\ref{prop:smooth-rho2};
we keep the numbering introduced there.

\begin{case}
We have $r=2$, $m \ge 0$ and the list of $n_i$ is 
$(n_0,2,2)$, where $n_0 \ge 3$. 
This leads to No.~1 and 
No.~2 in Theorems~\ref{thm:main1} and~\ref{thm:main2}.
\end{case}

\begin{proof}
In a first step we show that there occur 
weights $w_{0j}$ in each of $\tp$ and $\tm$.
Otherwise, we may assume that all $w_{0j}$ 
lie in $\tp$, see Lemma~\ref{lem:tau}~(i).
Then Lemma~\ref{lem:tau}~(ii) says that 
also all $w_k$ lie in $\tp$.
Moreover, we have $\deg(T_i^{l_i}) \in \tp$ 
for $i=0,1,2$.
Thus, we may assume $w_{11},w_{21} \in \tp$ 
and obtain $w_{12},w_{22} \in \tm$, as there 
must be at least two weights in $\tm$.
Finally, we may assume that 
$\cone(w_{01},w_{12})$ contains
$w_{02},\ldots,w_{0n_0}$ and 
$w_{22}$.
Applying Remark~\ref{rem:Q}
first to $\gamma_{01,12}$, then to 
all $\gamma_{0j,12}$, $\gamma_{12,k}$
and $\gamma_{01,22}$, $\gamma_{12,21}$ 
yields
$$
{
Q \ = \ \left[ 
\begin{array}{cccc|cc|cc|ccc}
0 & w_{02}^1 & \dots & w_{0n_0}^1 & w_{11}^1 & 1 & w_{21}^1 & 1 & w_1^1 & \dots & w_m^1
\\
1 & 1 & \dots & 1 & w_{11}^2 & 0 & 1 & w_{22}^2 & 1 & \dots  & 1 
\end{array}
\right],
}
$$
where $w_{0j}^1 \geq 0$ and $w_{22}^2 \geq 0$.
Since $\gamma_{01,12}, \gamma_{01,22} \in \rlv(X)$ holds,
Lemma~\ref{lem:sampleFF}~(iv) implies $l_{11}=l_{21}=1$.
Applying $P \cdot Q^t = 0$ to the first row of $P$ 
and the second row of $Q$ gives
$$
0
\ < \ 
3
\ \le \ 
n_0
\ \le \ 
l_{01} + \ldots + l_{0n_0}
\ = \ 
w_{11}^2
\ = \ 
1 + w_{22}^2w_{11}^1,
$$
where the last equality is due to 
$\gamma_{11,22} \in \rlv(X)$ and thus
$\det(w_{22},w_{11})=1$.
We conclude $w^2_{22} > 0$ and $w^1_{11} > 0$.
Because of $\gamma_{0j,22} \in \rlv(X)$,
we obtain $\det(w_{22},w_{0j})=1$.
This implies $w_{0j}^1 = 0$ for all 
$j=2,\ldots,n_0$.
Applying $P \cdot Q^t = 0$ to the first row of~$P$ 
and the first row of $Q$ gives
$w^1_{11} + l_{12} = 0$; a contradiction.

Knowing that each of $\tp$ and $\tm$ contains
weights $w_{0j}$, 
we can assume $w_{01}, w_{02} \in \tp$ and $w_{03} \in \tm$.
Lemma~\ref{lem:tau}~(ii) and~(iii) show
$n_0=3$ and $m=0$.
There is at least one other weight in $\tm$, 
say $w_{11} \in \tm$.
Applying Lemma~\ref{lem:sampleFF}~(iii) 
to $\gamma_{0j,03} \in \rlv(X)$ for $j=1,2$
and~(iv) to suitable 
$\gamma_{0j_1,i_2j_2} \in \rlv(X)$,
we obtain
$$ 
l_{01} = l_{02} = 1,
\qquad
l_{11} = l_{12} = 1,
\qquad
l_{21} = l_{22} = 1.
$$
Moreover, Remark~\ref{rem:Q} 
applied to $\gamma_{01,03}$ 
as well as $\gamma_{02,03}$ and $\gamma_{01,11}$
brings the matrix~$Q$ into the shape
$$
{
Q 
\ = \ 
\left[ 
\begin{array}{ccc|cc|cc}
0 & w_{02}^1 & 1 & 1 & w_{12}^1 & w_{21}^1 & w_{22}^1
\\
1 & 1 & 0 & w_{11}^2 & w_{12}^2 & w_{21}^2 & w_{22}^2
\end{array}
\right].
}
$$
Observe that the second component of the 
degree of the relation is $\mu^2 = 2$.
The possible positions of the weights $w_{2j}$ 
define three subcases:

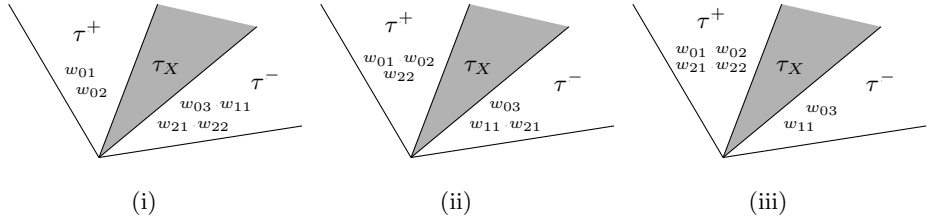
\vspace{0.3cm}
   \begin{tikzpicture}[scale=0.6]
    \path[fill=gray!60!] (0,0)--(3.5,2.9)--(1.3,3.4)--(0,0);
    \path[fill, color=black] (1.5,2) circle (0.0ex)  node[]{\small{$\tx$}};
    \path[fill, color=black] (-0.4,1.9) circle (0.0ex)  node[]{\tiny{$w_{01}$}};
    \path[fill, color=black] (-0.15,1.45) circle (0.0ex)  node[]{\tiny{$w_{02}$}};
    \path[fill, color=black] (-0.25,2.8) circle (0.0ex)  node[]{\small{$\tp$}};
    \draw (0,0)--(1.3,3.4);
    \draw (0,0) --(-2,3.4);
     \path[fill, color=black] (2.6,1.15) circle (0.0ex)  node[]{\tiny{$w_{03} \ w_{11}$}};
    \path[fill, color=black] (2.1,0.7) circle (0.0ex)  node[]{\tiny{$w_{21} \ w_{22}$}};
    \path[fill, color=black] (3.7,1.75) circle (0.0ex)  node[]{\small{$\tm$}};
    \draw (0,0)  -- (3.5,2.9);
    \draw (0,0)  -- (4.5,0.7);
    \path[fill, color=black] (1,-1) circle (0.0ex)  node[]{\small{(i)}};
  \end{tikzpicture}
    \
  \begin{tikzpicture}[scale=0.6]
    \path[fill=gray!60!] (0,0)--(3.5,2.9)--(1.3,3.4)--(0,0);
    \path[fill, color=black] (1.5,2) circle (0.0ex)  node[]{\small{$\tx$}};
    \path[fill, color=black] (-0.25,2.15) circle (0.0ex)  node[]{\tiny{$w_{01} \ w_{02}$}};
    \path[fill, color=black] (-0.25,1.8) circle (0.0ex)  node[]{\tiny{$w_{22}$}};
    \path[fill, color=black] (-0.25,3) circle (0.0ex)  node[]{\small{$\tp$}};
    \draw (0,0)--(1.3,3.4);
    \draw (0,0) --(-2,3.4);
    \path[fill, color=black] (2.1,1.15) circle (0.0ex)  node[]{\tiny{$w_{03}$}};
    \path[fill, color=black] (2.1,0.7) circle (0.0ex)  node[]{\tiny{$w_{11} \ w_{21}$}};
    \path[fill, color=black] (3.5,1.7) circle (0.0ex)  node[]{\small{$\tm$}};
    \draw (0,0)  -- (3.5,2.9);
    \draw (0,0)  -- (4.5,0.7);
    \path[fill, color=black] (1,-1) circle (0.0ex)  node[]{\small{(ii)}};
  \end{tikzpicture}
    \
  \begin{tikzpicture}[scale=0.6]
    \path[fill=gray!60!] (0,0)--(3.5,2.9)--(1.3,3.4)--(0,0);
    \path[fill, color=black] (1.5,2) circle (0.0ex)  node[]{\small{$\tx$}};
    \path[fill, color=black] (-0.25,2.35) circle (0.0ex)  node[]{\tiny{$w_{01} \ w_{02}$}};
    \path[fill, color=black] (-0.25,2) circle (0.0ex)  node[]{\tiny{$w_{21} \ w_{22}$}};
    \path[fill, color=black] (-0.25,3.1) circle (0.0ex)  node[]{\small{$\tp$}};
    \draw (0,0)--(1.3,3.4);
    \draw (0,0) --(-2,3.4);
    \path[fill, color=black] (2.2,1.05) circle (0.0ex)  node[]{\tiny{$w_{03}$}};
    \path[fill, color=black] (1.7,0.7) circle (0.0ex)  node[]{\tiny{$w_{11}$}};
    \path[fill, color=black] (3.5,1.7) circle (0.0ex)  node[]{\small{$\tm$}};
    \draw (0,0)  -- (3.5,2.9);
    \draw (0,0)  -- (4.5,0.7);
    \path[fill, color=black] (1,-1) circle (0.0ex)  node[]{\small{(iii)}};
  \end{tikzpicture}

\noindent
We will see that cases~(i) and~(ii) give No.~1 and No.~2 
of Theorem~\ref{thm:main1} respectively
and case~(iii) will not provide any smooth variety.

In~(i) we assume $w_{21},w_{22}\in\tm$.
Then $\gamma_{01,21}, \gamma_{01,22} \in \rlv(X)$ 
holds and Remark~\ref{rem:Q} shows 
$w_{21}^1=w_{22}^1=1$.
This implies $\mu^1=2$.
Similarly, considering $\gamma_{02,21}, \gamma_{02,22} \in \rlv(X)$,
we obtain $w_{02}^1=0$ or $w_{21}^2=w_{22}^2=0$.
The latter contradicts $\mu^2=2$ and thus
$w_{02}^1=0$ holds.
We conclude $l_{03}=\mu^1=2$.
Furthermore $w_{12}^1=\mu^1-w_{11}^1=1$.
Together, we have
$$
g_0 
\ = \ 
T_{01}T_{02}T_{03}^2+T_{11}T_{12}+T_{21}T_{22},
\qquad
Q 
\ = \ 
\left[ 
\begin{array}{ccc|cc|cc}
0 & 0 & 1 & 1 & 1 & 1 & 1 
\\ 
1 & 1 & 0 & a & 2-a & b & 2-b
\end{array}
\right],
$$
where $a, b \in \ZZ$.
Observe that $w_{12} \in \tm$
must hold;
otherwise, $\gamma_{03,12} \in \rlv(X)$ 
and Remark~\ref{rem:Q} yields $w_{12}^2=1$,
contradicting $w_{12}=(1,1)=w_{11}\in\tm$.
The semiample cone is 
$\SAmple(X) = \cone((0,1), (1,d))$, 
where $d=\max(a, 2-a, b, 2-b)$.
The anticanonical class is $-\mathcal{K}_X=(3,4)$.
Hence $X$ is an almost Fano variety if and only if $d=1$,
which is equivalent to $a=b=1$.
In this situation $X$ is already a Fano variety.

In (ii) we assume $w_{21} \in \tm$ and $w_{22} \in \tp$.
Remark~\ref{rem:Q}, applied to 
$\gamma_{01,21},  \gamma_{03,22} \in \rlv(X)$
shows $w_{21}^1 = w_{22}^2 = 1$.
The latter implies $w_{21}^2= \mu^2 -w_{22}^2=1$.
We claim $w_{11}^2 \ne 0$.
Otherwise, we have $w_{12}^2 = \mu^2 = 2$.
This gives $\det(w_{03}, w_{12})=2$.
We conclude $\gamma_{03,12} \not\in \rlv(X)$
and $w_{12} \in \tm$.
Then $\gamma_{01,12} \in \rlv(X)$ 
implies $w_{12}^1=1$.
Thus, $w_{22}=(1,1)$ and $w_{12}=(1,2)$ hold,
contradicting $w_{22} \in  \tp$ 
and $w_{12} \in  \tm$.
Now, $\gamma_{11,22} \in \rlv(X)$ yields
$w_{11}^2w_{22}^1=0$ and thus $w_{22}^1=0$. 
We obtain $\mu^1=1$ and, as a consequence
$l_{03}=1, w_{02}^1=0$ and $w_{12}^1=0$. 
Therefore $w_{12} \in \tp$ holds.
Now $\gamma_{03,12}\in\rlv(X)$ implies 
$w_{12}^2=1$ and $w_{11}^2=\mu^2-w_{12}^2=1$.
We arrive at
$$
g_0 
\ = \ 
T_{01}T_{02}T_{03}+T_{11}T_{12}+T_{21}T_{22},
\qquad
Q 
\ = \  
\left[ 
\begin{array}{ccc|cc|cc}
0 & 0 & 1 & 1 & 0 & 1 & 0 \\ 1 & 1 & 0 & 1 & 1 & 1 & 1
\end{array}
\right].
$$ 
The anticanonical class is $-\mathcal{K}_X=(2,4)$
and the semiample cone is $\SAmple(X) = \cone((0,1), (1,1))$.
In particular $X$ is Fano.

We turn to (iii), where both $w_{21}$ 
and $w_{22}$ lie in $\tp$.
The homogeneity of $g_0$ yields 
$w_{12} \in \tp$.
Thus, $\gamma_{03,12},\gamma_{03,21},\gamma_{03,22} \in \rlv(X)$ 
holds and Remark~\ref{rem:Q} implies 
$w_{12}^2=w_{21}^2=w_{22}^2=1$.
We conclude $w_{11}^2 = \mu^2-w_{12}^2=1$. 
Similarly, $\gamma_{02,11}, \gamma_{11, 21}, \gamma_{11, 22} \in \rlv(X)$
yields $w_{02}^1=w_{21}^1= w_{22}^1=0$.
This gives
$0 \neq l_{03} = \mu^1=w_{21}^1+ w_{22}^1=0$
which is not possible.
\end{proof}

\begin{case}
We have $r=2$, $m = 0$, $n = 6$
and the list of $n_i$ is $(3,2,1)$.
This leads to No. $3$ in 
Theorems~\ref{thm:main1} and~\ref{thm:main2}.
\end{case}

\begin{proof}
Since there are at least two weights in 
$\tp$ and another two in $\tm$,
we can assume $w_{01},w_{02}\in\tp$ 
and $w_{03},w_{12}\in\tm$.
By Lemma~\ref{lem:sampleFF}~(iii) and~(iv) 
we obtain $l_{01}=l_{02}=l_{11}=l_{12}=1$.
We may assume that $\cone(w_{01},w_{03})$ 
contains $w_{02}$.
Applying Remark~\ref{rem:Q} firstly to $\gamma_{01,03}$,
then to $\gamma_{02,03}$ and $\gamma_{01,12}$,
we obtain
$$
Q \ = \ \left[ 
\begin{array}{ccc|cc|c}
0 & w_{02}^1 & 1 & w_{11}^1 & 1 & w_{21}^1 
\\ 
1 & 1 & 0 & w_{11}^2 & w_{12}^2 & w_{21}^2
\end{array}
\right],
$$
where $w_{02}^1\ge0$.
For the degree $\mu$ of $g_0$, we 
have $\mu^2 = 2$. 
We conclude $w_{11}^2 = 2-w_{12}^2$ 
and 
$l_{21}w_{21}^2 = 2$ 
which in turn 
implies $l_{21}=2$ and $w_{21}^2=1$.
For $\gamma_{02,12} \in \rlv(X)$,
Remark~\ref{rem:Q} gives 
$\det(w_{12}, w_{02}) = 1$
and thus $w_{02}^1=0$ or $w_{12}^2=0$
must hold.

We treat the case $w_{02}^1=0$.
Then $\mu=(l_{03},2)$ holds.
We conclude $w_{11}^1=l_{03}-1$
and $w_{21}^1=l_{03}/2$.
With $c := l_{03}/2  \in\ZZ_{\ge1}$
and $a := w_{12}^2 \in \ZZ$, we 
obtain the degree matrix
$$
Q \ = \ 
\left[ 
\begin{array}{ccc|cc|c}
0 & 0 & 1 & 2c-1 & 1 & c 
\\ 1 & 1 & 0 & 2-a & a & 1
\end{array}
\right].
$$
We show $w_{11} \in \tm$.
Otherwise, $w_{11}\in\tp$ holds,
we have $\gamma_{03,11} \in \rlv(X)$  
and Remark~\ref{rem:Q} yields $a=1$.
But then $w_{01} = (0,1) \in \tp$ and
$w_{11} = (2c-1,1) \in \tp$ 
imply $w_{12}=(1,1) \in\tp$;
a contradiction.
So we have $w_{11}\in\tm$.
Then $\gamma_{01,11}\in\rlv(X)$ holds.
Remark~\ref{rem:Q} gives 
$\det(w_{11},w_{01}) = 1$ 
which means $c=1$ and,
as a consequence, $l_{03}=2$.
Together, we have
$$
g_0 
\ = \ 
T_{01}T_{02}T_{03}^2+T_{11}T_{12}+T_{21}^2,
\qquad
Q \ = \ \left[ 
\begin{array}{ccc|cc|c}
0 & 0 & 1 & 1 & 1 & 1 
\\ 
1 & 1 & 0 & 2-a & a & 1
\end{array}
\right],
$$
where we may assume $a \ge 2-a$ that means 
$a \in \ZZ_{\ge 1}$.
The semiample cone is $\SAmple(X)=\cone((0,1),(1,a))$,
and the anticanonical class is $-\mathcal{K}_X=(2,3)$.
In particular, $X$ is an almost Fano variety if and only 
$a=1$ holds.
In this situation $X$ is already a Fano variety.

We turn to the case $w_{12}^2=0$. 
Here, $w_{11}^2=\mu^2=2$ leads to
$\det(w_{03},w_{11})=2$ and 
thus the $\mathfrak{F}$-face 
$\gamma_{03,11}$ does not belong 
to $\rlv(X)$; see Remark~\ref{rem:Q}.
Hence $w_{11} \in \tm$ and thus 
$\gamma_{01,11}\in\rlv(X)$.
This gives $w_{11}^1=1$ and thus
$w_{11}=(1,2)$.
Because of $w_{02}=(w_{02},1)\in\tp$,
we must have $w_{02}^1=0$ and
the previous consideration applies.
\end{proof}

\begin{case}
We have $r=2$, $m = 0$, $n = 5$
and the list of $n_i$ is $(3,1,1)$.
This case does not provide smooth 
varieties.
\end{case}

\begin{proof}
Each of $\tp$ and $\tm$ contains at least
two weights.
We may assume $w_{01},w_{02}\in\tp$ and 
$w_{03},w_{11},w_{21}\in\tm$.
Then $\gamma_{01,03},\gamma_{02,03}\in\rlv(X)$
holds and Lemma~\ref{lem:sampleFF}~(iii) 
yields $l_{01}=l_{02}=1$.
By Remark~\ref{rem:Q} we can assume 
$w_{03}=(1,0)$ and $w_{01}^2=w_{02}^2=1$.
This implies $\mu^2=2$ and, as a consequence,
$l_{11}=l_{21}=2$.
By~\cite[Thm.~1.1]{HaHe:2013}, we have 
torsion in $\Cl(X)$; a contradiction 
to Corollary~\ref{cor:clfree}.
\end{proof}

\begin{case}
\label{case:d}
We have $r=2$, $m\ge 0$, $n = 6$
and the list of $n_i$ is $(2,2,2)$.
Suitable admissible operations lead
to one of the following configurations 
for the weights $w_{ij}$:
$$
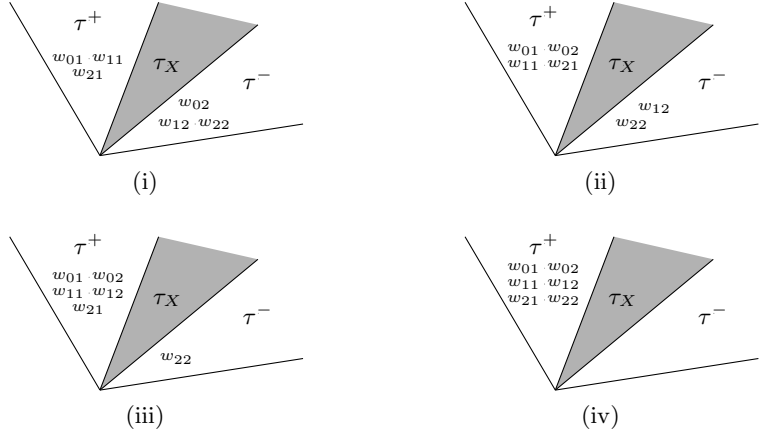
\begin{array}{ccc}
&&
\\
    \begin{tikzpicture}[scale=0.6]
    \path[fill=gray!60!] (0,0)--(3.5,2.9)--(1.3,3.4)--(0,0);
    \path[fill, color=black] (1.5,2) circle (0.0ex)  node[]{\small{$\tx$}};
    \path[fill, color=black] (-0.25,2.15) circle (0.0ex)  node[]{\tiny{$w_{01} \ w_{11}$}};
    \path[fill, color=black] (-0.25,1.8) circle (0.0ex)  node[]{\tiny{$w_{21}$}};
    \path[fill, color=black] (-0.25,3) circle (0.0ex)  node[]{\small{$\tp$}};
    \draw (0,0)--(1.3,3.4);
    \draw (0,0) --(-2,3.4);
    \path[fill, color=black] (2.1,1.15) circle (0.0ex)  node[]{\tiny{$w_{02}$}};
    \path[fill, color=black] (2.1,0.7) circle (0.0ex)  node[]{\tiny{$w_{12} \ w_{22}$}};
    \path[fill, color=black] (3.5,1.7) circle (0.0ex)  node[]{\small{$\tm$}};
    \draw (0,0)  -- (3.5,2.9);
    \draw (0,0)  -- (4.5,0.7);
    \path[fill, color=black] (1,-.6) circle (0.0ex)  node[]{\small{{\rm (i)}}};
  \end{tikzpicture}
& \qquad \qquad &
  \begin{tikzpicture}[scale=0.6]
    \path[fill=gray!60!] (0,0)--(3.5,2.9)--(1.3,3.4)--(0,0);
    \path[fill, color=black] (1.5,2) circle (0.0ex)  node[]{\small{$\tx$}};
    \path[fill, color=black] (-0.25,2.35) circle (0.0ex)  node[]{\tiny{$w_{01} \ w_{02}$}};
    \path[fill, color=black] (-0.25,2) circle (0.0ex)  node[]{\tiny{$w_{11} \ w_{21}$}};
    \path[fill, color=black] (-0.25,3.1) circle (0.0ex)  node[]{\small{$\tp$}};
    \draw (0,0)--(1.3,3.4);
    \draw (0,0) --(-2,3.4);
    \path[fill, color=black] (2.2,1.05) circle (0.0ex)  node[]{\tiny{$w_{12}$}};
    \path[fill, color=black] (1.7,0.7) circle (0.0ex)  node[]{\tiny{$w_{22}$}};
    \path[fill, color=black] (3.5,1.7) circle (0.0ex)  node[]{\small{$\tm$}}; 
    \draw (0,0)  -- (3.5,2.9);
    \draw (0,0)  -- (4.5,0.7);
    \path[fill, color=black] (1,-.6) circle (0.0ex)  node[]{\small{{\rm (ii)}}};
  \end{tikzpicture}
\\[1ex]
  \begin{tikzpicture}[scale=0.6]
    \path[fill=gray!60!] (0,0)--(3.5,2.9)--(1.3,3.4)--(0,0);
    \path[fill, color=black] (1.5,2) circle (0.0ex)  node[]{\small{$\tx$}};
    \path[fill, color=black] (-0.25,2.5) circle (0.0ex)  node[]{\tiny{$w_{01} \ w_{02}$}};
    \path[fill, color=black] (-0.25,2.15) circle (0.0ex)  node[]{\tiny{$w_{11} \ w_{12}$}};
    \path[fill, color=black] (-0.25,1.8) circle (0.0ex)  node[]{\tiny{$w_{21}$}};
    \path[fill, color=black] (-0.25,3.25) circle (0.0ex)  node[]{\small{$\tp$}}; 
    \draw (0,0)--(1.3,3.4);
    \draw (0,0) --(-2,3.4);
    \path[fill, color=black] (1.7,0.7) circle (0.0ex)  node[]{\tiny{$w_{22}$}};
    \draw (0,0)  -- (3.5,2.9);
    \draw (0,0)  -- (4.5,0.7);
    \path[fill, color=black] (3.5,1.7) circle (0.0ex)  node[]{\small{$\tm$}}; 
    \path[fill, color=black] (1,-.6) circle (0.0ex)  node[]{\small{{\rm (iii)}}};
  \end{tikzpicture}
& \qquad \qquad &
   \begin{tikzpicture}[scale=0.6]
    \path[fill=gray!60!] (0,0)--(3.5,2.9)--(1.3,3.4)--(0,0);
    \path[fill, color=black] (1.5,2) circle (0.0ex)  node[]{\small{$\tx$}};
    \path[fill, color=black] (-0.25,2.7) circle (0.0ex)  node[]{\tiny{$w_{01} \ w_{02}$}};
    \path[fill, color=black] (-0.25,2.35) circle (0.0ex)  node[]{\tiny{$w_{11} \ w_{12}$}};
    \path[fill, color=black] (-0.25,2) circle (0.0ex)  node[]{\tiny{$w_{21} \ w_{22}$}};
    \path[fill, color=black] (-0.25,3.25) circle (0.0ex)  node[]{\small{$\tp$}};
    \draw (0,0)--(1.3,3.4);
    \draw (0,0) --(-2,3.4);
    \path[fill, color=black] (3.5,1.7) circle (0.0ex)  node[]{\small{$\tm$}}; 
    \draw (0,0)  -- (3.5,2.9);
    \draw (0,0)  -- (4.5,0.7);
    \path[fill, color=black] (1,-.6) circle (0.0ex)  node[]{\small{{\rm (iv)}}};
  \end{tikzpicture}
\end{array}
$$
Configuration~(i) amounts to No.~4
in Theorems~\ref{thm:main1},~\ref{thm:main2} and~\ref{thm:main3},
configuration~(ii) to No.~5,
configuration~(iii) to Nos.~6 and~7,
and configuration~(iv) to Nos.~8 and~9.
\end{case}

\begin{proof}[Proof for configuration~(i)]
We have $w_{01},w_{11},w_{21}\in\tp$ and $w_{02},w_{12},w_{22}\in\tm$.
We may assume $w_k \in \tp$ for all $k = 1, \ldots, m$.
If $m>0$, we have $\gamma_{i2,1} \in \rlv(X)$ and
Lemma~\ref{lem:sampleFF}~(ii) gives 
$l_{i1} = 1$ for $i=0,1,2$.
If $m=0$, we use $\gamma_{i_11,i_22} \in \rlv(X)$
and Lemma~\ref{lem:sampleFF}~(iv) to obtain 
$l_{i_12}=1$ or $l_{i_21}=1$ for all $i_1\neq i_2$.
Thus, for $m=0$, we may assume $l_{01}=l_{11}=1$ 
and are left with $l_{21}=1$ or $l_{22}=1$.

We treat the case $m \ge 0$ and $l_{01}=l_{11}=l_{21}=1$.
Here we may assume $w_{11},w_{21},w_{22} \in \cone(w_{01},w_{12})$.
Applying Remark~\ref{rem:Q} firstly to $\gamma_{01,12}$ and 
then to $\gamma_{01,22}$, $\gamma_{12,21}$ and all 
$\gamma_{12,k}$ gives
$$
Q 
\ = \ 
\left[ 
\begin{array}{cc|cc|cc|ccc}
0 & w_{02}^1 & w_{11}^1 & 1 & w_{21}^1 & 1 & w_1^1 & \ldots & w_m^1
\\
1 & w_{02}^2 & w_{11}^2 & 0 & 1 & w_{22}^2 & 1 & \ldots  & 1 
\end{array}
\right].
$$
Using $w_{11},w_{21},w_{22} \in \cone(w_{01},w_{12})$ and the
fact that the determinants of $(w_{02},w_{01})$,
$(w_{12},w_{11})$ and $(w_{22},w_{21})$ 
are positive, we obtain
$$ 
w_{11}^1,\,  w_{21}^1,\,  w_{22}^2 \ \ge \ 0,
\qquad
w_{02}^1,\, w_{11}^2 \ > \ 0,
\qquad
1 \ > \ w_{22}^2 w_{21}^1.
$$
The degree $\mu$ of the relation satisfies
$$
0 
\ < \ 
\mu^1  
\ =  \ 
l_{02}w_{02}^1 
\ = \ 
w_{11}^1 + l_{12} 
\ = \ 
w_{21}^1+ l_{22},
$$
$$
0 
\ < \ 
\mu^2  
\ =  \ 
1+l_{02}w_{02}^2 
\ = \ 
w_{11}^2 
\ = \ 
1+l_{22}w_{22}^2.
$$
In particular, $w_{02}^2 \ge 0$ holds and thus
all components of the $w_{ij}$ are non-negative.
With 
$\gamma_{02,11},\gamma_{02,21}, \in \rlv(X)$ 
and Remark~\ref{rem:Q},
we obtain
$$
w_{02}^1w_{11}^2  \ =  \ 1 + w_{02}^2w_{11}^1,
\qquad\qquad
w_{02}^1-1  \ = \ w_{02}^2w_{21}^1.
$$

We show $w_{22}^2 = 0$. Otherwise, because of 
$1 > w_{22}^2 w_{21}^1$, we have $w_{21}^1=0$.
This implies $w_{02}^1=1$ and thus
$$
w_{11}^2  
\ =  \
1 + w_{02}^2w_{11}^1
\ = \
1+l_{02}w_{02}^2.
$$
This gives $w_{02}^2=0$ or $w_{11}^1=l_{02}$.
The first is impossible because of 
$l_{02}w_{02}^2 = l_{22}w_{22}^2$
and the second because of  
$l_{02} = l_{02}w_{02}^1 = w_{11}^1 + l_{12}$.

Knowing $w_{22}^2 = 0$, we directly conclude
$w_{11}^2 = 1$ and $w_{02}^2 = 0$ from $\mu^2  =  1$.
This gives $w_{02}^1=1$.
With 
$a := w_{11}^1 \in \ZZ_{\ge 0}$, 
$b := w_{21}^1 \in \ZZ_{\ge 0}$ and 
$c_k := w_{k}^1 \in \ZZ$ 
we are in the situation
$$
g_0 
\ = \ 
T_{01}T_{02}^{l_{02}}+T_{11}T_{12}^{l_{12}}+T_{21}T_{22}^{l_{22}},
\qquad
Q \ = \ \left[ 
\begin{array}{cc|cc|cc|ccc}
0 & 1 & a & 1 & b & 1 & c_1 & \dots & c_m 
\\ 
1 & 0 & 1 & 0 & 1 & 0 & 1 & \dots & 1
\end{array}
\right],
$$
where we may assume $0 \le a \le b$ and 
$c_1 \le \ldots \le c_m$.
Observe $l_{02}=a+l_{12}=b+l_{22}$.
The anticanonical class and the semiample 
cone of $X$ are given by 
\begin{eqnarray*}
-\mathcal{K}_X
& = &
(3 + b + c_1 + \ldots + c_m - l_{12}, \, 2 + m),
\\
\SAmple(X) 
& = & 
\cone((1,0), (d,1)),
\end{eqnarray*} 
where $d \sei \max(b, c_m)$. 
Consequently, $X$ is a Fano variety if and only if 
the following inequality holds
$$
3 + b + c_1 + \ldots + c_m - l_{12}
\ > \ 
(2+m)d.
$$
A necessary condition for this is
$0 \le d \le 1$
with $l_{12} = 1$ if $d=1$ and 
$l_{12} \le 2$ if $d=0$
The tuples $(a,b,d,l_{02},l_{12},l_{22})$ 
fulfilling that condition are
$$
(0,0,0,2,2,2), \qquad
(0,0,0,1,1,1), \qquad
(1,1,1,2,1,1).
$$
Each of these three tuples leads indeed
to a Fano variety $X$; the respectively
possible choices of the $c_k$ lead to 
Nos.~4.A, 4.B and~4.C of Theorem~\ref{thm:main2}
and are as follows:
$$
c_1 = \ldots =c_m=0, \qquad
-1 \le c_1 \le 0 = c_2 = \ldots = c_m, \qquad
c_1 = \ldots =c_m=1.
$$
Moreover $X$ is
a truly almost Fano variety if and only if the following 
equality holds
$$
3 + b + c_1 + \ldots + c_m - l_{12}
\ = \ 
(2+m)d.
$$
This implies $0 \le d \le 2$ and the only possible
parameters fulfilling that condition are
listed as Nos.~4.A to~4.F
in the table of Theorem~\ref{thm:main3}.

We turn to the case $m = 0$,
$l_{01}=l_{11}=1$ and $l_{21} \ge 2$.
Lemma~\ref{lem:sampleFF}~(iv) applied to
$\gamma_{01,22}, \gamma_{11,22} \in \rlv(X)$
gives $l_{02}=l_{12}=1$.
If $l_{22}=1$, then suitable admissible 
operations bring us to the previous case.
So, let $l_{22} \ge 2$.
We may assume $w_{11} \in \cone(w_{01},w_{12})$.
We apply Remark~\ref{rem:Q} firstly to~$\gamma_{01,12}$, 
then to $\gamma_{01,22}$, $\gamma_{12,21}$
and arrive at
$$
g_0 
\ = \ 
T_{01}T_{02}+T_{11}T_{12}+T_{21}^{l_{21}}T_{22}^{l_{22}}, 
\quad
Q 
\ = \ 
\left[ 
\begin{array}{cc|cc|cc}
0 & w_{02}^1 & w_{11}^1 & 1 & w_{21}^1 & 1
\\
1 & w_{02}^2 & w_{11}^2 & 0 & 1 & w_{22}^2 
\end{array}
\right],
$$
where $w_{11}^1 \ge 0$ and 
$w_{11}^2 = \det(w_{12},w_{11}) > 0$.
We have  
$\mu = w_{02}+w_{01} = w_{11}+w_{12}$
and thus 
$w_{02} = w_{11}+w_{12}-w_{01}$.
Because of $\gamma_{02,11} \in \rlv(X)$,
we obtain
$$ 
1 
\ = \ 
\det(w_{02},w_{11})
\ = \ 
\det(w_{12}-w_{01},w_{11})
\ = \ 
w_{11}^1+w_{11}^2.
$$
We conclude $w_{11} = (0,1)$
and $\mu = (1,1)$.
Using 
$\mu = l_{21}w_{21}+l_{22}w_{22}$
and 
$l_{21},l_{22} \ge 2$
we see
$w_{21}^1, w_{22}^2 <0$.
On the other hand,
$0 < \det(w_{22},w_{21}) = 1 - w_{21}^1w_{22}^2$,
a contradiction.
Thus $l_{22} \ge 2$ does not occur.
\end{proof}

\begin{proof}[Proof for configuration~(ii)]
We have $w_{01},w_{02},w_{11},w_{21}\in\tp$ and $w_{12},w_{22}\in\tm$.
We may assume that $w_{02},w_{12} \in \cone(w_{01},w_{22})$ holds.
Applying Remark~\ref{rem:Q} first to 
$\gamma_{01,22} \in \rlv(X)$ and then 
to $\gamma_{01,12}, \gamma_{02,22}, \gamma_{11,22} \in \rlv(X)$
we obtain
$$
Q 
\ = \ 
\left[ 
\begin{array}{cc|cc|cc|ccc}
0 & w_{02}^1 & w_{11}^1 & 1  & w_{21}^1 & 1 & w_1^1 & \dots & w_m^1
\\
1 & 1 & 1 & w_{12}^2 & w_{21}^2 & 0 & w_1^2 & \dots  & w_m^2 
\end{array}
\right],
$$
where we have $w_{02}^1, w_{12}^2 \ge 0$ due to 
$w_{02},w_{12} \in \cone(w_{01},w_{22})$.
Moreover, $w_{21}^2>0$ holds, as we infer 
from the conditions 
$$
0 
\ \le \ 
\mu^1  
\ =  \ 
l_{02}w_{02}^1 
\ = \ 
l_{11}w_{11}^1 + l_{12} 
\ = \ 
l_{21}w_{21}^1+ l_{22},
$$
$$
0 
\ < \ 
\mu^2  
\ =  \ 
l_{01}+l_{02} 
\ = \ 
l_{11}+l_{12}w_{12}^2 
\ = \ 
l_{21}w_{21}^2.
$$

We show $l_{11} \ge 2$. Otherwise, the above 
conditions give $l_{12}w_{12}^2 > 0$ and thus
$w_{12}^2 > 0$. 
For $\gamma_{02,12} \in \rlv(X)$, Remark~\ref{rem:Q}
gives $\det(w_{12},w_{02}) = 1$ which means 
$w_{12}^2w_{02}^1=0$ and thus $w_{02}^1 = 0$.
This implies $l_{21}w_{21}^1+ l_{22} = 0$ 
and thus $w_{21}^1 < 0$; a contradiction to 
$1= \det(w_{12},w_{21}) = w_{21}^2 - w_{12}^2w_{21}^1$
which in turn holds due to $\gamma_{12,21} \in \rlv(X)$
and Remark~\ref{rem:Q}.

Lemma~\ref{lem:sampleFF}~(iv) 
applied to $\gamma_{02,12}, \gamma_{01,12}, \gamma_{21,12} \in \rlv(X)$ 
shows $l_{01}=l_{02}=l_{22}=1$. 
Putting together $\mu^2=2= l_{11}+l_{12}w_{12}^2$ and 
$l_{11} \neq 1$, we conclude 
$l_{11}=2$ and $w_{12}^2=0$. 
With $\gamma_{12,21} \in \rlv(X)$ and  
Remark~\ref{rem:Q} we obtain $w_{21}^2=1$ 
and hence $l_{21}=\mu^2=2$.
From
$$
0
\ \le \ 
\mu^1
\ = \ 
w_{02}^1
\ = \ 
2w_{11}^1+1
\ = \ 
2w_{21}^1+1
$$ 
we conclude $w_{11}^1 =w_{21}^1 \geq 0$ 
and thus $w_{02}^1 > 0$. 
Lemma~\ref{lem:sampleFF}~(ii) implies 
that possible weights of type $w_k$ lie in $\tm$.
Thus Remark~\ref{rem:Q} and $\gamma_{01,k}$ 
imply $w_{k}^1=1$ for all $k$. 
Moreover, since $\gamma_{02,k} \in \rlv(X)$, the latter implies $w_k^2=0$.
All in all, we arrive at
$$
g_0 
\ =\ 
T_{01}T_{02}+T_{11}^2T_{12}+T_{21}^2T_{22},
\quad
Q \ = \ \left[ 
\begin{array}{cc|cc|cc|ccc}
0 & 2a+1 & a & 1 & a & 1 & 1 & \ldots & 1 
\\ 
1 & 1 & 1 & 0 & 1 & 0 & 0 & \ldots & 0
\end{array}
\right],
$$
where $a \in \ZZ_{\ge 0}$.
The anticanonical class is $-\mathcal{K}_X=(2a+2+m,2)$
and the semiample cone is $\SAmple(X) =\cone((1,0), (2a+1,1))$. 
Hence $X$ is an almost Fano variety if and only if $m \ge 2a$ holds
and $X$ is a Fano variety if and only if $m > 2a$ holds.
\end{proof}

\begin{proof}[Proof for configuration~(iii)]
We have $w_{01},w_{02},w_{11},w_{12},w_{21}\in\tp$ 
and $w_{22}\in \tm$.
As there must be another weight in $\tm$,
we obtain $m > 0$.  
Lemma~\ref{lem:tau}~(v) yields
$w_1, \ldots, w_m \in \tm$.
We may assume 
$w_{02}, w_{11}, w_{12}, w_k  \in \cone(w_{01},w_1)$,
where $k = 2, \ldots, m$.
Applying Remark~\ref{rem:Q} firstly to 
$\gamma_{01,1} \in \rlv(X)$ and then 
to the remaining faces 
$\gamma_{01,22}, \gamma_{01,k}, \gamma_{ij,1}$ 
from $\rlv(X)$ 
leads to the degree matrix
$$
Q 
\ = \ 
\left[ 
\begin{array}{cc|cc|cc|cccc}
0 & w_{02}^1 & w_{11}^1 & w_{12}^1 & w_{21}^1 & 1 & 1 & 1 & \ldots & 1 
\\
1 & 1 & 1 & 1 & 1 & w_{22}^2 & 0 & w_2^2 & \ldots & w_m^2
\end{array}
\right]
$$
with at most $w_{21}^1,w_{22}^2$ negative.
We infer $l_{01}=l_{02}=l_{11}=l_{12}=l_{22}=1$
from Lemma~\ref{lem:sampleFF}~(ii).
For $\gamma_{02,22},\gamma_{11,22},\gamma_{12,22} \in \rlv(X)$
Remark~\ref{rem:Q} tells us
$$
w_{22}^2 \ = \ 0 
\qquad \text{or} \qquad 
w_{02}^1 \ = \ w_{11}^1 \ = \ w_{12}^1 \ = \ 0.
$$

We treat the case $w_{22}^2=0$. 
Here $l_{21} = \mu^2 =2$ holds.
Thus $\mu^1=w_{02}^1 = 2w_{21}^1+1$
holds.
Because of $w_{02}^1 \ge 0$, we conclude 
$w_{02}^1 > 0$ and $w_{21}^1 \ge 0$.
Remark \ref{rem:Q} applied 
to $\gamma_{02,k} \in \rlv(X)$
gives $w_k^2=0$ for all $k=2,\dots,m$. 
We arrive at 
$$
g_0 
\ = \ 
T_{01}T_{02}+T_{11}T_{12}+T_{21}^2T_{22},
\quad
Q 
\ = \ 
\left[ 
\begin{array}{cc|cc|cc|ccc}
0 & 2c+1 & a & b & c & 1 & 1 &  \ldots & 1 
\\ 
1 & 1 & 1 & 1 & 1 & 0 & 0 & \ldots & 0
\end{array}
\right],
$$
where $a,b,c \in \ZZ_{\geq 0}$ and $a+b=2c+1$.
Furthermore, the anticanonical class is 
$-\mathcal{K}_X =(3c+2+m,3)$
and we have $\SAmple(X)=\cone((1,0), (2c+1,1))$.
In particular, $X$ is an almost Fano variety if and only 
if $3c+1 \le m$ holds
and a Fano variety if and only if
the corresponding strict inequality holds.

Now we consider the case $w_{02}^1=w_{11}^1=w_{12}^1=0$.
We have $\mu^1 =0$, which implies $l_{21}=1$, $w_{21}^1=-1$.
Consequently, $\mu^2 =2$ gives $w_{22}^2=1$.
Since $\gamma_{21,k}\in \rlv(X)$ for $2 \leq k \leq m$, 
we conclude~$w_k^2=0$ for all $k$.
Therefore we obtain
$$
g_0 
\ = \ 
T_{01}T_{02}+T_{11}T_{12}+T_{21}T_{22},
\quad
Q 
\ = \ 
\left[ 
\begin{array}{cc|cc|cc|ccc}
0 & 0 & 0 & 0 & -1 & 1 & 1 &  \ldots & 1 
\\ 
1 & 1 & 1 & 1 & 1 & 1 & 0 & \ldots & 0
\end{array}
\right].
$$
Finally, we have $-\mathcal{K}_X =(m,4)$
and $\SAmple(X)=\cone((1,1), (0,1))$.
Thus, $X$ is a Fano variety if and only if 
$m<4$ holds.
Moreover, $X$ is an almost Fano variety if and only if 
$m\le 4$ holds.
\end{proof}

\begin{proof}[Proof for configuration~(iv)]
All $w_{ij}$ lie in $\tp$. 
Then we have $m \ge 2$ and one and hence all 
$w_k$ in lie in $\tm$, see Lemma~\ref{lem:tau}~(v).
Applying Lemma~\ref{lem:sampleFF}~(ii) 
to $\gamma_{ij,1} \in \rlv(X)$,
we conclude $l_{ij} = 1$ for all $i,j$.
Thus we have the relation 
$$
g_0
\ = \ 
T_{01}T_{02}+T_{11}T_{12}+T_{21}T_{22}.
$$
We may assume that $\cone(w_{01},w_1)$ 
contains all $w_{ij}, w_k$.
Remark~\ref{rem:Q} applied to 
$\gamma_{01,1} \in \rlv(X)$ leads to 
$w_1  =  (1,0)$ and $w_{01}  =  (0,1)$.
All other weights lie in the positive orthant.
For $\gamma_{ij,1}, \gamma_{01,k} \in \rlv(X)$ 
Remark~\ref{rem:Q} shows 
$w_{ij}^2=w_k^1=1$ for all $i,j,k$.
Consider the case that all $w_k^2$ vanish.
Then the degree matrix is of the form
$$
Q 
\ = \ 
\left[ 
\begin{array}{cc|cc|cc|ccc}
0 & a_2 & a_3 & a_4 & a_5 & a_6 & 1 & \ldots & 1 
\\ 
1 & 1 & 1 & 1 & 1 & 1 & 0 & \ldots & 0
\end{array}
\right],
$$
where $a_i \in \ZZ_{\ge 0}$ and $a_2=a_3+a_4=a_5+a_6$. 
We have $-\mathcal{K}_X = (2a_2+m,4)$
and $\SAmple(X)= \cone((1,0),(a_2,1))$.
Hence~$X$ is a Fano variety if and only if 
$2a_2 < m$ holds and an almost Fano variety
if and only if $2a_2 \le m$ holds.

Finally, let $w_k^2>0$ for some $k$.
Note that we may assume $0\le w_2^2 \le\ldots\le w_m^2$; 
in particular $w_m^2>0$.
Since $\gamma_{ij,m} \in  \rlv(X)$ for all 
$i,j$, Remark~\ref{rem:Q} yields
$w_{ij}^1 = 0$ for all $i,j$.
Thus we obtain the degree matrix
$$
Q 
\ = \ 
\left[ 
\begin{array}{cc|cc|cc|cccc}
0 & 0 & 0 & 0 & 0 & 0 & 1 & 1 & \ldots & 1 
\\ 
1 & 1 & 1 & 1 & 1 & 1 & 0 & a_2 & \ldots & a_m
\end{array}
\right],
$$
where $0\le a_2 \le\ldots\le a_m$ 
and $a_m>0$.
The anticanonical class and the semiample cone 
are given as 
$$
-\mathcal{K}_X \ = \ (m,\, 4+a_2+ \ldots + a_m),
\qquad 
\SAmple(X) \ = \ \cone((0,1),(1,a_m)).
$$
In particular, $X$ is a Fano variety if and only if 
$4+a_2+ \ldots + a_m > ma_m$ holds.
Note that for the latter $a_m \le 3$ is necessary.
Moreover, $X$ is a truly almost Fano variety
if and only if the equality $4+a_2+ \ldots + a_m = ma_m$ holds.
\end{proof}

\begin{case}
We have $r=2$, $m\ge 0$, $n = 5$
and the list of $n_i$ is $(2,2,1)$.
This leads to Nos.~10, 11 and~12 in 
Theorems~\ref{thm:main1},~\ref{thm:main2} and~\ref{thm:main3}.
\end{case}

\begin{proof}
We divide this case into the following three configurations,
according to the way some weights lie with respect to $\tx$.

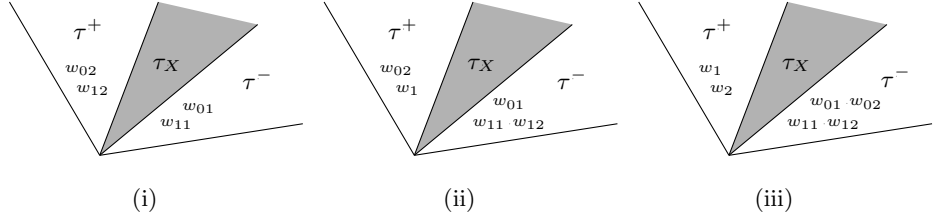
\vspace{0.3cm}
   \begin{tikzpicture}[scale=0.6]
    \path[fill=gray!60!] (0,0)--(3.5,2.9)--(1.3,3.4)--(0,0);
    \path[fill, color=black] (1.5,2) circle (0.0ex)  node[]{\small{$\tx$}};
    \path[fill, color=black] (-0.4,1.9) circle (0.0ex)  node[]{\tiny{$w_{02}$}};
    \path[fill, color=black] (-0.15,1.45) circle (0.0ex)  node[]{\tiny{$w_{12}$}};
    \path[fill, color=black] (-0.25,2.8) circle (0.0ex)  node[]{\small{$\tp$}};
    \draw (0,0)--(1.3,3.4);
    \draw (0,0) --(-2,3.4);
    \path[fill, color=black] (2.2,1.05) circle (0.0ex)  node[]{\tiny{$w_{01}$}};
    \path[fill, color=black] (1.7,0.7) circle (0.0ex)  node[]{\tiny{$w_{11}$}};
    \path[fill, color=black] (3.5,1.7) circle (0.0ex)  node[]{\small{$\tm$}};
    \draw (0,0)  -- (3.5,2.9);
    \draw (0,0)  -- (4.5,0.7);
    \path[fill, color=black] (1,-1) circle (0.0ex)  node[]{\small{(i)}};
  \end{tikzpicture}
\
   \begin{tikzpicture}[scale=0.6]
    \path[fill=gray!60!] (0,0)--(3.5,2.9)--(1.3,3.4)--(0,0);
    \path[fill, color=black] (1.5,2) circle (0.0ex)  node[]{\small{$\tx$}};
    \path[fill, color=black] (-0.4,1.9) circle (0.0ex)  node[]{\tiny{$w_{02}$}};
    \path[fill, color=black] (-0.15,1.45) circle (0.0ex)  node[]{\tiny{$w_{1}$}};
    \path[fill, color=black] (-0.25,2.8) circle (0.0ex)  node[]{\small{$\tp$}};
    \draw (0,0)--(1.3,3.4);
    \draw (0,0) --(-2,3.4);
    \path[fill, color=black] (2.1,1.15) circle (0.0ex)  node[]{\tiny{$w_{01}$}};
    \path[fill, color=black] (2.1,0.7) circle (0.0ex)  node[]{\tiny{$w_{11} \ w_{12}$}};
    \path[fill, color=black] (3.5,1.7) circle (0.0ex)  node[]{\small{$\tm$}};
    \draw (0,0)  -- (3.5,2.9);
    \draw (0,0)  -- (4.5,0.7);
    \path[fill, color=black] (1,-1) circle (0.0ex) node[]{\small{(ii)}};
  \end{tikzpicture}
\
     \begin{tikzpicture}[scale=0.6]
    \path[fill=gray!60!] (0,0)--(3.5,2.9)--(1.3,3.4)--(0,0);
    \path[fill, color=black] (1.5,2) circle (0.0ex)  node[]{\small{$\tx$}};
    \path[fill, color=black] (-0.4,1.9) circle (0.0ex)  node[]{\tiny{$w_{1}$}};
    \path[fill, color=black] (-0.15,1.45) circle (0.0ex)  node[]{\tiny{$w_{2}$}};
    \path[fill, color=black] (-0.25,2.8) circle (0.0ex)  node[]{\small{$\tp$}};
    \draw (0,0)--(1.3,3.4);
    \draw (0,0) --(-2,3.4);
    \path[fill, color=black] (2.6,1.15) circle (0.0ex)  node[]{\tiny{$w_{01} \ w_{02}$}};
    \path[fill, color=black] (2.1,0.7) circle (0.0ex)  node[]{\tiny{$w_{11} \ w_{12}$}};
    \path[fill, color=black] (3.7,1.75) circle (0.0ex)  node[]{\small{$\tm$}};
    \draw (0,0)  -- (3.5,2.9);
    \draw (0,0)  -- (4.5,0.7);
    \path[fill, color=black] (1,-1) circle (0.0ex) node[]{\small{(iii)}};
  \end{tikzpicture}
We show that configuration~(i) does not provide 
any smooth variety,
(ii) delivers No.~10 of Theorem~\ref{thm:main1}
and~(iii) delivers Nos.~11 and 12.

In configuration~(i) we have $w_{01},w_{11}\in\tm$ 
and $w_{02},w_{12}\in\tp$.
We may assume $w_{11} \in \cone(w_{01}, w_{12})$.
Remark~\ref{rem:Q} applied to $\gamma_{01,12}\in\rlv(X)$
leads to $w_{01}=(1,0)$ and $w_{12}=(0,1)$.
Observe $w_{11}^1, w_{11}^2 \ge 0$.
Due to 
$\det(w_{11},w_{12}) > 0$, we even have 
$w_{11}^1 > 0$
and 
$\det(w_{01},w_{02}) > 0$
gives $w_{02}^2 > 0$.
Since $T_0^{l_0}$ and $T_1^{l_1}$ share the same degree,
we have
$$ 
l_{01}w_{01} + l_{02}w_{02} \ = \ l_{11}w_{11} + l_{12}w_{12}.
$$
Lemma~\ref{lem:sampleFF}~(iv) says $l_{02}=1$ or $l_{11}=1$,
which allows us to resolve for $w_{02}$ or for $w_{11}$
in the above equation.
Using $\gamma_{02,11}\in\rlv(X)$, we obtain
\begin{eqnarray*}
l_{02} = 1
& \implies & 
1 
= \det(w_{11},w_{02}) 
= \det(w_{11},l_{12}w_{12}-l_{01}w_{01}) 
= l_{12}w_{11}^1+l_{01}w_{11}^2,
\\
l_{11} = 1
& \implies & 
1 
=
\det(w_{11},w_{02}) 
=
\det(l_{01}w_{01}-l_{12}w_{12},w_{02}) 
= 
l_{01}w_{02}^2+l_{12}w_{02}^1.
\end{eqnarray*}
We show $l_{02} > 1$. Otherwise, $l_{02} = 1$ holds.
The above consideration shows $w_{11}^2 = 0$ and 
$l_{12}= w_{11}^1 = 1$.
Thus, $l_{21}w_{21}^2 = l_{12} =1$ holds and we obtain 
$l_{21} = 1$; a contradiction to $P$ being 
irredundant.
Thus, $l_{02} > 1$ and $l_{11} = 1$ must hold.
Because of $w_{02}^2 > 0$, we must have 
$w_{02}^1 \le 0$. 
With
$$ 
1
\ = \ 
\det(w_{11},w_{02}) 
\ = \
w_{11}^1 w_{02}^2 - w_{11}^2 w_{02}^1 
$$
we see $w_{11}^2 w_{02}^1 = 0$ and $w_{11}^1 = w_{02}^2 = 1$.
But then we arrive at $1 = l_{11}w_{11}^1 = l_{21}w_{21}^1$.
Again this means $l_{21} = 1$; a contradiction to $P$ being 
irredundant.

In configuration~(ii) we have
$w_{01},w_{11},w_{12} \in \tm$ and $w_{02},w_{1} \in \tp$.
In particular $m\ge1$.
Lemma~\ref{lem:tau}~(v) yields $w_2,\ldots,w_m \in \tp$.
Applying Remark~\ref{rem:Q} first to $\gamma_{11,1}\in\rlv(X)$
an then to 
$\gamma_{01,1},\gamma_{12,1},\gamma_{02,11},\gamma_{11,k}\in\rlv(X)$
leads to 
$$
Q 
\ = \ 
\left[ 
\begin{array}{cc|cc|c|cccc}
1 & w_{02}^1 & 1 & 1 & w_{21}^1 & 0 & w_2^1 & \dots & w_m^1 
\\ 
w_{01}^2 & 1 & 0 & w_{12}^2 & w_{21}^2 & 1 & 1 & \dots & 1
\end{array}
\right].
$$
Applying Lemma~\ref{lem:sampleFF}~(ii) 
to $\gamma_{01,1},\gamma_{12,1},\gamma_{11,1}\in\rlv(X)$
we obtain $l_{02}=l_{11}=l_{12}=1$.
For the degree $\mu$ of the relation $g_0$ we note
$$
\mu^1 
\ = \ 
l_{01}+w_{02}^1 
\ =  \ 
2 
\ = \ 
l_{21} w_{21}^1,
\qquad
\qquad
\mu^2 
\ = \
l_{01}w_{01}^2+1 
\ = \  
w_{12}^2 
\ = \ 
l_{21}w_{21}^2.
$$
From $\mu^1=2$ we infer $l_{21}=2$ and $w_{21}^1=1$.
Consequently, $\mu^2$ is even and
both $l_{01}, w_{01}^2$ are odd. 
Using again $\mu^1=2$ gives $w_{02}^1 \ne 0$.
For $\gamma_{02,12}\in\rlv(X)$ 
Remark~\ref{rem:Q} yields $\det(w_{12},w_{02})=1$ 
which means $w_{02}^1w_{12}^2=0$.
We conclude $w_{12}^2=0=\mu^2$.
This implies $w_{21}^2=0$, $w_{01}^2=-1$,
$l_{01}=1$ and $w_{02}^1=1$.
We obtain
$$
g_0 
\ = \ 
T_{01}T_{02}+T_{11}T_{12}+T_{21}^2,
\qquad
Q 
\ = \ 
\left[ 
\begin{array}{cc|cc|c|ccc}
1 & 1 & 1 & 1 & 1 & 0 & \dots & 0
\\ 
-1 & 1 & 0 & 0 & 0 & 1 & \dots & 1
\end{array}
\right],
$$
where $w_2^1= \ldots =w_m^1=0$ follows from 
Remark~\ref{rem:Q} applied 
to $\gamma_{01,k}\in\rlv(X)$.
The semiample cone is given as $\SAmple(X)= \cone((1,0),(1,1))$
and the anticanonical class as $-\mathcal{K}_X=(3,m)$.
Therefore $X$ is a Fano variety if and only if $m<3$, i.e $m=1,2$.
Moreover, $X$ is an almost Fano variety if and only if $m\le3$.

In configuration~(iii) we have 
$w_{01},w_{02},w_{11},w_{12}\in\tm$ 
and $w_{1},w_{2}\in\tp$.
In particular $m\ge2$.
Lemma~\ref{lem:tau}~(v) ensures
$w_3, \ldots, w_m \in \tp$. 
We can assume that all $w_{ij}, w_k$ 
lie in $\cone(w_{01},w_{1})$.
Applying Remark~\ref{rem:Q},
firstly to $\gamma_{01,1}$ and then 
to all relevant faces of the types
$\gamma_{ij,1}$ and $\gamma_{01,k}$,
we achieve 
$$
w_{01}=(1,0), \quad
w_{1}=(0,1), \quad
w_{02}^1=w_{11}^1=w_{12}^1=1, \quad
w_2^2= \ldots = w_m^2 =1.
$$
Lemma~\ref{lem:sampleFF}~(ii) 
applied to all $\gamma_{ij,1}$ 
shows $l_{ij}=1$ for all $i,j$.
We conclude $\mu^1=2$ which in 
turn implies $l_{21}=2$ and $w_{21}^1=1$.
In particular, we have the relation
$$
g_0 \ = \ T_{01}T_{02} + T_{11}T_{12} + T_{21}^2.
$$

We treat the case that $w_1^1 = \ldots = w_m^1 =0$ holds.
All columns of the degree matrix lie in $\cone(w_{01},w_{1})$
and thus $Q$ is of the form
$$
Q 
\ = \ 
\left[ 
\begin{array}{cc|cc|c|cccc}
1 & 1 & 1 & 1 & 1 & 0 & 0 & \ldots & 0 
\\
0 & 2c & a & b & c & 1 & 1 & \ldots & 1
\end{array}
\right],
$$
where $a,b,c\in\ZZ_{\ge0}$ and $a+b=2c$.
The anticanonical class is $-\mathcal{K}=(3,m+3c)$
and we have $\SAmple(X)=\cone((0,1),(1,2c))$.
Therefore $X$ is a Fano variety if and only if $m>3c$.
Moreover, $X$ is an almost Fano variety 
if and only if $m\ge3c$.

We treat the case that $w_k^1>0$ holds for some $k$.
Then we obtain $w_{02}^2=0$ by applying Remark~\ref{rem:Q} 
to $\gamma_{02,k}$.
This yields $\mu^2=0$ and thus $w_{ij}^2=0$ for all $i,j$.
Consequently, the degree matrix is given as
$$
Q 
\ = \ 
\left[ 
\begin{array}{cc|cc|c|cccc}
1 & 1 & 1 & 1 & 1 & 0 & w_2^1 & \dots & w_m^1 
\\
0 & 0 & 0 & 0 & 0 & 1 & 1 & \dots & 1
\end{array}
\right],
$$
where we can assume $0 \le w_2^1 \le \ldots \le w_m^1$.
The semiample cone and the anticanonical divisor 
are given as
$$
\SAmple(X)=\cone((1,0),(w_m^1,1)),
\qquad
-\mathcal{K}=(3+w_2^1+\ldots+w_m^1,m).
$$
We see  that $X$ is an almost Fano variety if and only if 
$mw_m^1 \le 3+w_2^1+\ldots+w_m^1$
and that $X$ is a Fano variety if and only if the corresponding
strict inequality holds.
\end{proof}

\begin{case}
We have $r=2$, $m\ge 1$, $n = 4$
and the list of $n_i$ is $(2,1,1)$.
This case does not provide any smooth variety.
\end{case}

\begin{proof}
We can assume $w_{01}\in\tm$ and $w_{1}\in\tp$.
Lemma~\ref{lem:tau}~(v) ensures
$w_2, \ldots, w_m \in \tp$. 
Applying Remark~\ref{rem:Q} 
first to $\gamma_{01,1} \in \rlv(X)$
and then to the remaining $\gamma_{01,k} \in \rlv(X)$, 
we achieve
$$
Q 
\ = \ 
\left[ 
\begin{array}{cc|c|c|cccc}
1 & w_{02}^1 & w_{11}^1 & w_{21}^1 & 0 & w_2^1 & \dots & w_m^1
\\
0 & w_{02}^2 & w_{11}^2 & w_{21}^2 & 1 & 1 & \dots & 1
\end{array}
\right].
$$
Moreover $\gamma_{01,1}\in\rlv(X)$ implies 
$l_{02}=1$ by Lemma~\ref{lem:sampleFF}~(ii).
Recall from Corollary~\ref{cor:clfree} 
that $\Cl(X)$ is torsion-free.
Thus~\cite[Thm.~1.1]{HaHe:2013} implies that
$l_{11}$ and $l_{21}$ are coprime.

Consider the case $w_{02}\in\tm$.
Then $\gamma_{02,1} \in \rlv(X)$ holds,
Lemma~\ref{lem:sampleFF}~(ii) 
yields $l_{01}=1$ and
Remark~\ref{rem:Q} shows $w_{02}^1=1$. 
We conclude $\mu^1 = 2$ and thus 
obtain $l_{11}=l_{21}=2$; a contradiction.

Now let $w_{02}\in\tp$, which 
implies $\gamma_{01,02,11}\in\rlv(X)$.
Since $X$ is locally factorial, 
Remark \ref{rem:Qfact}~(ii) shows that 
$w_{02}^2$ and~$w_{11}^2$ are coprime. 
Now we look at 
$$
\mu^2 
\ = \ 
w_{02}^2 
\ = \ 
l_{11}w_{11}^2 
\ =  \ 
l_{21}w_{21}^2.
$$ 
We infer that~$l_{21}$ divides $w_{02}^2$ and $w_{11}^2$.
This contradicts coprimeness of $w_{02}^2$ and~$w_{11}^2$,
because by irredundancy of $P$ we have 
$l_{21} \ge 2$.
\end{proof}

\begin{case2}
We have $r=3$, $m=0$ and $2 = n_0 = n_1 \ge n_2 \ge n_3 \ge 1$.
This leads to No.~13 in Theorems~\ref{thm:main1} and~\ref{thm:main2}.
\end{case2}

\begin{proof}
We treat the constellations~(a), (b) and~(c) 
at once.
First observe that for every $w_{i_1j_1}$ 
with $n_{i_1}=2$,
there is at least one $w_{i_2j_2}$ 
with~$n_{i_2}=2$ and $i_1\neq i_2$
such that $\tx \subseteq Q(\gamma_{i_1j_1,i_2j_2})^{\circ}$
and thus 
$\gamma_{i_1j_1,i_2j_2} \in \rlv(X)$.
Since $r=3$, we conclude $l_{ij}=1$ for 
all $i$ with $n_i=2$; see Lemma~\ref{lem:sampleFF}~(iv).

We can assume $w_{01},w_{11}\in\tm$ and $w_{02},w_{12}\in\tp$ 
as well as $w_{11} \in \cone(w_{01},w_{12})$.
Applying Remark~\ref{rem:Q} to $\gamma_{01,12},\in\rlv(X)$, 
we obtain $w_{01}=(1,0)$ and $w_{12}=(0,1)$.
Moreover $w_{11}^1,w_{11}^2\ge 0$ holds and,
because of $w_{11} \not\in \tp$, we even have
$w_{11}^1>0$.
For the degree $\mu$ of $g_0$ and $g_1$ we note
$$
\mu^1 \ = \ w_{02}^1+1 \ = \ w_{11}^1,
\qquad
\qquad
\mu^2 \ = \ w_{02}^2 \ = \ w_{11}^2+1.
$$
Thus, we can express $w_{02}$ in terms of $w_{11}$.
Remark~\ref{rem:Q} applied to $\gamma_{02,11} \in \rlv(X)$
gives $1 = \det(w_{11},w_{02}) = w_{11}^1 + w_{11}^2$.
We conclude $w_{11}=(1,0)$ and $w_{02}=(0,1)$. 
In particular, the degree of the relations $g_0$ 
and $g_1$ is $\mu=(1,1)$.

In constellations~(b) and~(c), we have $n_3=1$
and $\mu=(1,1)$. This implies $l_{31}=1$, 
a contradiction to $P$ being irredundant.
Thus, constellations~(b) and~(c) do not occur.

We are left with constellation~(a), 
that means that we have  $n_0= \ldots = n_3=2$.
As seen before, $l_{ij} = 2$ for all $i,j$.
Thus, the relations are 
$$
g_0 
\ = \ 
T_{01}T_{02} + T_{11}T_{12} + T_{21}T_{22}, 
\qquad
g_1 
\ = \ 
\lambda T_{11}T_{12} + T_{21}T_{22} + T_{31}T_{32},
$$
where $\lambda\in\KK^*\setminus\{1\}$.
In this situation, we may assume 
$w_{21}, w_{31} \in \tm$. 
Applying Remark~\ref{rem:Q} to the relevant 
faces $\gamma_{02,21}, \gamma_{02,31}$, 
we conclude $w_{21}^1=w_{31}^1=1$.
Since $\mu^1=1$ and $l_{ij}=1$, 
we obtain $w_{22}^1=w_{32}^1=0$.
Thus, $w_{22}$ and $w_{32}$ lie in $\tp$. 
Again Remark~\ref{rem:Q}, this time  applied to 
$\gamma_{01,22}, \gamma_{01,32} \in \rlv(X)$, 
yields $w_{22}^2=w_{32}^2=1$.
Since $\mu^2=1$ and $l_{ij}=1$, 
we obtain $w_{21}^2=w_{31}^2=0$.
Hence we obtain the degree matrix
$$
Q \ = \ \left[ 
\begin{array}{cc|cc|cc|cc}
1 & 0 & 1 & 0 & 1 & 0 & 1 & 0 
\\
0 & 1 & 0 & 1 & 0 & 1 & 0 & 1
\end{array}
\right].
$$
The semiample cone is $\SAmple(X)=(\QQ_{\ge0})^2$
and the anticanonical divisor is $-\mathcal{K}_X=(2,2)$.
In particular, $X$ is a Fano variety.
\end{proof}

\begin{proof}[Proof of Theorems~\ref{thm:main1}, 
\ref{thm:main2} and~\ref{thm:main3}]
The preceeding analysis of the cases of 
Proposition~\ref{prop:smooth-rho2}
shows that every smooth rational non-toric 
projective variety of Picard number two 
coming with a torus action of complexity one
occurs in Theorem~\ref{thm:main1}
and, among these, the Fano ones in 
Theorem~\ref{thm:main2} 
and the truly almost Fano ones in
Theorem~\ref{thm:main3}.
Comparing the defining data, one directly 
verifies that any two different listed 
varieties are not isomorphic to each other.
Finally, using Remark~\ref{rem:Qfact} 
one explicitly checks that indeed all varieties 
listed in Theorem~\ref{thm:main1} are smooth.
\end{proof}

\section{Duplicating free weights}
\label{section:finite}

As mentioned in the introduction, there are (up to 
isomorphy) just two smooth non-toric projective 
varieties with a torus action of complexity one 
and Picard number one, namely the smooth projective 
quadrics in dimensions three and four.
In Picard number two we obtained examples in every 
dimension and this even holds when we restrict to 
the Fano case. 
Nevertheless, also in Picard number two we will 
observe a certain finiteness feature: 
each Fano variety listed in Theorem~\ref{thm:main2}
arises from a smooth, but not necessarily Fano, variety
of dimension at most seven via an iterated generalized 
cone construction.
In terms of the Cox ring the generalized cone construction
simply means \emph{duplicating a free weight}.

For the precise treatment, the setting of bunched rings 
$(R,\mathfrak{F},\Phi)$ is most appropriate.
Recall from~\cite[Sec.~3.2]{ArDeHaLa} that $R$ is 
a normal factorially 
$K$-graded $\KK$-algebra, $\mathfrak{F}$ a system of 
pairwise non-associated $K$-prime generators for $R$ 
and $\Phi$ a certain collection of polyhedral cones 
in $K_\QQ$ defining an open set 
$\wh{X} \subseteq \ol{X} = \Spec \, R$
with a good quotient $X = \wh{X} \quot H$ 
by the action of the quasitorus $H = \Spec \, \KK[K]$ 
on $\ol{X}$.
Dimension, divisor class group and Cox ring of $X$ are 
given by
$$ 
\dim(X) \ = \ \dim(R) - \dim(K_\QQ),
\qquad
\Cl(X) \ = \ K,
\qquad \
\mathcal{R}(X) \ = \ R.
$$
We call $X=X(R,\mathfrak{F},\Phi)$ the variety associated 
with the bunched ring $(R,\mathfrak{F},\Phi)$.
This construction yields for example all normal complete 
$A_2$-varieties with a finitely generated Cox ring,
e.g.~Mori dream spaces.
Observe that our Construction~\ref{constr:RAPu} presented 
earlier is a special case; it yields precisely the rational 
projective varieties with a torus action of complexity one.
The approach via bunched rings allows in particular 
an algorithmic treatment~\cite{HaKe:2014}.

\begin{construction}
\label{const:duplicate}
Let $R = \KK[T_1,\ldots,T_r] / \langle g_1, \ldots, g_s \rangle$
a $K$-graded algebra presented by $K$-homogeneous
generators $T_i$ and relations $g_j \in \KK[T_1,\ldots,T_{r-1}]$.
By \emph{duplicating the free weight $\deg(T_r)$} we mean passing
from $R$ to the $K$-graded algebra 
$$ 
R'
\ := \
\KK[T_1,\ldots,T_r,T_{r+1}] / \langle g_1, \ldots, g_s \rangle,
\qquad
\deg(T_{r+1}) \ := \ \deg(T_r) \ \in \ K, 
$$
where $g_j \in \KK[T_1,\ldots,T_{r-1}] \subseteq \KK[T_1,\ldots,T_r,T_{r+1}]$.
If in this situation $(R,\mathfrak{F},\Phi)$ is 
a bunched ring with $\mathfrak{F} = (T_1,\ldots,T_r)$,
then $(R',\mathfrak{F}',\Phi)$ is a bunched ring with 
$\mathfrak{F}' = (T_1,\ldots,T_r,T_{r+1})$.
\end{construction}

\begin{proof}
The $\KK$-algebra $R'$ is normal and, 
by~\cite[Thm.~1.4]{Be}, factorially 
$K$-graded.
Obviously, the $K$-grading is almost free
in the sense of~\cite[Def.~3.2.1.1]{ArDeHaLa}.
Moreover, $(R,\mathfrak{F})$ and 
$(R',\mathfrak{F}')$ have the same sets of
generator weights in the common grading 
group $K$ and the collection of 
projected $\mathfrak{F}'$-faces 
equals the collection of 
projected $\mathfrak{F}$-faces.
We conclude that $\Phi$ is a true $\mathfrak{F}'$-bunch
in the sense of~\cite[Def.~3.2.1.1]{ArDeHaLa}
and thus $(R',\mathfrak{F}',\Phi)$ is a bunched 
ring.
\end{proof}

The word ``free'' in  Construction~\ref{const:duplicate}
indicates that the variable $T_r$ does not occur in 
the relations $g_j$.
In the above setting, we say that $R$ is a complete intersection,
for short c.i., if $R$ is of dimension $r-s$. 
Here are the basic features of the procedure.

\begin{proposition}
\label{prop:duplicate}
Let $(R',\mathfrak{F}',\Phi)$ arise from the bunched ring 
$(R,\mathfrak{F},\Phi)$ via Construction~\ref{const:duplicate}.
Set $X':=X(R',\mathfrak{F}',\Phi)$ and 
$X:=X(R,\mathfrak{F},\Phi)$.
\begin{enumerate}
\item
We have $\dim(X') = \dim(X)+1$.
\item
The cones of semiample divisor classes satisfy
$\SAmple(X') = \SAmple(X)$.
\item
The variety $X'$ is smooth if and only if $X$ is smooth.
\item
The ring $R'$ is a c.i. if and only if $R$ is a c.i.. 
\item
If $R$ is a c.i., $\deg(T_r)$ semiample and 
$X$ Fano, then $X'$ is Fano.
\end{enumerate}
\end{proposition}

\begin{proof}
By construction, $\dim(R') = \dim(R)+1$ holds.
Since $R$ and $R'$ have the same grading 
group~$K$, we obtain~(i).
Moreover, $R$ and $R'$ have the same
defining relations $g_j$, hence we have~(iv).
According to~\cite[Prop.~3.3.2.9]{ArDeHaLa}, 
the semiample cone is the intersection of all
elements of $\Phi$ and thus~(ii) holds.

To obtain the third assertion, 
we show first that $\wh{X}'$ is smooth
if and only if $\wh{X}$ is smooth.
For every relevant $\mathfrak{F}$-face 
$\gamma_0 \preceq \QQ_{\ge 0}^r$ consider
$$ 
\gamma_0' 
\ := \ 
\gamma_0 + \cone(e_{r+1}),
\qquad
\gamma_0''
\ := \
\cone(e_{i}; \; 1 \le i < r, \ e_i \in \gamma_0 ) 
+ 
\cone(e_{r+1}).
$$ 
Then $\gamma_0,\gamma_0',\gamma_0''  \preceq \QQ_{\ge 0}^{r+1}$
are relevant $\mathfrak{F}'$-faces and,
in fact, all relevant $\mathfrak{F}'$-faces are
of this form.
Since the variables $T_r$ and $T_{r+1}$ do not appear
in the relations $g_j$, we see that a stratum
$\ol{X}(\gamma_0)$ is smooth if and only the strata
$\ol{X}'(\gamma_0)$,  $\ol{X}'(\gamma_0')$ and 
$\ol{X}'(\gamma_0'')$ are smooth.
Now~\cite[Cor.~3.3.1.11]{ArDeHaLa} gives~(iii).

Finally, we show~(v). 
As we have complete intersection Cox rings,
\cite[Prop.~3.3.3.2]{ArDeHaLa} applies 
and we obtain
$$
-\mathcal{K}_{X'} 
\ = \ 
\sum_{i=1}^{r+1} \deg(T_i) - \sum_{j=1}^s \deg(g_j)
\ = \ 
-\mathcal{K}_{X} + \deg(T_{r+1}).
$$
Since $X$ and $X'$ share the same ample cone,
we conclude that ampleness of  $-\mathcal{K}_{X}$ 
implies ampleness of $-\mathcal{K}_{X'}$,
\end{proof}

We interprete the duplication of free weights
in terms of birational geometry: it turns out to
be a composition of a Mori fiber space, a series 
flips and a birational divisorial contraction, 
where the two contractions both are elementary;
see~\cite{Ca} for a detailed study the latter 
type of maps in the context of general smooth 
Fano 4-folds.

\begin{proposition}
\label{prop:duplicate-geom}
Let $(R',\mathfrak{F}',\Phi)$ arise from the bunched ring 
$(R,\mathfrak{F},\Phi)$ via Construction~\ref{const:duplicate}.
Set $X':=X(R',\mathfrak{F}',\Phi)$ and 
$X:=X(R,\mathfrak{F},\Phi)$.
Assume that $X$ is $\QQ$-factorial.
Then there is a sequence
$$
X 
\ \longleftarrow \ 
\widetilde{X}_1 
\ \dashrightarrow \ \ldots \ \dashrightarrow \
\widetilde{X}_t 
\ \longrightarrow \ 
X',
$$
where $\widetilde{X}_1 \to X$ is a Mori fiber space 
with fibers $\PP_1$,
every $\widetilde{X}_i \dashrightarrow \widetilde{X}_{i+1}$ 
is a flip
and $\widetilde{X}_t \to X'$ is the contraction of 
a prime divisor.
If $\deg(T_r) \in K$ is Cartier, 
then $\widetilde{X}_1 \to X$ is 
the $\PP_1$-bundle associated with the divisor 
on $X$ corresponding to $T_r$.
\end{proposition}

\begin{proof}
In order to define $\widetilde{X}_1$, we consider 
the canonical toric embedding $X \subseteq Z$ in
the sense of~\cite[Constr.~3.2.5.3]{ArDeHaLa}.
Let $\Sigma$ be the fan of $Z$ and  
$P = [v_1,\ldots,v_r]$ be the matrix having the 
primitive generators $v_i \in \ZZ^{n}$ of the rays 
of $\Sigma$ as its columns. 
Define a further matrix
$$ 
\widetilde{P}
\ := \ 
\left[
\begin{array}{cccccc}
v_1 & \ldots & v_{r-1} & v_r & 0 & 0 
\\
 0  & \ldots &  0    & -1  & 1 & -1 
\\
\end{array}
\right].
$$
We denote the columns of  $\widetilde{P}$ by 
$\widetilde{v}_1, \ldots, \widetilde{v}_r,
\widetilde{v}_+,\widetilde{v}_- \in \ZZ^{n+1}$,
write $\varrho_+$, $\varrho_-$ for the rays through 
$\widetilde{v}_+$, $\widetilde{v}_-$ and define 
a fan 
$$ 
\widetilde{\Sigma}_1
\ := \ 
\{
\widetilde{\sigma} + \varrho_+, \,
\widetilde{\sigma} + \varrho_-, \,
\widetilde{\sigma}; \;
\sigma \in \Sigma
\},
\qquad\qquad
\widetilde{\sigma}
\ := \ 
\cone(\widetilde{v}_{i}; v_i \in \sigma).
$$
The projection $\ZZ^{n+1} \to \ZZ^{n}$ is a map
of fans $\widetilde{\Sigma}_1 \to  \Sigma$.
The associated toric morphism $\widetilde{Z}_1 \to Z$
has fibers $\PP_1$.
If the toric divisor $D_r$ corresponding to 
the ray through $v_r$ is Cartier, then 
$\widetilde{Z}_1 \to Z$ is the $\PP_1$-bundle
associated with $D_r$.
We define $\widetilde{X}_1 \subseteq \widetilde{Z}_1$ 
to be the preimage of $X \subseteq Z$.
Then $\widetilde{X}_1 \to X$ has fibers $\PP_1$.
If $\deg(T_r)$ is Cartier,
then so is $D_r$ and hence
$\widetilde{X}_1 \to X$ inherits the $\PP_1$-bundle structure.

Now we determine the Cox ring of the 
variety~$\widetilde{X}_1$.
For this, observe that the projection
$\ZZ^{r+2} \to \ZZ^{r}$ defines a lift of 
$\widetilde{Z}_1 \to Z$ to the toric 
characteristic spaces and thus leads 
to the commutative diagram
$$ 
\xymatrix{
{\widetilde{\pi}^\sharp(\widetilde{X}_1)}
\ar@{}[r]|\subseteq
\ar[d]_{\widetilde{\pi}}
&
{\widetilde{W}_1}
\ar[r]
\ar[d]_{\widetilde{\pi}}
& 
W
\ar[d]^{\pi}
&
{\pi^\sharp(X)}
\ar@{}[l]|\supseteq
\ar[d]^{\pi}
\\
{\widetilde{X}_1}
\ar@{}[r]|\subseteq
&
{\widetilde{Z}_1}
\ar[r]
&
Z
&
X
\ar@{}[l]|\supseteq
}
$$ 
where $\widetilde{\pi}^\sharp(\widetilde{X}_1)$
and $\pi^\sharp(X)$ denote the proper transforms 
with respect to the downwards toric morphisms.  
Pulling back the defining equations of 
$\pi^\sharp(X) \subseteq W$, we see that 
$\widetilde{\pi}^\sharp(\widetilde{X}_1) \subseteq \widetilde{W}_1$
has coordinate algebra
$\widetilde{R} :=R[S^+,S^-]$ 
graded by $\widetilde{K} := K\times \ZZ$ via
$$
\deg(T_i) := (w_i,0), 
\qquad 
w^+ := \deg(S^+) := (w_r,1), 
\qquad 
w^- := \deg(S^-) := (0,1),
$$
where $w_i := \deg(T_i)\in K$.
The $\KK$-algebra $\widetilde{R}$ is normal and, 
by~\cite[Thm.~1.4]{Be}, factorially 
$\widetilde{K}$-graded.
Moreover the $\widetilde{K}$-grading is 
almost free, as the $K$-grading of $R$ 
has this property
and 
$\widetilde{\mathfrak{F}}=(T_1, \ldots, T_r, S^+, S^-)$
is a system of pairwise non-associated 
$\widetilde{K}$-prime generators.
We conclude that $\widetilde{R}$ is the
Cox ring of $\widetilde{X}_1$.

Next we look for the defining bunch of cones 
for $\widetilde{X}_1$.
Observe that $K$ sits inside $\widetilde{K}$ 
as $K \times \{0\}$.
With $\theta := \SAmple(X)\times\{0\}$ we 
obtain a GIT-cone 
$\theta_1  := \cone(\theta,w^+)  \cap  \cone(\theta,w^-)$
of the $\widetilde{K}$-graded ring $\widetilde{R}$.
The associated bunch $\widetilde{\Phi}_1$ consists 
of all cones of the form
$$
\widetilde{\tau} + \cone(w^+), 
\qquad
\widetilde{\tau} + \cone(w^-), 
\qquad
\widetilde{\tau} + \cone(w^+, w^-),
$$
where $\widetilde{\tau} =\tau \times \{0\}$, $\tau \in \Phi$. 
Since $\Phi$ is a true bunch, so is $\widetilde{\Phi}_1$.
Together we obtain a bunched ring 
$(\widetilde{R}, \widetilde{\mathfrak{F}}, \widetilde{\Phi}_1)$.
By construction, the fan corresponding to 
$\widetilde{\Phi}_1$ via Gale duality is $\widetilde{\Sigma}_1$.
We conclude that $\widetilde{X}_1$ is the variety 
associated with $(\widetilde{R}, \widetilde{F}, \widetilde{\Phi}_1)$
and $\widetilde{X}_1 \subseteq \widetilde{Z}_1$ is 
the canonical toric embedding.

Observe that $\widetilde{X}_1 \to X$ corresponds to the 
passage from the GIT-cone $\theta_1$ 
to the facet $\theta$.
In particular, we see that $\widetilde{X}_1 \to X$ is a 
Mori fiber space.
To obtain the flips and the final divisorial contraction,
we consider the full GIT-fan.
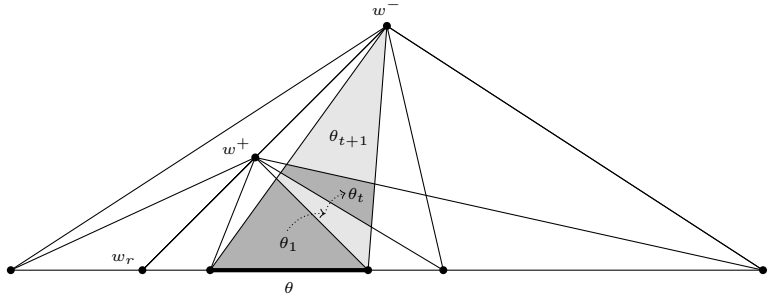
\begin{center}
\begin{tikzpicture}[scale=0.5]
   \path[fill=gray!20!] (1.8,0)--(6,0)--(6.5,6.5)--(1.8,0);  
   \path[fill=gray!60!] (1.8,0)--(6,0)--(57/16, 39/16)--(1.8,0);  
   \path[fill=gray!60!] (1713/466, 1209/466)--(207/34, 39/34)--(105/17, 39/17)--(372/97, 273/97)--(1371/370, 195/74); 
   \coordinate[] (w1) at (0,0);
   \node at (-0.5,0.3) {{\tiny{$w_r$}}};
   \fill (w1) circle (3pt);
   \coordinate[] (w2) at (1.8,0);
   \fill (w2) circle (3pt);
   \coordinate[] (w3) at (6,0);
   \fill (w3) circle (3pt);
   \coordinate[] (w4) at (8,0);
   \fill (w4) circle (3pt);
   \coordinate[] (w5) at (16.5,0);
   \fill (w5) circle (3pt);
   \coordinate (wp) at (3,3);   
   \node at (2.5,3.3) {\tiny{$w^+$}};
   \fill (wp) circle (3pt);
   \coordinate(wm) at (6.5,6.5);
   \node at (6.5,7) {\tiny{$w^-$}};
   \fill (wm) circle (3pt);
  \coordinate(w0) at (-3.5,0);
   \fill (w0) circle (3pt);
   \draw (w5) -- (wm) -- (wp) -- (w1);
   \draw (wp) -- (w1);
   \draw (wp) -- (w2);
   \draw (wp) -- (w3);
   \draw (wp) -- (w4);
   \draw (wp) -- (w5);
   \draw (wm) -- (w1);
   \draw (wm) -- (w2);
   \draw (wm) -- (w3);
   \draw (wm) -- (w4);
   \draw (wm) -- (w5);
   \draw (w0) -- (w5);
   \draw (w0) -- (wp);
   \draw (w0) -- (wm);
   \draw[decorate, ultra thick] (w2) -- (w3);
   \node at (3.9,-0.45) {\tiny{{$\theta$}}};
   \node at (3.9,0.7) {{\tiny{$\theta_1$}}};
   \node at (5.7,2) {{\tiny{$\theta_t$}}};
   \node at (5.5,3.5) {{\tiny{$\theta_{t+1}$}}};
   \path[densely dotted, ->] (3.85,1.05) edge [bend left] (4.85, 1.5); 
   \path[densely dotted, ->] (4.85, 1.5) edge [bend left] (5.4, 2.05); 
\end{tikzpicture}
\end{center}
Important are the GIT-cones inside $\theta + \cone(w^-)$.
There we have the facet $\theta$ 
and the semiample cone $\theta_1$ of $\widetilde{X}_1$.
Proceding in the direction of $w^-$,
we come across other full-dimensional GIT-cones,
say $\theta_2,\ldots,\theta_{t+1}$.
This gives a sequence of flips 
$\widetilde{X}_1\dashrightarrow \ldots\dashrightarrow\widetilde{X}_{t}$,
where $\widetilde{X}_i$ is the variety with 
semiample cone $\theta_i$.
Passing from $\theta_t$ to $\theta_{t+1}$
gives a morphism $\widetilde{X}_{t} \to \widetilde{X}_{t+1}$
contracting the prime divisor corresponding to the 
variable $S^-$ of the Cox ring $\widetilde{R}$ of
$\widetilde{X}_{t}$.
Note that $\widetilde{X}_{t+1}$ is $\QQ$-factorial,
as it is the GIT-quotient associated 
with a full-dimensional chamber.

We show $\widetilde{X}_{t+1}\cong X'$.
Recall that $X'$ arises from $X$ by duplicating the 
weight $\deg(T_r)$.
We have $\Cl(X') = K$ and the Cox ring 
$R'=R[T_{r+1}]$ of $X'$ is $K$-graded via 
$\deg(T_i)=w_i$ for $i=1,\ldots,r$ and $\deg(T_{r+1})=w_r$.
In particular, the fan of the canonical 
toric ambient variety of $X'$ has as its primitive
ray generators the columns of the matrix
$$
P' 
\ = \
\left[
\begin{array}{ccccc}
v_1 & \ldots & v_{r-1} & v_r & 0
\\
 0  & \ldots &  0    & -1  & 1
\\
\end{array}
\right].
$$
On the other hand, the canonical toric ambient 
variety $\widetilde{Z}_{t+1}$ of $\widetilde{X}_{t+1}$ 
is obtained from $\widetilde{Z}_{t}$ by contracting 
the divisor corresponding to the ray $\varrho_-$.
Hence $P'$ is as well the primitive generator 
matrix for the fan of $\widetilde{Z}_{t+1}$.
We conclude
$$
\Cl(\widetilde{X}_{t+1}) 
\ = \ 
\ZZ^{r+1} /\text{ im}((P')^*) 
\ = \ 
\Cl(X') 
\ = \ 
K.
$$
Similarly, we compare the Cox rings of 
$\widetilde{X}_{t+1}$ and $X'$.
Let $\widetilde{Z}_t$ denote the canonical 
toric ambient variety of $\widetilde{X}_t$.
Then the projection
$\ZZ^{r+2} \to \ZZ^{r+1}$ defines a lift of 
$\widetilde{Z}_t \to \widetilde{Z}_{t+1}$ 
to the toric characteristic spaces 
and thus leads to the commutative diagram
$$ 
\xymatrix{
{\widetilde{\pi}^\sharp(\widetilde{X}_t)}
\ar@{}[r]|\subseteq
\ar[d]_{\widetilde{\pi}}
&
{\widetilde{W}_t}
\ar[r]
\ar[d]_{\widetilde{\pi}}
& 
{\widetilde{W}_{t+1}}
\ar[d]^{\pi}
&
{\pi^\sharp(\widetilde{X}_{t+1})}
\ar@{}[l]|\supseteq
\ar[d]^{\pi}
\\
{\widetilde{X}_t}
\ar@{}[r]|\subseteq
&
{\widetilde{Z}_t}
\ar[r]
&
{\widetilde{Z}_{t+1}}
&
{\widetilde{X}_{t+1}}
\ar@{}[l]|\supseteq
}
$$ 
where the proper transforms
$\widetilde{\pi}^\sharp(\widetilde{X}_t)$
and $\pi^\sharp(\widetilde{X}_{t+1})$ 
are the characteristic spaces of 
$\widetilde{X}_t$ and $\widetilde{X}_{t+1}$
respectively and the first is mapped 
onto the second one.
We conclude that the Cox ring of 
$\widetilde{X}_{t+1}$ is $R[S^+]$ 
graded by
$\deg(T_i)=w_i$ for $i = 1, \ldots, r$
and $\deg(S^+)=w_{r}$
and thus is isomorphic to the Cox 
ring $R'$ of $X'$.

The final step is to compare the defining 
bunches of cones $\widetilde{\Phi}_{t+1}$ 
of $\widetilde{X}_{t+1}$ and $\Phi'$ of $X'$.
For this, observe that the fan of the toric 
ambient variety $\widetilde{Z}_{t+1}$ 
contains the cones 
$\widetilde{\sigma} + \varrho_+$,
where $\sigma \in \Sigma$.
Thus, every $\tau \in \Phi'$
belongs to $\widetilde{\Phi}_{t+1}$.
We conclude 
$$
\SAmple(\widetilde{X}_{t+1})
\ \subseteq \ 
\SAmple(X').
$$
Since $\widetilde{X}_{t+1}$ is $\QQ$-factorial,
its semiample cone is of full dimension.
Both cones belong to the GIT-fan,
hence we see that the above inclusion is in
fact an equality.
Thus $\widetilde{\Phi}_{t+1}$ equals $\Phi'$.
\end{proof}

We return to the Fano varieties of Theorem~\ref{thm:main2}.
We first list the (finitely many) examples which do not allow 
duplication of a free weight and then present the 
starting models for constructing the Fano varieties 
via duplication of weights.

\begin{proposition}
\label{prop:noisodivs}
The varieties of Theorem~\ref{thm:main2} containing
no divisors with infinite general isotropy are 
precisely the following ones.

\begin{center}
{\small
\setlength{\tabcolsep}{4pt}
\begin{longtable}{ccccc}
No.
&
\small{$\mathcal{R}(X)$}
&
\small{$[w_1,\ldots, w_r]$}
&
\small{$-\mathcal{K}_X$}
&
\small{$\dim(X)$}
\\
\toprule
1
&
$
\frac
{\KK[T_1, \ldots , T_7]}
{\langle T_{1}T_{2}T_{3}^2+T_{4}T_{5}+T_6T_7 \rangle}
$
&
\tiny{
\setlength{\arraycolsep}{2pt}
$
\left[\!\!
\begin{array}{ccccccc}
0 & 0 & 1 & 1 & 1 & 1 & 1 
\\ 
1 & 1 & 0 & 1 & 1 &1 & 1
\end{array}
\!\!\right]
$
}
&
\tiny{
\setlength{\arraycolsep}{2pt}
$
\left[\!\!
\begin{array}{c}
3
\\
4
\end{array}
\!\!\right]
$
}
&
\small{$4$}
\\
\midrule
2
&
$
\frac
{\KK[T_1, \ldots , T_7]}
{\langle T_{1}T_{2}T_{3}+T_{4}T_{5}+T_6T_7 \rangle}
$
&
\tiny{
\setlength{\arraycolsep}{2pt}
$
\left[\!\!
\begin{array}{ccccccc}
0 & 0 & 1 & 1 & 0 & 1 & 0 
\\ 
1 & 1 & 0 & 1 & 1 & 1 & 1
\end{array}
\!\!\right]
$
}
&
\tiny{
\setlength{\arraycolsep}{2pt}
$
\left[\!\!
\begin{array}{c}
2
\\
4
\end{array}
\!\!\right]
$
}
&
\small{$4$} 
\\
\midrule
3
&
$
\frac{\KK[T_1, \ldots , T_6]}
{\langle T_{1}T_{2}T_{3}^2+T_{4}T_{5}+T_{6}^2 \rangle}
$
&
\tiny{
\setlength{\arraycolsep}{2pt}
$
\left[\!\!
\begin{array}{cccccc}
0 & 0 & 1 & 1 & 1 & 1 
\\ 
1 & 1 & 0 & 1 & 1 & 1
\end{array}
\!\!\right]
$
}
&
\tiny{
\setlength{\arraycolsep}{2pt}
$
\left[\!\!
\begin{array}{c}
2
\\
3
\end{array}
\!\!\right]
$
}
&
\small{$3$} 
\\
\midrule
4.A
&
$
\frac
{\KK[T_1, \ldots , T_6]}
{\langle T_{1}T_{2}+T_{3}T_{4}+T_5T_{6} \rangle}
$
&
\tiny{
\setlength{\arraycolsep}{2pt}
$
\left[\!\!
\begin{array}{cccccc}
0 & 1 & 0 & 1 & 0 & 1  
\\ 
1 & 0 & 1 & 0 & 1 & 0 
\end{array}
\!\!\right]
$
}
&
\tiny{
\setlength{\arraycolsep}{2pt}
$
\left[\!\!
\begin{array}{c}
2 
\\
2
\end{array}
\!\!\right]
$
}
&
\small{$3$}
\\
\midrule
4.B
&
$
\frac
{\KK[T_1, \ldots , T_6]}
{\langle T_{1}T_{2}^2+T_{3}T_{4}+T_5T_{6} \rangle}
$
&
\tiny{
\setlength{\arraycolsep}{2pt}
$
\left[\!\!
\begin{array}{cccccc}
0 & 1 & 1 & 1 & 1 & 1 
\\ 
1 & 0 & 1 & 0 & 1 & 0 
\end{array}
\!\!\right]
$
}
&
\tiny{
\setlength{\arraycolsep}{2pt}
$
\left[\!\!
\begin{array}{c}
3
\\
2
\end{array}
\!\!\right]
$
}
&
\small{$3$}
\\
\midrule
4.C
&
$
\frac
{\KK[T_1, \ldots , T_6]}
{\langle T_{1}T_{2}^2+T_{3}T_{4}^2+T_5T_{6}^2 \rangle}
$
&
\tiny{
\setlength{\arraycolsep}{2pt}
$
\left[\!\!
\begin{array}{cccccc}
0 & 1 & 0 & 1 & 0 & 1  
\\ 
1 & 0 & 1 & 0 & 1 & 0 
\end{array}
\!\!\right]
$
}
&
\tiny{
\setlength{\arraycolsep}{2pt}
$
\left[\!\!
\begin{array}{c}
1
\\
2
\end{array}
\!\!\right]
$
}
&
\small{$3$}
%
%
%
%
\\
\midrule
13
&
$
\begin{array}{c}
\frac
{\KK[T_1, \ldots , T_8]}
{
\left\langle 
\begin{array}{l}
\scriptstyle T_{1}T_{2}+T_{3}T_{4}+T_5T_{6},
\\[-3pt]
\scriptstyle \lambda T_{3}T_{4}+T_{5}T_{6}+T_{7}T_{8} 
\end{array}
\right\rangle
}
\\
\scriptstyle \lambda \in \KK^*  \setminus \{1\}
\end{array}
$
&
\tiny{
\setlength{\arraycolsep}{2pt}
$
\left[\!\! 
\begin{array}{cccccccc}
1 & 0 & 1 & 0 & 1 & 0 & 1 & 0 
\\ 
0 & 1 & 0 & 1 & 0 & 1 & 0 & 1
\end{array}
\!\!\right]
$
}
&
\tiny{
\setlength{\arraycolsep}{2pt}
$
\left[\!\! 
\begin{array}{c}
2
\\
2
\end{array}
\!\!\right]
$
}
&
\small{$4$}
\\
\bottomrule
\end{longtable}
}
\end{center}

\end{proposition}

\begin{proof}
For a $T$-variety $X = X(A,P,u)$, the divisors 
having infinite general $T$-isotropy are 
precisely the vanishing sets of the variable 
$S_k$. 
Thus we just have to pick out the cases with $m=0$
from Theorem~\ref{thm:main2}. 
\end{proof}

\begin{theorem}
\label{thm:duplicate}
Let $X$ be a smooth rational Fano variety with a torus 
action of complexity one and Picard number two.
If there is a prime divisor with infinite general isotropy
on $X$, then $X$ arises via iterated duplication of the 
free weight~$w_r$ from one of the following varieties~$Y$.
\begin{center}
{\small
\setlength{\tabcolsep}{4pt}
\begin{longtable}{ccccc}
No.
&
\small{$\mathcal{R}(Y)$}
&
\small{$[w_1,\ldots, w_r]$}
&
\small{$u$}
&
\small{$\dim(Y)$}
\\
\toprule
4.A
&
$
\frac
{\KK[T_1, \ldots , T_6, S_1]}
{\langle T_{1}T_{2}+T_{3}T_{4}+T_5T_{6} \rangle}
$
&
\tiny{
\setlength{\arraycolsep}{2pt}
$
\left[\!\!
\begin{array}{cccccc|c}
0 & 1 & 0 & 1 & 0 & 1 & 0 
\\ 
1 & 0 & 1 & 0 & 1 & 0 & 1
\end{array}
\!\!\right]
$
}
&
\tiny{
\setlength{\arraycolsep}{2pt}
$
\left[\!\!
\begin{array}{c}
1
\\
1
\end{array}
\!\!\right]
$
}
&
\small{$4$}
\\
\midrule
4.A
&
$
\frac
{\KK[T_1, \ldots , T_6, S_1,S_2]}
{\langle T_{1}T_{2}+T_{3}T_{4}+T_5T_{6} \rangle}
$
&
\tiny{
\setlength{\arraycolsep}{2pt}
$
\left[\!\!
\begin{array}{cccccc|cc}
0 & 1 & 0 & 1 & 0 & 1 & -1 & 0
\\ 
1 & 0 & 1 & 0 & 1 & 0 & 1 & 1
\end{array}
\!\!\right]
$
}
&
\tiny{
\setlength{\arraycolsep}{2pt}
$
\left[\!\!
\begin{array}{c}
1
\\
1
\end{array}
\!\!\right]
$
}
&
\small{$5$}
\\
\midrule
4.B
&
$
\frac
{\KK[T_1, \ldots , T_6, S_1]}
{\langle T_{1}T_{2}^2+T_{3}T_{4}+T_5T_{6} \rangle}
$
&
\tiny{
\setlength{\arraycolsep}{2pt}
$
\left[\!\!
\begin{array}{cccccc|c}
0 & 1 & 1 & 1 & 1 & 1 & 1 
\\ 
1 & 0 & 1 & 0 & 1 & 0 & 1 
\end{array}
\!\!\right]
$
}
&
\tiny{
\setlength{\arraycolsep}{2pt}
$
\left[\!\!
\begin{array}{c}
2
\\
1
\end{array}
\!\!\right]
$
}
&
\small{$4$}
\\
\midrule
4.C
&
$
\frac
{\KK[T_1, \ldots , T_6, S_1]}
{\langle T_{1}T_{2}^2+T_{3}T_{4}^2+T_5T_{6}^2 \rangle}
$
&
\tiny{
\setlength{\arraycolsep}{2pt}
$
\left[\!\!
\begin{array}{cccccc|c}
0 & 1 & 0 & 1 & 0 & 1 & 0 
\\ 
1 & 0 & 1 & 0 & 1 & 0 & 1 
\end{array}
\!\!\right]
$
}
&
\tiny{
\setlength{\arraycolsep}{2pt}
$
\left[\!\!
\begin{array}{c}
1
\\
1
\end{array}
\!\!\right]
$
}
&
\small{$4$}
\\
\midrule
5
&
$
\frac
{\KK[T_1, \ldots , T_6, S_1]}
{\langle T_{1}T_{2}+T_{3}^2T_{4}+T_5^2T_{6} \rangle}
$
&
\tiny{
\setlength{\arraycolsep}{2pt}
$
\begin{array}{c}
\left[\!\!
\begin{array}{cccccc|c}
0 & 2a+1 & a & 1 & a & 1 & 1
\\ 
1 & 1 & 1 & 0 & 1 & 0 & 0
\end{array}
\!\!\right]
\\[1em]
a \ge 0
\end{array}
$
}
&
\tiny{
\setlength{\arraycolsep}{2pt}
$
\left[\!\!
\begin{array}{c}
2a+2
\\
1
\end{array}
\!\!\right]
$
}
&
\small{$4$}
\\
\midrule
6
&
$
\frac
{\KK[T_1, \ldots , T_6, S_1]} 
{\langle T_{1}T_{2}+T_{3}T_{4}+T_5^2T_{6} \rangle}
$
&
\tiny{
\setlength{\arraycolsep}{2pt}
$
\begin{array}{c}
\left[\!\!
\begin{array}{cccccc|c}
0 & 2c+1 & a & b & c & 1 & 1 
\\ 
1 & 1 & 1 & 1 & 1 & 0 & 0 
\end{array}
\!\!\right]
\\[1em]
a, b, c \ge 0, \quad a<b,\\
a+b=2c+1
\end{array}
$
}
&
\tiny{
\setlength{\arraycolsep}{2pt}
$
\left[\!\!
\begin{array}{c}
2c+2
\\
1
\end{array}
\!\!\right]
$
}
&
\small{$4$}
\\
\midrule 
7
&
$
\frac
{\KK[T_1, \ldots , T_6, S_1]}
{\langle T_{1}T_{2}+T_{3}T_{4}+T_5T_{6} \rangle}
$
&
\tiny{
\setlength{\arraycolsep}{2pt}
$
\left[\!\!
\begin{array}{cccccc|c}
0 & 0 & 0 & 0 & -1 & 1 & 1
\\ 
1 & 1 & 1 & 1 & 1 & 1 & 0
\end{array}
\!\!\right]
$
}
&
\tiny{
\setlength{\arraycolsep}{2pt}
$
\left[\!\! 
\begin{array}{c}
1
\\
2
\end{array}
\!\!\right]
$
}
&
\small{$4$}
\\
\midrule
8
&
$
\frac
{\KK[T_1, \ldots , T_6, S_1,S_2]}
{\langle T_{1}T_{2}+T_{3}T_{4}+T_5T_{6} \rangle}
$
&
\tiny{
\setlength{\arraycolsep}{2pt}
$
\begin{array}{c}
\left[\!\! 
\begin{array}{cccccc|cc}
0 & 0 & 0 & 0 & 0 & 0 & 1 & 1
\\
1 & 1 & 1 & 1 & 1 & 1 & 0 & a
\end{array}
\!\!\right]
\\[1em]
a\in\{1,2,3\}
\end{array}
$
}
&
\tiny{
\setlength{\arraycolsep}{2pt}
$
\left[\!\! 
\begin{array}{c}
1
\\
a+1
\end{array}
\!\!\right]
$
}
&
\small{$5$} 
\\
\midrule
8
&
$
\frac
{\KK[T_1, \ldots , T_6, S_1,S_2,S_3]}
{\langle T_{1}T_{2}+T_{3}T_{4}+T_5T_{6} \rangle}
$
&
\tiny{
\setlength{\arraycolsep}{2pt}
$
\begin{array}{c}
\left[\!\! 
\begin{array}{cccccc|ccc}
0 & 0 & 0 & 0 & 0 & 0 & 1 & 1 & 1
\\
1 & 1 & 1 & 1 & 1 & 1 & 0 & a-1 & a
\end{array}
\!\!\right]
\\[1em]
a\in\{1,2\}
\end{array}
$
}
&
\tiny{
\setlength{\arraycolsep}{2pt}
$
\left[\!\! 
\begin{array}{c}
1
\\
a+1
\end{array}
\!\!\right]
$
}
&
\small{$6$} 
\\
\midrule
8
&
$
\frac
{\KK[T_1, \ldots , T_6, S_1,\ldots,S_4]}
{\langle T_{1}T_{2}+T_{3}T_{4}+T_5T_{6} \rangle}
$
&
\tiny{
\setlength{\arraycolsep}{2pt}
$
\left[\!\! 
\begin{array}{cccccc|cccc}
0 & 0 & 0 & 0 & 0 & 0 & 1 & 1 & 1 & 1
\\
1 & 1 & 1 & 1 & 1 & 1 & 0 & 0 & 0 & 1
\end{array}
\!\!\right]
$
}
&
\tiny{
\setlength{\arraycolsep}{2pt}
$
\left[\!\! 
\begin{array}{c}
1
\\
2
\end{array}
\!\!\right]
$
}
&
\small{$7$} 
\\
\midrule
9
&
$
\frac
{\KK[T_1, \ldots , T_6, S_1, S_2]}
{\langle T_{1}T_{2}+T_{3}T_{4}+T_5T_{6} \rangle}
$
&
\tiny{
\setlength{\arraycolsep}{2pt}
$
\begin{array}{c}
\left[\!\! 
\begin{array}{cccc|cc}
0 & a_2 & \ldots & a_6 & 1 & 1 
\\ 
1 & 1 & \ldots & 1 & 0 & 0
\end{array}
\!\!\right]
\\[1em]
0 \le a_3 \le a_5 \le a_6 \le a_4 \le a_2,\\
 a_2=a_3+a_4=a_5+a_6
\end{array}
$
}
&
\tiny{
\setlength{\arraycolsep}{2pt}
$
\left[\!\! 
\begin{array}{c}
a_2+1
\\
1
\end{array}
\!\!\right]
$
}
&
\small{$5$}
\\
\midrule
10
&
$
\frac
{\KK[T_1, \ldots , T_5, S_1]}
{\langle T_{1}T_{2}+T_{3}T_{4}+T_{5}^2 \rangle}
$
&
\tiny{
\setlength{\arraycolsep}{2pt}
$
\left[\!\! 
\begin{array}{ccccc|c}
1 & 1 & 1 & 1 & 1 & 0
\\ 
-1 & 1 & 0 & 0 & 0 & 1
\end{array}
\!\!\right]
$
}
&
\tiny{
\setlength{\arraycolsep}{2pt}
$
\left[\!\! 
\begin{array}{c}
2
\\
1
\end{array}
\!\!\right]
$
}
&
\small{$3$}
\\
\midrule
11
&
$
\frac
{\KK[T_1, \ldots , T_5, S_1,S_2]}
{\langle T_{1}T_{2}+T_{3}T_{4}+T_{5}^2 \rangle}
$
&
\tiny{
\setlength{\arraycolsep}{2pt}
$
\begin{array}{c}
\left[\!\! 
\begin{array}{ccccc|cc}
1 & 1 & 1 & 1 & 1 & 0 & a
\\ 
0 & 0 & 0 & 0 & 0 & 1 & 1
\end{array}
\!\!\right]
\\[1em]
a\in\{1,2\}
\end{array}
$
}
&
\tiny{
\setlength{\arraycolsep}{2pt}
$
\left[\!\! 
\begin{array}{c}
a+1
\\
1
\end{array}
\!\!\right]
$
}
&
\small{$4$}
\\
\midrule
11
&
$
\frac
{\KK[T_1, \ldots , T_5, S_1,S_2,S_3]}
{\langle T_{1}T_{2}+T_{3}T_{4}+T_{5}^2 \rangle}
$
&
\tiny{
\setlength{\arraycolsep}{2pt}
$
\left[\!\! 
\begin{array}{ccccc|ccc}
1 & 1 & 1 & 1 & 1 & 0 & 0 & 1
\\ 
0 & 0 & 0 & 0 & 0 & 1 & 1 & 1
\end{array}
\!\!\right]
$
}
&
\tiny{
\setlength{\arraycolsep}{2pt}
$
\left[\!\! 
\begin{array}{c}
2
\\
1
\end{array}
\!\!\right]
$
}
&
\small{$5$}
\\
\midrule 
12
&
$
\frac{\KK[T_1, \ldots , T_5, S_1,S_2]}
{\langle T_{1}T_{2}+T_{3}T_{4}+T_{5}^2 \rangle}
$
&
\tiny{
\setlength{\arraycolsep}{2pt}
$
\begin{array}{c}
\left[\!\!
\begin{array}{ccccc|cc}
1 & 1 & 1 & 1 & 1 & 0 & 0
\\ 
0 & 2c & a & b & c & 1 & 1
\end{array}
\!\!\right]
\\[1em]
0 \le a \le c \le b,  \ a+b=2c
\end{array}
$
}
&
\tiny{
\setlength{\arraycolsep}{2pt}
$
\left[\!\! 
\begin{array}{c}
1
\\
2c+1
\end{array}
\!\!\right]
$
}
&
\small{$4$}
\\
\bottomrule
\end{longtable}
}
\end{center}
For Nos.~4, 8 and~11, the variety $Y$ is Fano
and any iterated duplication of $w_r$ produces a 
Fano variety $X$.
For the remaining cases, the following table 
tells which~$Y$ are Fano and gives the 
characterizing condition when an iterated 
duplication of $w_r$ produces a Fano variety $X$: 
\begin{longtable}{cccccccc}
No.
&
5
&
6
&
7
&
9
&
10
&
12
\\
\toprule
$Y$ Fano
&
$a=0$
&
$c=0$
&
$\checkmark$
&
$a_2=0$
&
$\checkmark$
&
$c=0$
\\
\midrule
$X$ Fano
&
$m > 2a$
&
$m > 3c+1$
&
$m \le 3$
&
$m > 2a_2$
&
$m \le 2$
&
$m > 3c$
\\
\bottomrule
\end{longtable}
\end{theorem}

\begin{proof}
A $T$-variety $X = X(A,P,u)$ has a divisor 
with infinite general $T$-isotropy
if and only if $m\ge1$ holds.
In the cases 4.A, 4.B, 4.C, 5, 6, 7, 9, 10 and 12
we directly infer from Theorem~\ref{thm:main2}
that the examples with higher $m$ arise from those
listed in the table above via iterated 
duplication of $w_r$.

We still have to consider Nos.~8 and 11.
If $X$ is a variety of type~8,
then the condition for $X$ to be a Fano variety is
$$
4 + a_2 + \ldots, + a_m > ma_m,
$$
where $a_m=1,2,3$ and $0\le a_2 \le \ldots \le a_m$.
This is satisfied if and only if one of the following
conditions holds:
\begin{enumerate}
\item
$a_2 = \ldots = a_m \in \{1,2,3\}$.
\item
$a_2+1 = a_3 \ldots = a_m \in \{1,2\}$, with $m\ge3$.
\item
$a_2=a_3=0$ and $a_4 = \ldots = a_m = 1$, with $m\ge4$.
\end{enumerate}
Similarly for No.~11 the Fano condition in the table of Theorem~\ref{thm:main2}
is equivalent to the fulfillment of one of the following:
\begin{enumerate}
\item
$a_2 = \ldots = a_m \in \{1,2\}$.
\item
$a_2=0$ and $a_3 = \ldots = a_m = 1$, with $m\ge3$.
\end{enumerate}
In both cases this explicit characterization
makes clear that we are in the setting
of the duplication of a free weight.
\end{proof}

\begin{remark}
Consider iterated duplication of $w_r$
for a variety $X=X(A,P,u)$ as in Theorem~\ref{thm:duplicate}.
Recall that the effective cone of $X$ is decomposed as
$\tp \cup \tx \cup \tm$, where $\tx =\Ample(X)$.
Lemma \ref{lem:tau}~(i) says $w_r \not\in \tx$
and thus we have a unique $\kappa \in \{ \tp, \tm\}$
with $w_r \notin \kappa$.
Then the number of flips per duplication step 
equals
$$
|\{\cone(w_{ij}), \cone(w_k); \ w_{ij},w_k \in \kappa\}|-1.
$$
In particular, for Nos.~4.A, 4.B, 4.C, 8, 11, 9 with $a_i =0$, 
12 with $b=0$ the duplications steps require no flips.
\end{remark}

\begin{remark}
\label{rem:toric-not}
For toric Fano varieties, 
there is no statement like 
Corollary~\ref{cor:duplicate}. 
Recall from~\cite{BeHa:2004} that all 
smooth projective toric varieties $Z$ 
with $\Cl(Z)=\ZZ^2$ admit
a description via the following data:
\begin{itemize}
\item
weight vectors $w_1 := (1,0)$ and $w_i := (b_i,1)$ with
$0=b_n < b_{n-1} < \ldots < b_2$,
\item
multiplicities $\mu_i := \mu(w_i)\ge1$,
where $\mu_1\ge2$ and $\mu_2+\ldots+\mu_n\ge2$.
\end{itemize}
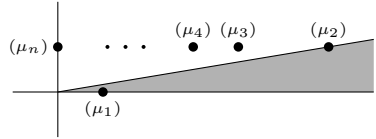
\begin{center}
\begin{tikzpicture}[scale=0.6]
   \path[fill=gray!60!] (0,0)--(7,0)--(7,1+1/6)--(0,0);  
   \coordinate[] (w1) at (1,0);
   \fill (w1) circle (3pt);
   \node[below] at (w1) {\tiny{$(\mu_1)$}};
   \coordinate[] (w2) at (6,1);
   \fill (w2) circle (3pt);
   \node[above] at (w2) {\tiny{$(\mu_2)$}};
   \coordinate[] (w3) at (4,1);
   \fill (w3) circle (3pt);
   \node[above] at (w3) {\tiny{$(\mu_3)$}};
   \coordinate[] (w4) at (3,1);
   \fill (w4) circle (3pt);
   \node[above] at (w4) {\tiny{$(\mu_4)$}};
   \coordinate[] (w5) at (0,1);   
   \node[left] at (w5) {\tiny{$(\mu_n)$}};
   \fill (w5) circle (3pt);
   \draw (-1,0) -- (7,0);
   \draw (0,-1) -- (0,2);
   \draw (0,0) -- (7,1+1/6);
   \fill (1.1, 1) circle (1.25pt);
   \fill (1.5, 1) circle (1.25pt);
   \fill (1.9, 1) circle (1.25pt);
\end{tikzpicture}
\end{center}
The variety $Z$ arises from the bunched polynomial ring 
$(R,\mathfrak{F},\Phi)$,
where $R$ equals $\KK[S_{ij} ;\, 1 \le i \le n, 1 \le j \le \mu_i ]$
with the system of generators
$\mathfrak{F}=(S_{11},\ldots,S_{n\mu_n})$ 
and the bunch
$\Phi = \{ \cone(w_{1},w_{i}); i=2,\ldots,n \}.$
In this setting $Z$ is Fano if and only if
$$
b_2(\mu_3+\ldots+\mu_n) < \mu_1 + \mu_3b_3 + \ldots + \mu_{n-1}b_{n-1}.
$$
For any $n\in\ZZ_{\ge4}$ and $i=2,\ldots,n$ set $\mu_i := 1$
and $w_i := (n-i,1)$.
Then, with $\mu_1 := 2$ we obtain a smooth (non-Fano) toric 
variety $Z_n'$ of Picard number two and dimension $n-1$.
Moreover, for $\mu_1 := 1+(n-2)(n-1)/2$ we obtain
a smooth toric Fano variety $Z_n$ of Picard number two
that is Fano and is obtained from $Z_n'$ via iterated 
duplication of $w_1$ but cannot be constructed from any 
lower dimensional smooth variety this way.
\end{remark}

\section{Geometry of the Fano varieties}
\label{sec:geomfanos}

We take a closer look at the Fano varieties~$X$
listed in Theorem~\ref{thm:main2} and describe 
their possible divisorial contractions in detail,
i.e., the morphisms $X \to Y$ arising from 
semiample divisors which either birationally 
contract a prime divisor or are Mori fiber spaces.
The approach is via a suitable
ambient toric variety. 
The following Remark can be found, at least partially, 
for example in~\cite[Section~7.3]{CLS}.

\begin{remark}
\label{rem:torgeo}
Let $Z$ be a smooth projective toric variety 
of Picard number~2, given by 
weight vectors $w_1 := (1,0)$ and $w_i := (b_i,1)$ with
$0=b_n < b_{n-1} < \ldots < b_2$,
and multiplicities $\mu_i := \mu(w_i)\ge1$,
where $\mu_1\ge2$ and $\mu_2+\ldots+\mu_n\ge2$
as in Remark~\ref{rem:toric-not}.
Then the toric variety $Z$ is a projectivized 
split vector bundle of rank~$r$ over a projective 
space $Y_Z := \PP_s$, where $s := \mu_1-1$ 
and $r := \mu_2 + \ldots + \mu_n-1$.
More precisely, we have
$$ 
Z
\ \cong \ 
\PP
\left( 
\bigoplus_{i=1}^{\mu_n} \mathcal{O}_{\PP_s} 
\oplus
\bigoplus_{i=1}^{\mu_{n-1}} \mathcal{O}_{\PP_s}(b_{n-1}) 
\oplus
\ldots
\oplus
\bigoplus_{i=1}^{\mu_{2}} \mathcal{O}_{\PP_s}(b_{2}) 
\right).
$$
The bundle projection $Z \to Y_Z$ is the 
divisorial contraction associated to the divisor class 
$w_1 \in \ZZ^2 = \Cl(Z)$.
If $n=2$ holds, then we have 
$Z\cong\PP_s \times \PP_{r}$.
If $n=3$ and $\mu_3 = 1$ hold, then the class 
$w_3  \in \ZZ^2 = \Cl(Z)$ gives rise to  
a birational divisorial contraction 
onto a weighted projective space:
$$ 
Z 
\ \to \ 
Z' 
\ := \ 
\PP(
\underbrace{1,\ldots,1}_{\mu_1},
\underbrace{b_2,\ldots,b_2}_{\mu_2}
).
$$
The exceptional divisor $E_Z \subseteq Z$ 
is isomorphic to $\PP_s \times \PP_{\mu_2 -1}$
and the center $C(Z') \subseteq Z'$ of 
the contraction is isomorphic to $\PP_{\mu_2 -1}$.
In particular, for $\mu_2=1$, we have
$E_Z \cong \PP_s$ and $C(Z')$ is a point.
\end{remark}

From the explicit description of the Cox ring 
of our Fano variety $X$, we obtain via 
Construction~\ref{constr:RAPu} a closed 
embedding $X \to Z$ into a toric variety~$Z$.
As a byproduct of our classification, it turns 
out that, whenever~$X$ admits a divisorial contraction, 
then $X$ 
inherits all its divisorial contractions from~$Z$.
Remark~\ref{rem:torgeo} together with the 
explicit equations for $X$ in $Z$ will then allow us to study 
the situation in detail.
We now present the results. 
The cases are numbered according to the 
table of Theorem~\ref{thm:main2}.
Moreover, we denote by $Q_3\subseteq\PP_4$ 
and $Q_4\subseteq\PP_5$
the three and four-dimensional smooth 
projective quadrics
and we write $\PP(a_1^{\mu_1},\ldots,a_r^{\mu_r})$ 
for the weighted projective space,
where the superscript $\mu_i$ indicates that the weight 
$a_i$ occurs $\mu_i$ times.

\begin{vartype}{1} \label{rem::no1} 
The variety $X$ is of dimension four and admits 
two divisorial contractions,
$Q_4 \leftarrow  X \to \PP_1$.
The morphism $X \to Q_4$ is birational with exceptional 
divisor isomorphic to $\PP_1\times\PP_1\times\PP_1$
and center isomorphic to $\PP_1\times\PP_1$.
The morphism $X \to \PP_1$ is a Mori fiber space with 
general fiber isomorphic to $Q_3$ and singular fibers 
over $[0,1]$ and $[1,0]$ each isomorphic to 
the singular quadric $V(T_2T_3+T_4T_5) \subseteq \PP_4$.
\end{vartype}

\begin{vartype}{2} \label{rem::no2} 
The variety $X$ is of dimension four and admits 
two divisorial contractions,
$Q_4 \leftarrow  X \to \PP_3$.
The morphism $X \to Q_4$ is birational with exceptional 
divisor isomorphic to a hypersurface of bidegree $(1,1)$ in 
$\PP_1\times\PP_3$ and center isomorphic to~$\PP_1$.
The morphism $X \to \PP_3$ is a Mori fiber space 
with fibers isomorphic to $\PP_1$.
\end{vartype}

\begin{vartype}{3} \label{rem::no3} 
The variety $X$ is of dimension three and occurs 
as No.~2.29 in the Mori-Mukai classification~\cite{MoMu}.
Moreover, $X$ admits two divisorial contractions,
$Q_3 \leftarrow  X \to \PP_1$.
The morphism $X \to Q_3$ is birational with exceptional 
divisor isomorphic to  $\PP_1\times \PP_1$ 
and center isomorphic to $\PP_1$.
The morphism $X \to \PP_1$ is a Mori fiber space with 
general fiber isomorphic to $\PP_1 \times \PP_1$ and singular fibers 
over $[0,1]$ and $[1,0]$ each isomorphic to 
$V(T_1T_2+T_3^2) \subseteq \PP_3$.
\end{vartype}

\begin{vartype}{4A} \label{rem::no4A} 
\emph{Case~1:} we have $c= -1$.
Then $X$ admits  two divisorial 
contractions $Y \leftarrow  X \to \PP_2$, where 
$Y := V(T_1T_2+T_3T_4+T_5T_6) \subseteq \PP_{m+4}$
is a terminal factorial Fano variety which 
is smooth if and only if $m=1$ holds.
The morphism $X \to Y$ is birational with exceptional 
divisor isomorphic to  a  hypersurface of bidegree $(1,1)$ in 
$\PP_2\times\PP_{m+1}$ and center isomorphic to $\PP_{m+1}$.
The morphism $X \to \PP_2$ is a Mori fiber space with 
fibers isomorphic to $\PP_{m+1}$.

\medskip
\noindent
\emph{Case~2:} we have $c = 0$.
Then $X$ is a hypersurface of bidegree 
$(1,1)$ in $\PP_2\times\PP_{m+2}$.
Moreover, $X$ admits two Mori fiber spaces
$\PP_{m+2} \leftarrow  X \to \PP_2$.
The Mori fiber space $X \to \PP_2$ has fibers isomorphic 
to $\PP_{m+1}$, whereas the Mori fiber space $X \to \PP_{m+1}$ has 
general fiber isomorphic to $\PP_{1}$ and special fibers 
over $V(T_1,T_2,T_3) \subseteq \PP_{m+2}$ isomorphic
to $\PP_2$.
For $m=0$, we have $\dim(X)=3$ and $X$ is
the variety No.~2.32 in~\cite{MoMu}.
\end{vartype}

\begin{vartype}{4B} \label{rem::no4B} 
The variety $X$ admits  two divisorial 
contractions $Y \leftarrow  X \to \PP_2$, where 
$Y := V(T_1^2+T_2T_3+T_4T_5) \subseteq \PP_{m+4}$
is a terminal factorial Fano variety.
The variety~$Y$ is smooth
if and only if $m=0$ holds and in this case $X$ 
occurs as No.~2.31 in~\cite{MoMu}.
The morphism $X \to Y$ is birational with exceptional 
divisor isomorphic to  a  hypersurface of bidegree $(1,1)$ in 
$\PP_2\times\PP_{m+1}$
and center isomorphic to $\PP_{m+1}$.
The morphism $X \to \PP_2$ is a Mori fiber space with 
fibers isomorphic to $\PP_{m+1}$.
\end{vartype}

\begin{vartype}{4C} \label{rem::no4C} 
The variety $X$ is a hypersurface of bidegree 
$(2,1)$ in $\PP_2\times\PP_{m+2}$;
for $m=0$ we have $\dim(X)=3$ and $X$ is
No.~2.24 in~\cite{MoMu}.
Moreover, $X$ admits two Mori fiber spaces
$\PP_{m+2} \leftarrow  X \to \PP_2$.
The morphism $X \to \PP_2$ has fibers
isomorphic to $\PP_{m+1}$.
To describe the fibers of 
$\varphi \colon X \to \PP_{m+2}$,
set $Y_{i} := V_{\PP_{m+2}}(T_i)$,
$Y_{ij} := V_{\PP_{m+2}}(T_i, T_j)$ and  
$Y_{123} := V_{\PP_{m+2}}(T_1, T_2, T_3)$.
Then we have
\begin{equation*}
\varphi^{-1}(z) \ \cong \
\begin{cases}
\PP_2 & \text{if } z\in Y_{123}, \\
\PP_1 & \text{if } z\in (Y_{12} \cup Y_{13}  \cup Y_{23}) \setminus Y_{123}, \\
V_{\PP_2}(T_1T_2) & \text{if } z\in (Y_{1} \cup Y_{2}  \cup Y_{3}) \setminus (Y_{12} \cup Y_{13}  \cup Y_{23}), \\
\PP_1  & \text{otherwise}.
\end{cases}
\end{equation*}
\end{vartype}

\begin{vartype}{5} \label{rem::no5} 
The variety $X$ admits a Mori fiber space
$\varphi \colon X \to \PP_{m+1}$, whose general fiber is 
isomorphic to $\PP_1\times\PP_1$. More precisely,
with $Y_1 := V_{\PP_{m+1}}(T_1)$
and $Y_2 := V_{\PP_{m+1}}(T_2)$, we have
\begin{equation*}
\varphi^{-1}(z) \ \cong \
\begin{cases}
V_{\PP_3}(T_1T_2)  & \text{if } z\in Y_1\cap Y_2, \\
V_{\PP_3}(T_1T_2+T_3^2)  & \text{if }
z\in Y_1\setminus Y_2 \text{ or } z\in Y_2\setminus Y_1, \\
\PP_1\times\PP_1 & \text{otherwise}.
\end{cases}
\end{equation*}
\end{vartype}

\begin{vartype}{6} \label{rem::no6} 
The variety $X$ admits a Mori fiber space
$X \to \PP_{m}$, with 
general fiber isomorphic to $Q_3$ and singular fibers 
over $V(T_1) \subseteq \PP_{m}$ each isomorphic
to $V(T_1T_2+T_3T_4) \subseteq \PP_4$. 
\end{vartype}

\begin{vartype}{7} \label{rem::no7} 
The variety $X$ admits a birational divisorial contraction $X \to \PP_{m+3}$
with exceptional divisor isomorphic to
the projectivized split bundle
$$
\PP \ \biggl( \ \bigoplus_{i=1}^{m} \mathcal{O}_{\PP_1\times\PP_1}
\oplus \mathcal{O}_{\PP_1\times\PP_1}(1,1) \ \biggr)
$$
and center isomorphic to $\PP_1\times\PP_1$.
Moreover, if $m=1$ holds, $X$ admits a 
further birational divisorial 
contraction $X \to Q_4$ with exceptional divisor isomorphic to 
$\PP_3$  and center a point.
\end{vartype}

\begin{vartype}{8} \label{rem::no8}
Here we have $X = \PP(\mathcal{O}_{Q_4} 
\oplus \mathcal{O}_{Q_4} (a_2) \ldots
\oplus \mathcal{O}_{Q_4} (a_m) )$.
Thus, there is a Mori fiber space $X \to Q_4$ 
with fibers isomorphic to $\PP_{m-1}$.
If $a_2= \ldots = a_m > 0$ holds, then $X$ admits 
in addition a birational divisorial contraction $X \to Y$, where 
$Y := V(T_1T_2+T_3T_4+T_5T_6) \subseteq \PP(1^6, a_2^{m-1})$.
The exceptional divisor is isomorphic to $Q_4 \times \PP_{m-2}$ 
and the center to $\PP_{m-2}$.
\end{vartype}

\begin{vartype}{9} \label{rem::no9} 
The variety $X$ is a bundle over $\PP_{m-1}$
with fibers isomorphic to $Q_4$.
In particular, if $a_i=0$ holds for all $2\le i \le 6$, then
$X\cong Q_4\times\PP_{m-1}$.
\end{vartype}

\begin{vartype}{10} \label{rem::no10} 
The variety $X$ admits a birational divisorial 
contraction $X \to \PP_{m+2}$ with exceptional 
divisor isomorphic to the projectivized split bundle
$$
\PP \ \biggl( \ \bigoplus_{i=1}^{m} \mathcal{O}_{\PP_1}
\oplus \mathcal{O}_{\PP_1}(1) \ \biggr)
$$
and center isomorphic to $\PP_1$.
For $m=1$, we have $\dim(X)=3$ and 
$X$ is No.~2.30 from~\cite{MoMu};
in this case it admits a further birational divisorial 
contraction $X \to Q_3$ with exceptional divisor isomorphic 
to $\PP_2$ and center a point.
\end{vartype}

\begin{vartype}{11} \label{rem::no11} 
Here $X = \PP(\mathcal{O}_{Q_3} 
\oplus \mathcal{O}_{Q_3} (a_2) \ldots
\oplus \mathcal{O}_{Q_3} (a_m) ) $ holds.
Thus, there is a Mori fiber space $X \to Q_3$ 
with fibers isomorphic to $\PP_{m-1}$.
If $a_2= \ldots = a_m > 0$ holds, then $X$ admits 
a birational divisorial contraction $X \to Y$, where 
the variety~$Y$ equals 
$V(T_1T_2+T_3T_4+T_5^2) \subseteq \PP(1^5, a_2^{m-1})$.
The exceptional divisor is isomorphic to $Q_3 \times \PP_{m-2}$ 
and the center  to $\PP_{m-2}$.
\end{vartype}

\begin{vartype}{12} \label{rem::no12} 
The variety $X$ is a bundle over $\PP_{m-1}$
with fibers isomorphic to $Q_3$.
In particular, if $a=b=c=0$ holds, then
$X\cong Q_3\times\PP_{m-1}$.
\end{vartype}

\begin{vartype}{13} \label{rem::no13}
This case presents a one-parameter family
of varieties $X_\lambda$,
with parameter $\lambda\in\KK^*\!\setminus\!\{1\}$.
They are generally non-isomorphic to each other,
except for the pairs $X_\lambda \cong X_{\lambda^{-1}}$
for all $\lambda$.
The variety $X_\lambda$ is the intersection of two
hypersurfaces
$$
D_1
\ = \ 
V(T_1S_1+T_2S_2+T_3S_3),
\qquad
D_2
\ = \ 
V(\lambda T_2S_2+T_3S_3+T_4S_4),
$$ 
both of bidegree (1,1) in $\PP_3\times\PP_3$,
where the $T_i$ are the coordinates of the first $\PP_3$
and the $S_j$ those of the second. 
Note that each $D_i$ has an isolated singularity,
which is not contained in the other hypersurface.
Both $D_1,D_2$ are terminal and factorial.
Moreover, $X$ admits two Mori fiber spaces 
$\PP_3 \leftarrow  X \to \PP_3$, both with 
typical fiber $\PP_1$ and having four 
special fibers, all isomorphic to $\PP_2$ 
and lying over the points 
$[1,0,0,0]$, $ [0,1,0,0]$, $[0,0,1,0]$
and $[0,0,0,1]$.
\end{vartype}

\begin{remark}
In contrast to the toric case,
a smooth projective variety of 
Picard number~$2$ with torus action
of complexity one need not admit a 
non-trivial Mori fiber space.
For example, in Theorem~\ref{thm:main2},
this happens in precisely two cases,
namely No.~7 and No.~10, both with $m=1$.
\end{remark}

\begin{remark}
In the list of Theorem~\ref{thm:main2}
there are several examples, where the 
effective cone coincides with the cone 
of movable divisor classes:
No.~4A with $c=0$,  No.~4C,  No.~5 with $a=0$,  
No.~6 with $a=0$, No.~8 with $a_2=0$,  No.~9 with $a_3=0$,
No.~11 with $a_2=0$,  No.~12 with $a=0$ and  No.~13.
Thus, these varieties admit no birational 
divisorial contraction.
\end{remark}

\begin{remark}
In Theorem~\ref{thm:main1} it is possible 
that non-isomorphic varieties share the same 
Cox ring and thus differ from each other 
by a small quasimodification, i.e. only by the choice
of the ample class.
This happens exactly in the following cases:
\begin{enumerate}
\item
No.~4 with $\l_2=l_4=2$, $l_6=1$, $a=0$,
$b=1$, $c_i=0$ for all $i=1,\ldots,m$ 
has the same Cox ring as 
No.~5 with $a=0$.
Note that for $m=0$ both varieties are truly almost Fano,
whereas for $m \ge 1$ No.~5 is Fano.
\item
For $m\ge1$, No.~4 with $\l_2=2$, $l_4=l_6=1$, $a=b=1$,
$c_i=0$ for all $i=1,\ldots,m$
has the same Cox ring as 
No.~6 with $a=c=0$ and $b=1$.
Note that for $m=1$ both varieties are truly almost Fano,
whereas for $m\ge2$ No.~6 is Fano.
\item
For $m\ge2$, No.~7
has the same Cox ring as 
No.~9 with $a_2=2$ and $a_3=\ldots=a_6=1$.
Note that for $m=2,3$ No.~7 is Fano,
for $m=4$ both varieties are truly almost Fano,
whereas for $m\ge5$ No.~9 is Fano.
\item
For $m\ge2$, No.~10 
has the same Cox ring as No.~12 with $a=b=c=1$.
Note that for $m=2$ No.~10 is Fano,
for $m=3$ both varieties are truly almost Fano,
whereas for $m\ge4$ No.~12 is Fano.
\end{enumerate}
\end{remark}


\end{document}